\documentclass[11pt]{article}

\usepackage{latexsym}
\usepackage{bbm}
\usepackage{amsmath}
\usepackage{amsfonts}
\usepackage{amssymb}
\usepackage{multirow}
\usepackage{latexsym}
\usepackage[dvips]{graphicx}
\usepackage{overpic}
\usepackage{graphicx}
\usepackage{relsize}
\usepackage{leftidx}
\newtheorem{theorem}{\bf Theorem}[section]
\newtheorem{lemma}[theorem]{\bf Lemma}
\newtheorem{prop}[theorem]{\bf Proposition}
\newtheorem{coro}[theorem]{\bf Corollary}

\newtheorem{defn}{\bf Definition}[section]

\newtheorem{remark}{{\bf Remark}}[section]
\newenvironment{proof}{\noindent{\em Proof.}}{\quad \hfill$\Box$\vspace{2ex}}

\def \bR {\Bbb R}

\def \and {\, \mbox{\rm and}\, }

\def \l {\left}
\def \r {\right}
\makeatletter

\newcommand{\Rmnum}[1]{\expandafter\@slowromancap\romannumeral #1@}
\makeatother

\begin{document}
\title{\bf On the Multilinear Fractional Integral Operators with Correlation Kernels}
\author{Zuoshunhua Shi \thanks{1.~School of Mathematics and Statistics,~Central South University, Changsha Hunan 410083, P.R. China. \qquad
2.~School of Mathematical Science, Graduate University
of the Chinese Academy of Sciences, Beijing 100049, P.R. China. E-mail address:
{\it shizsh@163.com}.},\quad Di Wu\thanks{School of Mathematical
Science, Graduate University of the Chinese Academy of Sciences,
Beijing 100190, P. R. China. E-mail address: {\it
wudi08@mails.ucas.ac.cn}.} \quad and Dunyan Yan\thanks{School of
Mathematical Science, Graduate University of the Chinese Academy of
Sciences, Beijing 100190, P. R. China. E-mail address: {\it
ydunyan@ucas.ac.cn}.} }
\date{}
\maketitle{}
\begin{abstract}
In this paper, we study a class of multilinear fractional integral
operators associated with correlation kernels $\prod_{1\leq i<j \leq
k}|x_i-x_j|^{-\alpha_{ij}}$. We obtain the necessary and sufficient condition under which these operators are bounded from $L^{p_1}\times \cdots \times L^{p_k}$ into $L^q$. As a consequence, we also get the endpoint estimates from $L^{p_1}\times \cdots \times L^{p_k}$ to $BMO$.
\end{abstract}
\section{Introduction}
\quad Fractional integral operators arise frequently in various
subjects such as Fourier analysis and partial differential
equations. The Riesz potentials are classical fractional integral
operators which were generalized to multilinear variants by many
authors; see \cite{christ}, \cite{beckner}, \cite{grafakos92},
\cite{kenigstein}, \cite{gra-kal}, \cite{valdimarsson}, \cite{beckner2013}, \cite{wudi} and \cite{beckner2015}.
In this paper, we mainly study mapping properties of the multilinear fractional integral operators with correlation kernels of the form $\prod|x_i-x_j|^{-\alpha_{ij}}$. These operators can be written as
\begin{equation}\label{section1. multilinear operators}
T(f_1,\cdots,f_k)(x_{k+1})
=\int_{\mathbb{R}^{nk}}\frac{\prod\limits_{i=1}^{k}f_{i}(x_{i})}
{\prod\limits_{1\leq i<j\leq k+1}|x_i-x_j|^{\alpha_{ij}}}
dx_1dx_2\cdots dx_k
\end{equation}
for $f_{i}\in C_{0}^\infty(\mathbb{R}^n)$ and $\alpha_{ij}\geq 0$.
It is clear that $T$ reduces to Riesz potentials when $k=1$. It is natural to assume that the kernel of $T$ is a Schwartz kernel such that $T$ maps $(C_0^{\infty}(\mathbb{R}^n))^k$ into $\mathcal{D}'(\mathbb{R}^n)$ continuously. This requires some restrictions on the exponents $\alpha_{ij}$. Another issue is to
determine the necessary and sufficient condition under which $T$
is bounded from $L^{p_1}\times \cdots\times L^{p_k}$
into $L^q$. More precisely, we shall establish the following
inequality
\begin{equation}\label{L^p boundedness of multilinear
fractional operators} \|T(f_1,f_2, \cdots,f_{k})\|_{L^q} \leq C
\prod\limits_{i=1}^{k}\|f_{i}\|_{L^{p_{i}}}
\end{equation}
with the constant $C$ independent of $f_i\in L^{p_i}$. The Hardy-Littlewood-Sobolev inequality is a special case of this inequality for $k=1$. For $k\geq 2$, it is more convenient to consider the following multilinear functional:
\begin{equation}\label{multilinear functional}
\Lambda(f_1,f_2,\cdots,f_{k+1})=\int_{\mathbb{R}^{n(k+1)}}
\frac{\prod\limits_{i=1}^{k+1}f_{i}(x_{i})}
{\prod\limits_{1 \leq i<j\leq k+1}|x_i-x_j|^{\alpha_{ij}}}
dx_1dx_2\cdots dx_{k+1}.
\end{equation}
Then the boundedness of $T$ is equivalent to
\begin{equation}\label{inequality of multilinear fractional functionals}
|\Lambda(f_1,f_2,\cdots,f_{k+1})| \leq C
\prod\limits_{i=1}^{k+1}\|f_{i}\|_{L^{p_{i}}}.
\end{equation}

The problems discussed above have close relation with several topics in Fourier analysis. In \cite{christ}, Christ applied a special case of (\ref{inequality of multilinear fractional functionals}) to establish the endpoint estimates of the restriction of the Fourier transform to curves in higher dimensions. Beckner \cite{beckner} obtained a conformally invariant inequality of the form
(\ref{inequality of multilinear fractional functionals}) which generalizes the result of Lieb \cite{Lieb}. Morpurgo obtained sharp inequalities for trace
functionals of pseudo-differential operators on the sphere $S^n$ in \cite{morpurgo}, where multilinear fractional integrals appear explicitly in the calculation of zeta functions of these operators. The sharp inequalities obtained in
\cite{morpurgo} also rely on the strict rearrangement of a class
of functionals with kernels $\prod K_{ij}(|x_i-x_j|)$. In this paper, we shall characterize $\alpha_{ij}$ and $p_i$ such that $\Lambda$ is bounded on $L^{p_1}\times\cdots\times L^{p_{k+1}}$. In this direction, the second author obtained partial results in his dissertation \cite{wudi}. One of our main results can be stated as follows.

\begin{theorem}\label{main theorem 1}
Let $\Lambda $ be the multilinear functional defined by
{\rm(\ref{multilinear functional})} with all $\alpha_{ij}\geq 0$.
Assume $1<p_i<\infty$ for $1\leq i\leq k+1$. Then there exists a constant $C$ such that {\rm(\ref{inequality of multilinear fractional functionals})} is
true if and only if the following three conditions hold simultaneously.
\begin{flalign}
&(i)\quad~ \sum_{i=1}^{k+1} \frac{1}{p_i}+
\sum\limits_{1\leq i<j\leq k+1}\frac{a_{ij}}{n}=k+1;&\nonumber\\
&(ii) \quad \sum\limits_{I}\frac{a_{ij}}{n}<|I|-1
~~\textrm{ for $I\subseteq \{1,2,\cdots,k+1\}$ with $|I|\geq 2$};&\nonumber\\
&(iii)\quad \textrm{For all nonempty proper
$I\subseteq\{1,2,\cdots,k+1\}$, one of the following two
statements is
true:}&\nonumber \\
&\quad~(a) \quad \sum\limits_{I}\frac{1}{p_i}
+\sum\limits_{I}\frac{\alpha_{ij}}{n}<|I|; &\nonumber \\
&\quad~(b) \quad \sum\limits_{I}\frac{1}{p_i}
+\sum\limits_{I}\frac{\alpha_{ij}}{n}=|I|, \quad
\sum\limits_{I^c}\frac{1}{p_i}\geq 1 \quad and \quad
\sum\limits_{J}\l(\frac{1}{p_i} +\sum\limits_{u\in
I}\frac{\alpha_{iu}}{n}\r)
+\sum\limits_{J}\frac{\alpha_{ij}}{n}\leq |J|&\nonumber\\
&\quad \quad \textrm{for all subsets $J$ of $I^c$}.&\nonumber
\end{flalign}
Here we assume $\alpha_{ij}=\alpha_{ji}$ and
$\alpha_{ii}=0$ for $1\leq i,\;j \leq k+1$. The cardinality of $I$
is denoted by $|I|$ and $I^c$ is the complement of $I$. The above
summations are defined by
\begin{equation*}
\sum\limits_{I}\frac{\alpha_{ij}}{n}=\sum\limits_{i,j\in I;
i<j}\frac{\alpha_{ij}}{n}\quad and \quad
\sum\limits_{I}\frac{1}{p_i}=\sum\limits_{i\in I}\frac{1}{p_i}
\end{equation*}
for all subsets $I$ of $\{1,2,\cdots,k+1\}$.
\end{theorem}

Some remarks will help clarify the necessity of conditions in the
theorem. The system of inequalities $(ii)$ ensures that $T(f_1,\cdots,f_k)$ is locally integrable for bounded and compactly supported functions $f_i$. Similarly, we shall also prove that $(ii)$ is the necessary and sufficient condition ensuring finiteness of the sphere Selberg integral
$\int_{(S^n)^k}\prod|\xi_i-\xi_j|^{-\alpha_{ij}}d\sigma(\xi_1)\cdots
d\sigma(\xi_k)$. This integral appears explicitly in the formula of
the sharp constant of $\Lambda$; see \cite{beckner} and \cite{grafakos1}. Our theorem generalizes some earlier results in Christ \cite{christ} and Grafakos-Kalton\cite{gra-kal}, and also extends the Stein-Weiss potentials in \cite{steinweisspaper} to the multilinear setting.

It should be pointed out that $\Lambda$ in (\ref{multilinear functional}) is related to the well-known Brascamp-Lieb inequality. For $1\leq i \leq m$, let $H$ and $H_i$ be real Hilbert spaces and $B_i:H\rightarrow H_i$ surjective linear transformations. Given an $m-$tuple $(p_i)$ in $[1,\infty]^m$, the Brascamp-Lieb inequality takes the form:
\begin{equation}\label{Brascamp-Lieb ineq sec1}
\Psi(f_1,\cdots,f_m):=\int_{H}\prod_{i=1}^mf_i(B_ix)dx\leq C\prod_{i}\|f_i\|_{L^{p_i}(H_i)},
\end{equation}
where we equip Lebesgue measures on $H$ and $H_i$, and take nonnegative measurable functions $f_i$. This inequality unifies some classical inequalities such as the H\"{o}lder inequality, the Young inequality and the Loomis-Whitney inequality; see Brascamp-Lieb \cite{Brascamp76}, Lieb \cite{Lieb90}, Barthe \cite{barthe} and  Bennett-Carbery-Christ-Tao \cite{bennet1, bennet2}. Lieb \cite{Lieb90} proved that the best constant $C$ in (\ref{Brascamp-Lieb ineq sec1}) is saturated by Gaussian functions. In \cite{bennet1, bennet2}, Bennett-Carbery-Christ-Tao characterized $B_i$ and $p_i$ for which (\ref{Brascamp-Lieb ineq sec1}) holds. A proof of Theorem \ref{main theorem 1} is based on the Brascamp-Lieb inequality for Lorentz spaces which was considered by Christ in \cite{perry}.

Now we review some basic properties of Riesz potentials. For $0<\alpha<n$, the Riesz potentials $I_\alpha$ are defined by
\begin{equation*}
I_\alpha(f)(x)=\frac{\Gamma({n}/{2}-{\alpha}/{2})}{\pi^{n/2}2^\alpha\Gamma(\alpha/2)}
\int_{\mathbb{R}^n}\frac{f(y)}{|x-y|^{n-\alpha}}dy
\end{equation*}
for $f\in C_{0}^\infty$. Then $\|I_\alpha(f)\|_{q}\leq C\|f\|_{p}$ for $1<p, q<\infty$ satisfying $\frac{1}{q}=\frac{1}{p}-\frac{\alpha}{n}$; see Stein \cite{stein}. For the endpoint $p=1$, Stein-Weiss \cite{stein-weissbook} proved that $I_{\alpha}$ is also bounded from $H^1$ to $L^{n/(n-\alpha)}$. By duality, $I_{\alpha}$ has a bounded extension from $L^{n/\alpha}$ to $BMO$. This is also true for $T$ in (\ref{section1. multilinear operators}). Moreover, we shall establish $L^{1}$ estimates which are only  in the multilinear case $k\geq 2$.

\begin{theorem}\label{section1. main theorem 2}
Let $T$ be the multilinear operator as in {\rm(\ref{section1.
multilinear operators})} with all $\alpha_{ij}\geq 0$. Assume that
$1<p_i<\infty$ for $1\leq i \le k$ and $p_{k+1}\in \{1,\infty\}$.
Suppose $\{\alpha_{ij}\}$ and $\{p_i\}$ satisfy the conditions $
(i)$ and $ (ii)$ in Theorem {\rm \ref{main theorem 1}} and $(a)$ of $(iii)$ except $p_{k+1}=1$ and $I=\{k+1\}$, i.e.,
\begin{equation*}
\sum\limits_{I}\frac{1}{p_i}
+\sum\limits_{I}\frac{\alpha_{ij}}{n}<|I|
\end{equation*}
for any nonempty proper subset $I$ of $\{1,2,\cdots,k+1\}$
unless $p_{k+1}=1$ and $I=\{k+1\}$. Then there exists a constant $C$ such that for all $f_i\in C_0^{\infty}$
\begin{eqnarray}
&&\|T(f_1,f_2,\cdots,f_{k})\|_{L^1} \leq
C\prod\limits_{i=1}^{k}\|f_i\|_{p_i}\quad\quad {\textrm{if \quad$p_{k+1}=\infty$}},\nonumber\\
& & \|T(f_1,f_2,\cdots,f_{k})\|_{BMO} \leq
C\prod\limits_{i=1}^{k}\|f_i\|_{p_i}\quad {\textrm{if \quad
$p_{k+1}=1$}}.\nonumber
\end{eqnarray}
Moreover, if in addition $\{p_i\}_{i=1}^k$ satisfies $\sum_{i=1}^k1/p_i\geq 1$ when $p_{k+1}=1$, then the space $BMO$ can be replaced by $L^{\infty}$.
\end{theorem}

Concerning notations, the parameters $\{\alpha_{ij}\}$ are defined to be symmetric. In other words, we assume $\alpha_{ii}=0$ and $\alpha_{ij}=\alpha_{ji}$ for all $1\leq i,j\leq k+1$. For any given subset $J$ of $\{1,2,\cdots,k+1\}$, we use the summation conventions
$\sum_J1/p_i$ and $\sum_J \alpha_{ij}$ to denote $\sum_{i\in
J}1/p_i$ and $\sum_{i< j;i,j\in J}\alpha_{ij}$, respectively. If $J$
consists of a single point, we set $\sum_J\alpha_{ij}=0$. This
convention is also extended to general parameters $\{\gamma_i\}$ and
symmetric $\{\beta_{ij}\}$. The constant $C$ means a positive number
which may vary from place to place. For $A,B\geq 0$, $A\lesssim B$
means $A\leq CB$ for some constant $C>0$. We use $A\wedge
B$ to denote $\min\{A,B\}$ and $|J|$ to denote the cardinality of an index set $J$. For a measurable set $E$ in $\mathbb{R}^n$, $|E|$ is its Lebesgue measure. For two sets $E$ and $F$, $E-F$ means $E\cap F^c$ where $F^c$ is the complement of $F$. For brevity, we use $S=\{1,2,\cdots,k\}$ throughout the paper.

The present paper is organized as follows. The section 2 contains
some previously known results which will be used in
subsequent sections. The necessity part of Theorem \ref{main theorem 1} will be proved in $\S$3. We shall prove the sufficiency of Theorem \ref{main theorem 1} in $\S$4. Local integrability of the correlation kernel $\prod{|x_i-x_j|^{-\alpha_{ij}}}$ will be discussed in $\S$5. In $\S$6, we shall give the proof of Theorem \ref{section1. main theorem 2}. In the appendix, the finiteness of a class of Selberg integrals on the sphere will be proved by invoking the methods in $\S$5.

\section{Preliminaries}  
In this section, we shall present some results to be used in subsequent sections.

For $k+1$ points $x_1$, $x_2$, $\cdots$, $x_{k+1}$ in $\bR^k$, we
say that these points are affinely independent if they do not lie in a hyperplane in $\mathbb{R}^k$ simultaneously.

\begin{theorem}\label{section2. Multilinear interpolation}
Assume that $T$ is a $k-linear$ operator which is bounded from
$L^{p_{1j},1}\times \cdots \times L^{p_{kj},1}$ into
$L^{q_j,\infty}$ for $0<p_{ij}\leq \infty$ and $0<q_j\leq \infty$
with $1\leq i \leq k$ and $1\leq j \leq k+1$. Assume also that these $k+1$ points $(1/p_{1j},\cdots,1/p_{kj})$ are affinely independent
in $\mathbb{R}^{k}$. If there are $k+1$ real numbers $\lambda_i$
with positive $\lambda_1,\cdots,\lambda_{k}$ such that
\begin{equation*}
\frac1{q_j}
=\sum\limits_{i=1}^{k}\frac{\lambda_i}{p_{ij}}+\lambda_{k+1},
\quad 1\leq j\leq k+1,
\end{equation*}
then $T$ is bounded from $L^{p_1,t_1}\times \cdots \times L^{p_{k},t_{k}}$
into $L^{q,t}$ with $(1/p_1,\cdots,1/p_{k},1/q)$ lying in the open
convex hull of the $k+1$ points $(1/p_{1j},\cdots,1/p_{kj}, 1/q_j)$ in $\mathbb{R}^{k+1}$ and $0<t_i,t\leq \infty$ satisfying
\begin{equation*}
\sum\limits_{i=1}^{k}\frac1{t_i}\geq \frac1{t}.
\end{equation*}
\end{theorem}
This theorem was previously known; see Janson \cite{janson}. We also refer the reader to a similar variant called the multilinear Marcinkiewicz interpolation in
Grafakos-Kalton \cite{gra-kal}.

The $L^1$ estimate in Theorem \ref{section1. main theorem 2} implies
that $\Lambda(f_1,\cdots,f_{k+1})$ is bounded by a constant multiple
of $\prod_{i=1}^{k+1}\|f_i\|_{p_i}$ with $p_{k+1}=\infty$. Let
$f_{k+1}\equiv 1$. We see that the integral of (\ref{section1.
multilinear operators}) with respect to $x_{k+1}$ is a
generalization of the beta integral with $k=2$. An
induction argument requires that upper bounds of the integral are of the form $\prod_S |x_i-x_j|^{-\beta_{ij}}$ with suitable parameters $\{\beta_{ij}\}$. In other words, we need the following type estimates:
\begin{equation}\label{section2. esti of int with respect to xk+1}
\int_{\mathbb{R}^n}\prod\limits_{i=1}^k
\big|x_i-x_{k+1}\big|^{-\alpha_{i,k+1}}dx_{k+1}\leq
C\prod\limits_S|x_i-x_j|^{-\beta_{ij}}.
\end{equation}
The following theorem gives us the desired estimates.


\begin{theorem}\label{Estimtates of Generalization of Riesz Potential}
Assume $\alpha_1,~\alpha_2,~\cdots,\alpha_k$ satisfy $0<\alpha_i<n$.
If $\sum_{i=1}^k \alpha_i>n$, then the following estimate
\begin{equation*}
\int_{\mathbb{R}^n}\prod\limits_{i=1}^k
\big|t-x_i\big|^{-\alpha_i}dt \leq C \sum\limits_{u=1}^k
\boldsymbol{L}_u(x_1,x_2,\cdots,x_k)
\end{equation*}
 holds for arbitrary $x_1,x_2,\cdots,x_k$ in $\mathbb{R}^n$, where each $\boldsymbol{L}_u$ is defined by
\begin{equation*}
\boldsymbol{L}_u(x_1,x_2,\cdots,x_k)=
\begin{cases}
d_S^{n-\sum_{S}\alpha_i}
\l(\mathlarger{\mathlarger{\chi}}_{\{\sum_{S-\{u\}}\alpha_i<n\}}
+\chi_{\{\sum
_{S-\{u\}}\alpha_i=n\}}\log \frac{2d_S}{d_{S-\{u\}}}\r)\\
d_S^{-\alpha_u}
\int_{\mathbb{R}^n}\prod_{S-\{u\}}|t-x_i|^{-\alpha_i}dt ,\quad if
~~\sum_{S-\{u\}}\alpha_i>n
\end{cases}
\end{equation*}
with the characteristic function $\chi$ being taken relative to
$\alpha_1,\cdots,\alpha_k$. Here $S=\{1,2,\cdots,k\}$ and
$d_I=\sum_I|x_i-x_j|$ for subsets $I$ of $S=\{1,2,\cdots,k\}$ with
$|I|\geq 2$.
\end{theorem}


\begin{remark}
There are some explicit formulas concerning the integral in the above theorem. These formulas take the form
$$\int_{\mathbb{R}^n}\prod\limits_{i=1}^k
\big|t-x_i\big|^{-\alpha_i}dt=C\prod_{1\leq i<j\leq
k}|x_i-x_j|^{-\gamma_{ij}}$$ with a constant $C$ independent of
$x_i$. When $k=2$, this is just the $n-$dimension version of the
beta integral formula; see Stein {\rm \cite{stein}}. For $k=3$,
Grafakos and Morpurgo {\rm \cite{grafakos1}} proved the equality
with $\gamma_{ij}=\alpha_i+\alpha_j-n$ when
$\alpha_1+\alpha_2+\alpha_3=2n$. However, its generalization to
other cases is impossible. Recently, Wu and Yan have proved that the above equality cannot be true in the remaining cases {\rm (i)}
$\alpha_1+\alpha_2+\alpha_3\neq 2n$ when $k=3$ and {\rm (ii)} $k\geq 4$; see {\rm \cite{wudi}}.
\end{remark}

\begin{lemma}\label{Sec2. estimate of FI on the critical order}
Let $\alpha_1,~\alpha_2,~\cdots,\alpha_k$ be positive numbers
satisfying $\sum_{i=1}^k\alpha_k=n$ with $k\geq 2$. For $k$ points
$x_1,~x_2,~\cdots,x_k$ in the unit ball
$B_1(0)\subseteq\mathbb{R}^n$, it is true that
\begin{equation}\label{sec2 lemma2.3}
\int_{|t|\leq 2}\prod\limits_{i=1}^k |t-x_i|^{-\alpha_i}dt \leq
C\log\frac{C}{d_S},
\end{equation}
where $C$ depends on $\alpha_1,\cdots,\alpha_k$ and the dimension $n$. Moreover, the reverse inequality is also true for another constant $C$ depending on $n$ and $\alpha_1,\cdots,\alpha_k$.
\end{lemma}

\begin{proof}
Without loss of generality, we may assume $|x_1-x_k|=\max_S|x_i-x_j|>0$. Recall that $d_S=\sum_{i<j}|x_i-x_j|$. By the assumption $x_i\in B_1(0)$, we see that $d_S$ is bounded by a constant $C(k)$ depending only on $k$; for example, one may take $C (k)=k(k-1)$. By translation,
$$
\int_{|t|\leq 1}\prod_{i=1}^k |t-(x_i-x_1)|^{-\alpha_i}dt\leq \int_{|t|\leq 2}\prod_{i=1}^k |t-x_i|^{-\alpha_i}dt\leq
\int_{|t|\leq 3}\prod_{i=1}^k |t-(x_i-x_1)|^{-\alpha_i}dt.
$$
Since $\frac{x_k-x_1}{d_S}$ has length not less than $c_k=\frac2{k(k-1)}$, we deduce from the assumption $\sum\alpha_i=n$ that
\begin{eqnarray*}
\int_{|t|\leq 3d_S}\prod_{i=1}^k |t-(x_i-x_1)|^{-\alpha_i}dt
&\leq&
\int_{|t|\leq 3}|t|^{-\alpha_1}\prod_{i=2}^k |t-(x_i-x_1)/d_S|^{-\alpha_i}dt\\
&=&
\l(\int_{|t|\leq c_k/2}+\int_{c_k/2<|t|\leq 3}\r)|t|^{-\alpha_1}\prod_{i=2}^k |t-(x_i-x_1)/d_S|^{-\alpha_i}dt\\
&\leq& C(n,\alpha_1,\cdots,\alpha_k).
\end{eqnarray*}
In the case $d_S\geq1/3$, it follows that (\ref{sec2 lemma2.3}) and its reverse are true. Now assume $d_S< 1/3$. It is easy to see that
\begin{eqnarray*}
\int_{3d_S\leq |t|\leq 3}\prod\limits_{i=1}^k
|t-(x_i-x_1)|^{-\alpha_i}dt &\approx &
\int_{3d_S\leq |t|\leq 3}|t|^{-n}dt\\
&= & C \log\frac{1}{d_S}.
\end{eqnarray*}
Thus (\ref{sec2 lemma2.3}) is also true if $d_S<1/3$. Now we shall prove its reverse form in the case $d_S<1/3$. Notice that
\begin{eqnarray*}
\int_{|t|\leq 1}\prod_{i=1}^k |t-(x_i-x_1)|^{-\alpha_i}dt
&\geq&
\int_{3d_S\leq|t|\leq 1}\prod_{i=1}^k |t-(x_i-x_1)|^{-\alpha_i}dt\\
&\approx &
\int_{3d_S\leq |t|\leq 3}|t|^{-n}dt\\
&\geq & C \log\frac{1}{d_S}.
\end{eqnarray*}
Thus we complete the proof of the lemma.
\end{proof}

We now turn to the proof of Theorem \ref{Estimtates of
Generalization of Riesz Potential}.\\
\begin{proof}
We assume
$L=\max_S|x_i-x_j|=|x_1-x_k|>0$. Then $L\approx d_S$. We shall estimate the integral over $B_{L/2}(x_1)$ and its complement separately. Observe that
\begin{equation*}
\int_{B_{L/2}(x_1)}\prod\limits_{i=1}^k
\Big|t-x_i\Big|^{-\alpha_i}dt \leq
Cd_S^{-\alpha_k}\int_{B_{L/2}(0)}\prod\limits_{i=1}^{k-1}
\Big|t-x_i+x_1\Big|^{-\alpha_i}dt.
\end{equation*}
In the case $\sum_{i=1}^{k-1}\alpha_i\leq n$, it follows from Lemma \ref{Sec2. estimate of FI on the critical order} that the integral of $\prod_{i=1}^{k-1}|t-x_i+x_1|^{-\alpha_i}$ over $B_{L/2}(0)$ is bounded by a constant multiple of $d_S^{n-\sum_{i=1}^{k-1}\alpha_i}
\l({\chi}_{\{\sum_{S-\{k\}}\alpha_i<n\}} +\chi_{\{\sum
_{S-\{k\}}\alpha_i=n\}}\log \frac{2d_S}{d_{S-\{k\}}}\r)$, where $d_{S-\{k\}}=\sum_{1\leq i<j\leq k-1}|x_i-x_j|$. For $\sum_{i=1}^{k-1}\alpha_i> n$, it is clear that
$$\int_{B_{L/2}(0)}\prod\limits_{i=1}^{k-1}
\Big|t-x_i+x_1\Big|^{-\alpha_i}dt \leq
\int_{\mathbb{R}^n}\prod_{S-\{k\}}|t-x_i|^{-\alpha_i}dt.
$$
Now we treat the integral outside the ball $B_{L/2}(x_1)$. It is
easy to see that
\begin{eqnarray*}
\int_{B_{L/2}^c(x_1)}\prod\limits_{i=1}^k
\Big|t-x_i\Big|^{-\alpha_i}dt
&\leq&
C\int_{L/2\leq |t|\leq
2L}\prod\limits_{i=1}^k\Big|t-(x_i-x_1)\Big|^{-\alpha_i}dt\\
&\leq& Cd_S^{-\alpha_1}\int_{|t|\leq
2L}\prod\limits_{i=2}^k\Big|t-(x_i-x_1)\Big|^{-\alpha_i}dt.
\end{eqnarray*}
The integral of $\prod_{i=2}^k|t-(x_i-x_1)|^{-\alpha_i}$ over
$|t|\leq 2L$ can be treated similarly as above. Combining above
estimates, we conclude that the integral in the theorem is bounded by a constant multiple of $\boldsymbol{L_1}(x_1,\cdots,x_k)+\boldsymbol{L_k}(x_1,\cdots,x_k)$. This completes the proof.
\end{proof}

We shall see that the proof of Theorem \ref{main theorem 1} and Theorem \ref{section1. main theorem 2} is closely related to existence of solutions to systems of linear inequalities. A system of linear inequalities in $\mathbb{R}^n$ is given by
\begin{equation}
{\rm{(II.1)}}
\begin{cases}
f_i(x)=(v_i, x)< a_i, \qquad\quad 1\le i \le m\nonumber\\
f_i(x)=(v_i, x)\le a_i, \quad m+1 \le i \le k \nonumber
\end{cases}
\end{equation}
where $v_i\in \mathbb{R}^n$, $a_i \in \mathbb{R}$ and
$(\cdot,\cdot)$ denotes the standard inner product in
$\mathbb{R}^n$. It is worthwhile noting that we may incorporate an
linear equality into a system of linear inequalities. Indeed, we may
write $g(x)=(v,x)=a$ as an equivalent system of two linear
inequalities given by $g(x)\leq a$ and $-g(x)\leq -a$.
\begin{lemma}\label{section2. existence System of LI}
Suppose that the system $f_i(x)=(v_i, x)\le a_i$ for $m+1\leq i \leq
k$ has at least one solution. Then there exists a solution $x\in
\mathbb{R}^n$ to the system {\rm{(II.1)}} if and only if
\begin{equation*}
\sum\limits_{i=1}^k \lambda_i a_i >0
\end{equation*}
for all nonnegative numbers $\lambda_i$ satisfying $\sum_{i=1}^k
\lambda_i f_i =0$ with at least one $\lambda_i>0$ for $1\leq i \leq
m$.
\end{lemma}
This lemma is a special case of the existence theorem of systems of
convex inequalities in $\mathbb{R}^n$. However, it can be proved by
a simple method using the concept of elementary vectors of an
subspace of $\mathbb{R}^n$. We refer the reader to $\S$22 (Page 198) in Rockafellar \cite{rockafellar}; see also \cite{fanky} for its
extension to general vector spaces.

\section{Necessity Part of Theorem \ref{main theorem 1}}
In this section, we shall prove the necessity of conditions $(i)$,
$(ii)$ and $(iii)$ in Theorem \ref{main theorem 1}. Indeed the second author obtained these necessary conditions in his thesis \cite{wudi}. We present here the details of proof for convenience of the reader.

Assume the inequality (\ref{inequality of multilinear
fractional functionals}) is true for some constant $C$ independent of $f_i$. We replace $f_i$ by its dilation
$\delta_\lambda(f_i)(x)=f_i(\lambda x)$ for $\lambda>0$. By a change
of variables, we see that $(i)$ must hold by letting
$\lambda\rightarrow 0$ and $\lambda\rightarrow \infty$.

To show the necessity of (ii), we take all $f_i=\chi_{B_1(0)}$. We
shall replace $\{1,2,\cdots,k+1\}$ by $S=\{1,2,\cdots,k\}$. We claim
that if there were a subset $J\subset S$ with $|J|\geq 2$ such that
$\sum_J \alpha_{ij}\geq (|J|-1)n, $ we would obtain $
\int_{(B_1(0))^{|J|}}\prod\limits_J|x_i-x_j|^{-\alpha_{ij}}dV_J
=\infty, $ where $B_1(0)$ is the unit ball centered at the origin in
$\mathbb{R}^n$ and $dV_J$ is the product Lebesgue measure
$\prod_Jdx_{i}$. Since $\alpha_{ij}$ are nonnegative, the argument is essentially the same for different
$J's$. Assume $J=S$. Then
$\int_{(B_1(0))^{k}}\prod\limits_S|x_i-x_j|^{-\alpha_{ij}}dV_S$ is
equal to
\begin{eqnarray*}
& &\int_{B_1(0)} \left(\int_{(B_{1}(x_1))^{k-1}}
\prod\limits_{i=2}^{k}|x_i|^{-\alpha_{1i}}
\prod\limits_{2\leq i<j \leq k}|x_i-x_j|^{-\alpha_{ij}}dx_2\cdots dx_k\right)dx_1\\
&\geq& C\int_{(B_{1/2}(0))^{k-1}}
\prod\limits_{i=2}^{k}|x_i|^{-\alpha_{1i}}
\prod\limits_{2\leq i<j \leq k}|x_i-x_j|^{-\alpha_{ij}}dx_2\cdots dx_k.
\end{eqnarray*}
Write $X=(x_2,\cdots,x_k)\in \mathbb{R}^{n(k-1)}$. Let
$B_{r}(0_{m})$ be the unit ball centered at the origin in
$\mathbb{R}^m$ with radius $r>0$. It is clear that
$B_{1/2}(0_{n(k-1)})$ is contained in $(B_{1/2}(0_{n}))^{k-1}$. We
may regard the integrand $\prod_{i=2}^{k}|x_i|^{-\alpha_{1i}}
\prod_{2\leq i<j \leq k}|x_i-x_j|^{-\alpha_{ij}}$ as a homogeneous
function of degree $-\sum_S\alpha_{ij}$ in $\mathbb{R}^{n(k-1)}$.
Its integral over a ball centered at the origin in
$\mathbb{R}^{n(k-1)}$ is infinite since its order of homogeneity is
less than or equal to $-(k-1)n$.

It remains to prove the necessity of (iii). We first prove that for any $J\subseteq S$ with $|J|\geq 2$,
\begin{equation}\label{section3. inequality (iii)}
\sum_J\alpha_{ij}+\sum_J\frac{n}{p_i}\leq |J|n.
\end{equation}
Assume the converse holds, i.e., there exists some $J_0\subseteq S$ such that $|J|\geq 2$ and the above inequality is not true for $J_0$. We choose $0<\lambda_i<n/p_i$ for each $i\in J_0$, such that
\begin{equation*}
\sum_{J_0}\alpha_{ij}+\sum_{J_0}\lambda_i=|J_0|n.
\end{equation*}
Let $f_i(y)=\chi_{\{|y|\leq 1\}}|y|^{-\lambda_i}$ for each $i\in J_0$ and $f_i$ be the characteristic function of the unit ball $B_1(0)$ for $i\notin J_0$. It follows that
\begin{equation*}
\Lambda(f_1,\cdots,f_{k+1})\geq C\int_{(B_1(0))^{|J_0|}}
\prod\limits_{J_0}|x_i|^{-\lambda_i}
\prod\limits_{J_0}|x_i-x_j|^{-\alpha_{ij}}dV_{J_0}=\infty
\end{equation*}
since $\sum_{J_0}\alpha_{ij}+\sum_{J_0}\lambda_i=|J_0|n$.

If for some proper subset $J$ of $\{1,2,\cdots,k+1\}$ containing at least two elements such that (\ref{section3. inequality (iii)}) becomes an equality, we claim that $\sum_{J^c}1/p_i\geq 1$. Assume $\sum_{J^c}1/p_i<1$. We can choose $0<\lambda_i<p_i$ such that $\sum_{J^c}1/\lambda_i<1$. Let $f_i(y)=|y|^{-n/p_i}\chi_{\{|y|>2\}}(\log |y|)^{-1/\lambda_i}$ for each $i\in J^c$ and $f_i=\chi_{B_1(0)}$ for $i\in J$. Substituting these functions into the functional $\Lambda$, we have
\begin{eqnarray}\label{3.2}
\Lambda(f_1,\cdots,f_{k+1})
&\geq&
 C\int_{(B_2^c(0))^{|J^c|}}
\prod\limits_{J^c}|f_i(x_i)|
\prod\limits_{J^c}|x_i|^{-\beta_i}
\prod\limits_{J^c}|x_i-x_j|^{-\alpha_{ij}}dV_{J^c}\\
&=&
C\int_{(B_2^c(0))^{|J^c|}}
\prod\limits_{J^c}|x_i|^{-n/p_i}(\log |x_i|)^{-1/\lambda_i}
\prod\limits_{J^c}|x_i|^{-\beta_i}
\prod\limits_{J^c}|x_i-x_j|^{-\alpha_{ij}}dV_{J^c}\nonumber,
\end{eqnarray}
where $\beta_i=\sum_{u\in J}\alpha_{iu} $ for each
$i \in J^c$.

Since the conditions (i), (ii) and (iii) in Theorem \ref{main theorem 1}
are invariant under the permutation group on $k+1$ letters, we may assume
$J^c=\{1,2,\cdots,l\}$ with $1\leq l \leq k$. If $l=1$, it follows
immediately from the fact $n/p_1+\beta_1=n$ that the right side integral
in (\ref{3.2}) is infinite since $1/\lambda_1<1.$ Now we treat the case
$l>1$.
Replacing the region $(B_2^c(0))^l$ by its proper subset $\Omega_l$
consisting of all points $(x_1,\cdots,x_l)$ such that $x_1\in B_2^c(0)$
and  $|x_i|\geq 2|x_{i-1}|$ for $2\leq i \leq l$, we obtain
\begin{eqnarray*}
&&\int_{|x_l|\geq 2|x_{l-1}|}
|x_l|^{-n/p_l}(\log |x_l|)^{-1/\lambda_l}|x_l|^{-\beta_l}
|x_l|^{-\sum_{i=1}^{l-1}\alpha_{il}}dx_l
\geq C|x_{l-1}|^{-\xi_l}(\log |x_{l-1}|)^{-1/\lambda_l},
\end{eqnarray*}
where $\beta_l=\sum_{i=l+1}^{k+1}\alpha_{il}$ and $\xi_l=n/p_l+\sum_{i=1}^{k+1}\alpha_{il}-n$. Substituting this estimate into the integral in (\ref{3.2}), we see that $\Lambda(f_1,\cdots,f_{k+1})$ is not less than a constant multiple of
\begin{equation*}
\int_{\Omega_{l-1}}
\prod\limits_{i=1}^{l-1}|x_i|^{-n/p_i-\boldsymbol{\delta}_{i}^{(l-1)}\xi_l}(\log |x_i|)^{-1/\lambda_i-\boldsymbol{\delta}_{i}^{(l-1)}/\lambda_l}
\prod\limits_{i=1}^{l-1}|x_i|^{-\beta_i}
\prod\limits_{1\leq i<j\leq l-1}
|x_i-x_j|^{-\alpha_{ij}}dx_1\cdots dx_{l-1},
\end{equation*}
where $\Omega_{l-1}$ is the region $x_1\in B_2^c(0)$ and  $|x_i|\geq 2|x_{i-1}|$ for $2\leq i \leq l-1$, $\boldsymbol{\delta}_{i}^{j}$ equals one if $i=j$ and zero otherwise. Repeating the process with $l-1$ steps, we obtain the resulting estimate
\begin{eqnarray*}\label{Resulting estimate 3.3}
\Lambda(f_1,\cdots,f_{k+1})\geq C \int_{|x_1|\geq 2}|x_1|^{-\xi_1}(\log |x_1|)^{-\sum_1^l 1/\lambda_i}dx_1
\end{eqnarray*}
with $$\xi_1=\sum_{i=1}^l\frac{n}{p_i}+\sum_{1\leq i<j \leq l}\alpha_{ij}+\sum_{i=1}^l\sum_{j=l+1}^{k+1}
\alpha_{ij}-(l-1)n=n$$
by the condition (i). Recall that $\sum_{i=1}^{l}1/\lambda_i$ is less than 1. The above integral is infinite. This contradicts the boundedness of $\Lambda$. Hence $\sum_{J^c}1/p_i\geq 1$.

It remains to show that certain additional requirements are
necessary in Theorem \ref{main theorem 1} when $(iii)$ contains
equalities for some proper subsets $I$ of $\{1,2,\cdots,k+1\}$. This means that the Brascamp-Lieb datum corresponding to $\Lambda$ has critical subspaces; see Bennett-Carbery-Christ-Tao \cite{bennet1,bennet2}.
\begin{theorem}\label{sec 3. weighted fracInt Operator}
Assume $1\leq p_i \leq \infty$ and $\alpha_{ij}\geq 0$ for $1\leq i,
j \leq k+1$. Suppose the multilinear functional $\Lambda$ given by
${\rm (\ref{multilinear functional})}$ satisfies
\begin{equation*}
|\Lambda(f_1,f_2,\cdots,f_{k+1})|\leq C\prod\limits_{i=1}^{k+1} \|f_i\|_{p_i}
\end{equation*}
for some $C$ independent of $f_i$. If $J_0$ is a nonempty proper subset
of $\{1,2,\cdots,k+1\}$ satisfying
\begin{equation*}
\sum_{J_0}\frac{1}{p_i}+\sum_{J_0} \frac{\alpha_{ij}}{n}=|J_0|,
\end{equation*}
then we have
\begin{equation}\label{section3. Type I of weighted FI}
\int_{(\mathbb{R}^{n})^{|J_0|}}\prod\limits_{J_0}|f_i(x_i)|
\prod\limits_{J_0}|x_i-x_j|^{-\alpha_{ij}}dV_{J_0}
\leq
C\prod\limits_{J_0}\|f_i\|_{p_i}
\end{equation}
\begin{equation}\label{section3. Type II of weighted FI}
and \quad \int_{(\mathbb{R}^{n})^{|J_0^c|}}\prod\limits_{{J_0}^{c}}|f_i(x_i)|
\prod\limits_{{J_0}^{c}}|x_i-x_j|^{-\alpha_{ij}}
\prod\limits_{{J_0}^{c}}|x_i|^{-\beta_i}dV_{J_0^c}
\leq C\prod\limits_{{J_0}^c}\|f_i\|_{p_i}
\end{equation}
with $\beta_i=\sum_{u\in J_0} \alpha_{iu}$ for each $i\in {J_0}^c$
and both constants $C$ independent of $f_i$. Moreover, it is also
true that
\begin{equation}\label{section3. II inequality (iii)}
\sum\limits_{J}\l(\frac1{p_i}+\frac {\beta_i}{n}\r)
+\sum\limits_{J}\frac{\alpha_{ij}}{n}\leq |J|
\end{equation}
for all nonempty subsets $J$ of $J_0^c$. Conversely, {\rm(\ref{section3. Type I of weighted FI})} and {\rm(\ref{section3. Type II of weighted FI})} imply the boundedness of $\Lambda$.
\end{theorem}

\begin{proof}
For each $\epsilon>0$, let $f_i=\epsilon^{-n/p_i}\chi_{\{|y|<\epsilon/2\}}$ for each $i\in J_0$. By $\sum_{J_0}1/p_i+\sum_{J_0}\alpha_{ij}/n=|J_0|$, we have
\begin{equation*}
\int_{(\mathbb{R}^{n})^{|J_0|}}\prod\limits_{{J_0}}|f_i(x_i)|
\prod\limits_{{J_0}}|x_i-x_j|^{-\alpha_{ij}}dV_{J_0}
\geq C,
\end{equation*}
where $C$ is a constant depending only on $n$ and $|J_0|$ but not on $\epsilon$. For given nonnegative $f_i\in L^{p_i}$ with $i\in J_0^c$, it follows from the boundedness of $\Lambda$ that
\begin{equation*}
\int_{(B_{\epsilon}^c(0))^{|J_0^c|}}\prod\limits_{{J_0}^{c}}|f_i(x_i)|
\prod\limits_{{J_0}^{c}}|x_i-x_j|^{-\alpha_{ij}}
\prod\limits_{{J_0}^{c}}|x_i|^{-\beta_i}dV_{J_0^c}
\leq C\prod\limits_{{J_0}^c}\|f_i\|_{p_i}
\end{equation*}
and this inequality becomes (\ref{section3. Type II of weighted FI}) by letting $\epsilon\rightarrow 0$. The first inequality (\ref{section3. Type I of weighted FI}) can be obtained by a similar argument. Indeed, put $f_i=\epsilon^{-n/p_i}\chi_{\{\epsilon<|y|<2\epsilon\}}$ for each $i\in J_0^c$. For nonnegative functions $f_i$ with $i\in J_0$, we also have
\begin{equation*}
\int_{(\mathbb{R}^n)^{|J_0^c|}}\prod\limits_{{J_0}^{c}}|f_i(x_i)|
\prod\limits_{{J_0}^{c}}|x_i-x_j|^{-\alpha_{ij}}
\left(\prod\limits_{i\in J_0^{c}}
\prod\limits_{j\in J_0}|x_i-x_j|^{-\alpha_{ij}}\right)dV_{J_0^c}
\geq C
\end{equation*}
for $|x_j|<\epsilon/2$ with $j\in J_0$, where the constant $C$ is independent of $\epsilon$ and $x_j\in B_{\epsilon/2}(0)$ with $j\in J_0$. Similarly, we then obtain
\begin{equation*}
\int_{(B_{\epsilon/2}(0))^{|J_0|}}\prod\limits_{{J_0}}|f_i(x_i)|
\prod\limits_{{J_0}}|x_i-x_j|^{-\alpha_{ij}}dV_{J_0}
\leq C\prod\limits_{J_0}\|f_i\|_{p_i}
\end{equation*}
where $f_i\in L^{p_i}$ with $i\in J_0$ and the constant $C$ is independent of $f_i$ and $\epsilon$. By letting $\epsilon\rightarrow \infty$, the desired inequality follows. The inequalities in (\ref{section3. II inequality (iii)}) can be proved similarly as (\ref{section3. inequality (iii)}) by invoking
(\ref{section3. Type II of weighted FI}). We omit the details here.
If {\rm(\ref{section3. Type I of weighted FI})} and {\rm(\ref{section3. Type II of weighted FI})} are true, we first consider the integral in (\ref{multilinear functional}) with respect to $dV_{J_0^c}=\prod_{i\in J_0^c}dx_i$. Though the integrand depends on $x_i$ with $i\in J_0$, it is bounded by a constant multiple of $\prod_{J_0^c}\|f_i\|_{p_i}$. Hence the boundedness of $\Lambda$ follows.
\end{proof}

Combining above results, the proof of the necessity part of Theorem
\ref{main theorem 1} is complete.\\

\section{Sufficiency Part of Theorem \ref{main theorem 1}}
In this section, we shall prove the sufficiency part of Theorem \ref{main theorem 1}. The argument is similar to Christ's proof of the Brascamp-Lieb inequality for Lorentz spaces; see Perry \cite{perry}. Our main tool is the powerful Brascamp-Lieb inequality. For the rank-one case, Barthe \cite{barthe} applied Lieb's theorem \cite{Lieb90} and the Cauchy-Binet formula to obtain the necessary and sufficient condition for which (\ref{Brascamp-Lieb ineq sec1}) holds. Bennett-Carbery-Christ-Tao \cite{bennet1, bennet2} proved that the general Brascamp-Lieb inequality (\ref{Brascamp-Lieb ineq sec1}) is true for some $C<\infty$ if and only if
\begin{equation}\label{scaling condition}
\dim H=\sum_{i=1}^{m}\frac1{p_i}\dim H_i
\end{equation}
and for all subspaces $V$ of $H$
\begin{equation}\label{dimension condition}
\dim V \leq \sum_{i=1}^{m}\frac1{p_i}\dim (B_i V).
\end{equation}
It is clear that (\ref{dimension condition}) consists of finitely many inequalities. For the rank-one case, Barthe \cite{barthe} characterized the extreme points of $\{1/p_i\}$ for which the Brascamp-Lieb inequality holds. Valdimarsson \cite{S.i.valdimarsson} considered the corank-one and certain mixed rank cases, and constructed a procedure to find the full list of dimension inequalities in (\ref{dimension condition}).\\

 Consider the following multilinear functional
\begin{equation}\label{Brascamp-Lieb Psi}
\Psi(\{f_i\}_{i=1}^N;\{g_{ij}\}_{1\leq i<j \leq N}):=\int_{\mathbb{R}^N}\prod_{i=1}^Nf_i(x_i)
\prod_{1\leq i<j \leq N}g_{ij}(x_i-x_j)dx_1dx_2\cdots dx_N.
\end{equation}
Define linear transformations:
\begin{eqnarray*}
& &B_i: ~\mathbb{R}^N \rightarrow \mathbb{R},\qquad B_ix=x_i,\qquad\qquad 1\leq i \leq N,\\
& &B_{ij}: \mathbb{R}^N \rightarrow \mathbb{R},\qquad B_{ij}x=x_i-x_j,
\quad 1\leq i,j \leq N,
\end{eqnarray*}
where $x=(x_1,x_2,\cdots,x_N)\in\mathbb{R}^N$.
For $p_i,p_{ij}\in [1,\infty]$, we want to characterize these exponents such that there exists a constant $C$ satisfying
\begin{equation}\label{Brascamp-Lieb ineq sec4}
|\Psi(\{f_i\};\{g_{ij}\})|\leq
C\prod_{i=1}^{N}\|f_i\|_{L^{p_i}(\mathbb{R})}
\prod_{1\leq i<j\leq N}\|g_{ij}\|_{L^{p_{ij}}(\mathbb{R})}.
\end{equation}
By (\ref{scaling condition}) and (\ref{dimension condition}), this inequality is true if and only if
\begin{equation}\label{scaling condition2}
\sum_{i=1}^N\frac1{p_i}+\sum_{i<j}\frac1{p_{ij}}=N
\end{equation}
and for all subspaces $V\subseteq \mathbb{R}^N$
\begin{equation}\label{dimension condition2}
\dim V\leq \sum_{i=1}^N\frac1{p_i}\dim (B_iV)
+\sum_{i<j}\frac1{p_{ij}}\dim (B_{ij} V).
\end{equation}

For any fixed subspace $V$ of $\mathbb{R}^N$, we define $X$ to be the subspace of linear transformations from $\mathbb{R}^N$ into $\mathbb{R}$ given by
$$X=span \l\{B_i,B_{jk}:V\subseteq \ker(B_i)\bigcap \ker(B_{jk})\r\}.$$
It is clear that $X$ depends on the given subspace $V$. This fact will be used throughout this section. Take the subspace $W\subseteq \mathbb{R}^N$ as follows.
$$W=\l[\bigcap_{B_i\in X}\ker(B_i)\r]\bigcap \l[\bigcap_{B_{jk}\in X}\ker (B_{jk})\r].$$
Then it is easy to see that
$$\sum_{i=1}^N\frac1{p_i}\dim (B_iV)
+\sum_{i<j}\frac1{p_{ij}}\dim (B_{ij} V)
=
\sum_{i=1}^N\frac1{p_i}\dim (B_iW)
+\sum_{i<j}\frac1{p_{ij}}\dim (B_{ij} W).$$
Since $V\subseteq W$, we have $\dim V\leq \dim W$. The inequality (\ref{dimension condition2}) for $V$ is true provided that it holds for $W$. Since $W$ is the intersection of null spaces of $B_i$ and $B_{ij}$ in $X$, we have $\dim W=N-\dim X$. The inequality (\ref{dimension condition2}) with $W$ in place of $V$ becomes
$$N-\dim X \leq \sum_{i=1}^N\frac1{p_i}-\sum_{B_i\in X}\frac1{p_i}
+\sum_{1\leq j<k\leq N}\frac1{p_{jk}}-\sum_{B_{jk}\in X}\frac1{p_{jk}}.
$$
By the scaling condition (\ref{scaling condition2}), we see that (\ref{dimension condition2}) is equivalent to
\begin{equation} \label{dimension condition3}
\dim X\geq \sum_{B_i\in X}\frac1{p_i}+\sum_{B_{jk}\in X}\frac1{p_{jk}}.
\end{equation}
Now we can write the full list of possible inequalities in (\ref{dimension condition2}). This contains precisely two type conditions:

$~~$ (a) For any subset $J\subseteq \{1,2,\cdots,N\}$ with $|J|\geq 2$,
$$\sum_{i\in J}\frac1{p_i}+\sum_{j<k;j,k\in J}\frac1{p_{jk}}\leq |J|,$$
where we take $X=span\{B_i,B_{jk}:~i\in J,~j,k\in J\}$ in (\ref{dimension condition3}).\\

$~~$ (b) For any subset $J\subseteq \{1,2,\cdots,N\}$ of cardinality $|J|\geq 2$,
$$\sum_{j<k;j,k\in J}\frac1{p_{jk}}\leq |J|-1,$$
where $X=span\{B_{jk}:~j,k\in J\}$ in (\ref{dimension condition3}).

Now we turn to prove that conditions (a) and (b) imply (\ref{dimension condition2}). For any subspace $V\subseteq \mathbb{R}^N$, let $X$ be defined as above. Now we define an equivalence relation for the index set $\{1,2,\cdots,N\}$. For arbitrary $i,j$, we define $i\sim j$ if $B_{ij}=B_i-B_j\in X$. Here it is obvious that $B_{ii}=0$, $B_{ij}+B_{ji}=0$ and $B_{ij}+B_{jk}=B_{ik}$. It follows that $\sim$ is indeed an equivalence relation. Now we decompose the index set into a disjoint union of subsets, with each subset being an equivalence class, i.e.,
$$\{1,2,\cdots,N\}=\bigcup_{i=1}^m J_i,$$
where we let $J_1=\{i:B_i\in X\}$ for convenience. For some $s\leq m$, $J_i$ consists of one point for $i>s$. Of course, it is possible that $J_1$ is empty and each $J_i$ consists of one single index for all $i\geq 2$. It is easy to see that
\begin{eqnarray*}
\dim span\{B_i:B_i\in X\}&=& |J_1|,\\
\dim span\{B_{jk}:j,k\in J_i\}&=& |J_i|-1,~~2\leq i \leq s.
\end{eqnarray*}
Then the decomposition gives us
$$X=span\{B_i:B_i\in X\}\bigoplus \left [\bigoplus_{i=2}^sspan\{B_{jk}:j,k\in J_i\}\right].$$
Hence
$$\dim X=|J_1|+\sum_{i=2}^s(|J_i|-1).$$

Now we can derive (\ref{dimension condition2}) from (a) and (b). Recall that (\ref{dimension condition2}) is equivalent to (\ref{dimension condition3}). We see that (a) and (b) imply
\begin{eqnarray*}
\sum_{B_i\in X}\frac1{p_i}+\sum_{B_{jk}\in X}\frac1{p_{jk}}
&=& \sum_{J_1}\frac1{p_i}+\sum_{J_1}\frac1{p_{jk}}
+\sum_{i=2}^s\sum_{J_i}\frac1{p_{jk}}\\
&\leq&|J_1|+\sum_{i=2}^s(|J_i|-1)\\
&=&\dim X.
\end{eqnarray*}
Thus (\ref{dimension condition3}) holds.

Combining above results, we have proved that $(\ref{Brascamp-Lieb ineq sec4})$ is true if and only if the following three conditions hold:\\
$~~$($\alpha$) The scaling condition is true:
$$\sum_{i=1}^N\frac1{p_i}+\sum_{1\leq j<k\leq N}\frac1{p_{jk}}=N;$$\\
$~~$($\beta$) For all $J\subseteq \{1,2,\cdots,N\}$ with $|J|\geq 2$,
$$\sum_{J}\frac1{p_i}+\sum_{J}\frac1{p_{jk}}\leq |J|;$$
$~~$($\gamma$) For all $J\subseteq \{1,2,\cdots,N\}$ with $|J|\geq 2$;
$$\sum_{J}\frac1{p_{jk}}\leq |J|-1.$$

Now we can apply the boundedness of $\Psi$ in (\ref{Brascamp-Lieb ineq sec4}) to give a complete proof of Theorem \ref{main theorem 1}. By Fubini's theorem, the statement in Theorem \ref{main theorem 1} for dimension $n=1$ implies that for higher dimensions $n\geq 2$. In fact, assume Theorem \ref{main theorem 1} is true for dimension one. For general $n\geq 2$, we write $x_i=(x_i^{(1)},\cdots,x_i^{(n)})$ and suppose that $\{p_i\}$ and $\{\alpha_{ij}\}$ satisfy all conditions in Theorem \ref{main theorem 1}. Then for $x_i,x_j\in\mathbb{R}^n$ and $\alpha_{ij}\geq 0$, it is easy to see that
$$|x_i-x_j|^{-\alpha_{ij}}\leq \prod_{t=1}^n|x_i^{(t)}-x_j^{(t)}|^{-\alpha_{ij}/n}.$$
Hence we have
\begin{eqnarray*}
& &\int_{\mathbb{R}^{n(k+1)}}
{\prod\limits_{i=1}^{k+1}|f_{i}(x_{i})}|
{\prod\limits_{1 \leq i<j\leq k+1}|x_i-x_j|^{-\alpha_{ij}}}
dx_1dx_2\cdots dx_{k+1}\\
&\leq &
\int_{\mathbb{R}^{n(k+1)}}
{\prod\limits_{i=1}^{k+1}|f_{i}(x_{i})}|
\left(\prod_{t=1}^{n}\prod\limits_{1 \leq i<j\leq k+1}|x_i^{(t)}-x_j^{(t)}|^{-\alpha_{ij}/n}\right)
\prod_{t=1}^{n}dx_1^{(t)}dx_2^{(t)}\cdots dx_{k+1}^{(t)}.
\end{eqnarray*}
For each $t$, the data $\{p_i\}$ and $\{\alpha_{ij}/n\}$ satisfy all conditions in Theorem \ref{main theorem 1}. By repeated use of the inequality (\ref{inequality of multilinear fractional functionals}) for $n=1$, we can apply Theorem \ref{main theorem 1} for $n=1$ to deduce the general result for $n\geq 2$. For this reason, it suffices to show the theorem for dimension one.

Now we assume $n=1$. For convenience, we first prove Theorem \ref{main theorem 1} in the case when the data $\{p_i\}$ and $\{\alpha_{ij}\}$ satisfy $(i)$, $(ii)$ and $(a)$ of $(iii)$. This corresponds to simple Brascamp-Lieb data in Bennett-Carbery-Christ-Tao \cite{bennet1}.  Define $p_{ij}=1/\alpha_{ij}$. Then the data $\{p_i\}$ and $\{p_{ij}\}$ satisfy the above conditions ($\alpha$), ($\beta$) and ($\gamma$), where inequalities in ($\beta$) and ($\gamma$) are strict. By this fact and the assumption $1<p_i<\infty$, we can choose $k+1$ affinely independent points $(1/p_1^{(i)},\cdots, 1/p_{k}^{(i)},1/p_{k+1}^{(i)})$ near $(1/p_1,\cdots, 1/p_{k},1/p_{k+1})$ from the hyperplane
$$x_1+\cdots+x_{k}+x_{k+1}=k+1-\sum_{\{1,\cdots,k+1\}}\alpha_{ij}/n$$
such that the data $(1/p_1^{(i)},\cdots, 1/p_{k}^{(i)})$ and $\{p_{ij}\}$ still satisfy ($\alpha$), ($\beta$) and ($\gamma$). Also, we may assume that $(1/p_1,1/p_2,\cdots, 1/p_{k+1})$ lies in the open convex hull of points
$(1/p_1^{(i)},1/p_2^{(i)},\cdots, 1/p_{k+1}^{(i)})$. Since $\Psi$ satisfies (\ref{Brascamp-Lieb ineq sec4}), $T$ is bounded from
$L^{p_1^{(i)}}\times\cdots L^{p_{k}^{(i)}}$ into
$L^{q_{k+1}^{(i)}}$ with $q_{k+1}^{(i)}$ being the conjugate number of $p_{k+1}^{(i)}$. Observe that $\sum_{i=1}^{k}1/p_i>1/p_{k+1}'=1-1/p_{k+1}$ by conditions $(i)$
and $(ii)$ in Theorem \ref{main theorem 1}. Therefore we may apply
Theorem \ref{section2. Multilinear interpolation} to conclude that
$T$ is bounded from $L^{p_1}\times\cdots L^{p_{k}}$ into
$L^{p_{k+1}'}$. By duality, we see that $\Lambda$ is bounded on $L^{p_1}\times\cdots\times L^{p_{k+1}}$.

We now prove the remaining case of Theorem \ref{main theorem 1}. For the same reason as above, we still assume $n=1$. Suppose that there are proper subsets $J$ of $\{1,2,\cdots,k+1\}$ with $|J|\geq 2$ such that $\sum_J1/p_i+\sum_J\alpha_{ij}=|J|$. This implies existence of critical subspaces for Brascamp-Lieb data; see Bennett-Carbery-Christ-Tao \cite{bennet1, bennet2}.
It is worth noting that the fact $|J|\geq 2$ follows from the assumption
$1<p_i<\infty $. Let $k+1-m$ be the maximum of $|J|$ over all these
proper subsets $J$ for some $1\leq m \leq k-1$. We take a $J_0$ such that
$|J_0|$ attains the maximum $k+1-m$ and
$\sum_{J_0}1/p_i+\sum_{J_0}\alpha_{ij}=|J_0|.$
By Theorem \ref{sec 3. weighted fracInt Operator}, we can reduce
matters to two inequalities. By the condition $(iii)$, first observe that $J_0^c$ contains at least two elements since $\sum_{J_0^c}1/p_i\geq 1$. The choice of $J_0$ implies
\begin{equation}\label{section 6. suff cond for weighted FI}
\sum_{J}\left(\frac{1}{p_i}
+\sum\limits_{j\in J_0} \alpha_{ij}\right)
+ \sum_{J}\alpha_{ij}<|J|
\end{equation}
for all nonempty proper subsets $J$ of $J_0^c$. Indeed, if it were true that
$$\sum_{J}\left(\frac{1}{p_i}
+\sum_{j\in J_0} \alpha_{ij}\right)
+ \sum_{J}\alpha_{ij}=|J|$$ for some nonempty $J\subsetneqq J_0^c$, we would obtain
$$\sum_{J\cup J_0}\frac{1}{p_i}+ \sum_{J\cup J_0}
\alpha_{ij}=|J\cup J_0|$$
which contradicts our choice of $J_0$ since $|J\cup J_0|>|J|$.

Now we claim that the inequality (\ref{section3. Type II of
weighted FI}) is true for $J_0$. By symmetry, we
may assume $J_0=\{m+1,\cdots,k+1\}$ for some $2\leq m \leq k-1$.
It is clear that $J_0^c=\{1,2,\cdots,m\}$. By H\"{o}lder's inequality for Lorentz spaces, we have
$$\|f_i|\cdot|^{-\beta_i}\|_{L^{q_i}}
\leq
C\|f_i|\cdot|^{-\beta_i}\|_{L^{q_i,1}}
\leq C\|f_i\|_{L^{p_i,1}}~~~~\frac1{q_i}=\frac1{p_i}+\beta_i$$
for each $1\leq i\leq m$. The datum $\{q_i,\alpha_{ij}:1\leq i<j\leq m\}$
is simple, i.e., $\{q_i\}$ and $\{\alpha_{ij}:1\leq i<j\leq m\}$ satisfy (i), (ii) and (a) of (iii) in Theorem \ref{main theorem 1}. Hence (\ref{section3. Type II of
weighted FI}) is true with the $L^{p_i,1}$ norm in place of the $L^{p_i}$ norm. By a similar multilinear interpolation as above, we can prove that the inequality (\ref{section3. Type II of weighted FI}) holds. By an induction argument, we can also prove (\ref{section3. Type I of weighted FI}).

Combining above results, we have completed the proof of Theorem \ref{main theorem 1}.

Let $t_1=\cdots=t_k=t=\infty$ in Theorem \ref{section2. Multilinear interpolation}. We can use the same multilinear interpolation as above to prove the following point-wise estimate.
\begin{coro}\label{section 6. special integrals general k}
Assume $\alpha_{ij}\geq 0$ for $1\leq i<j \leq k+1$ and
$1<p_i<\infty$ for $1\leq i \leq k+1$ satisfy the conditions
$(\romannumeral1)$, $(\romannumeral2)$  and $(a)$ of $(iii)$ in Theorem {\rm\ref{main theorem 1}}.
Then there exists a constant $C$ such that
\begin{equation*}\label{section 6. est spe integral for weak functions}
\int_{\bR^{nk}}\prod\limits_{1\leq i<j \leq k+1}|x_i-x_j|^{-\alpha_{ij}}
\prod\limits_{i=1}^{k}|x_i|^{-n/p_i}dx_1\cdots dx_{k}
\leq C|x_{k+1}|^{-n(1-1/p_{k+1})}.
\end{equation*}
\end{coro}

\section{Local integrability conditions and $L^\infty$ estimates}
In this section, we shall characterize $\{\alpha_{ij}\}$ for which $T(f_1,f_2,\cdots,f_k)$ in (\ref{section1. multilinear operators}) is locally integrable for $f_i\in C_0^{\infty}$. This implies that the Selberg integral associated with the correlation kernel
$\prod|x_i-x_j|^{-\alpha_{ij}}$ is finite on any bounded region in
$\bR^{n(k+1)}$. For given $\{\alpha_{ij}\}$, one natural problem arises whether there exists an $k+1-$tuple ${p_i}\in (1,\infty)$ such that the data $\{p_i,\alpha_{ij}\}$ satisfies all conditions in Theorem \ref{main theorem 1}. We shall prove in $\S$6 that this is true; see Theorem \ref{section 6. exi of infty pairs p}. The argument in this section turns out to be very
useful throughout the rest of this paper.

\begin{theorem}\label{section5. nec-suff Selberg integral}
Assume $\alpha_{ij}\geq 0$ for $1\leq i<j\le k+1$ satisfy the
integrability condition
\begin{equation}\label{section5. condition for local integrability}
\sum\limits_J \alpha_{ij}< (|J|-1)n
\end{equation}
for any subset $J$ of $\{1,2,\cdots, k+1\}$ with $|J|\geq 2$.
Then we have
\begin{equation}\label{section5. Selberg integral}
\mathcal{I}_{k+1}(\{\alpha_{ij}\})=\int_{\boldsymbol{(B_1(0))^{k+1}}}\prod\limits_{\{1,2,\cdots,k+1\}} |x_i-x_j|^{-\alpha_{ij}}dx_1dx_2\cdots dx_{k+1}<\infty,
\end{equation}
where $B_1(0)\subset \mathbb{R}^n$ is the unit ball centered at the origin.
\end{theorem}
\begin{proof}
In the case $k=1$, it is clear that the above integral converges
absolutely. For $k\geq 2$, we begin with the simplest case $k=2$ and
then make induction for general $k$. For $k=2$, it is convenient to
divide the proof into three cases.

Case 1. $\alpha_{13}+\alpha_{23}<n$\\
It is clear that
\begin{eqnarray*}
&&\int_{(B_1(0))^3}\prod\limits_S |x_i-x_j|^{-\alpha_{ij}}dx_1dx_2 dx_3\\
&\leq &
 C \int_{(B_1(0))^2}|x_1-x_2|^{-\alpha_{12}}dx_1dx_2
\end{eqnarray*}
which is finite by the assumption $\alpha_{12}<n$.

Case 2. $\alpha_{13}+\alpha_{23}=n$\\
Using Lemma \ref{Sec2. estimate of FI on the critical order}, we obtain
\begin{eqnarray*}
&&\int_{(B_1(0))^3}\prod\limits_S |x_i-x_j|^{-\alpha_{ij}}dx_1dx_2 dx_3\\
&\leq& C\int_{(B_1(0))^2}|x_1-x_2|^{-\alpha_{12}}\log(4/|x_1-x_2|)
dx_1dx_2 < \infty.
\end{eqnarray*}

Case 3. $\alpha_{13}+\alpha_{23}>n$\\
Observe that
$$\int_{B_1(0)}|x_1-x_3|^{-\alpha_{13}}|x_2-x_3|^{-\alpha_{23}}dx_3
\leq C|x_1-x_2|^{n-\alpha_{13}-\alpha_{23}}$$
which implies the integral in (\ref{section5. Selberg integral}) is finite by the assumption
 $\alpha_{12}+\alpha_{13}+\alpha_{23}<2n$.

We now consider the general case $k\geq 3$. Assume
all $k$ fold integrals of form $(\ref{section5. Selberg integral})$ converge under the assumption (\ref{section5. condition for local integrability}). We shall prove that the $k+1$ fold integral is absolutely convergent. Let
$$\Theta=\{i: 1\le i \le k,\quad \alpha_{i,k+1}>0\}.$$
By simple calculations, it is easy to verify our claim in the case
$\sum_\Theta \alpha_{i,k+1}\leq n$. Indeed, if $\sum_\Theta
\alpha_{i,k+1}$ is less than $n$, we first take integration with respect to $x_{k+1}$ and then the matter reduces to a $k-$multiple integral. If $\sum_\Theta \alpha_{i,k+1}=n$, it follows from Lemma \ref{Sec2.
estimate of FI on the critical order} that
$$\int_{B_1(0)}\prod_{\Theta}|x_i-x_{k+1}|^{-\alpha_{i,k+1}}dx_{k+1}
\leq C \l(\sum_\Theta |x_i-x_j|\r)^{-\varepsilon}$$ where
$\varepsilon>0$ is a small number to be determined. Choose
$i_0,j_0\in \Theta$ with $i_0<j_0$. Let $$\overline{\alpha_{ij}}
=\alpha_{ij}+\boldsymbol{\delta}_i^{i_0}
\boldsymbol{\delta}_j^{j_0}\epsilon,\quad 1\leq i<j\leq k,$$ where
${\textrm{\boldmath $\delta$}}_{s}^{t}$ is the Kronecker symbol. In
other words, ${\textrm{\boldmath $\delta$}}_{s}^{t}=1$ if $s=t$ and
${\textrm{\boldmath $\delta$}}_{s}^{t}=0$ otherwise. If $\epsilon$
is sufficiently small, then $\{\overline{\alpha_{ij}}\}$ still
satisfies the integrability condition. Therefore the $k+1$-multiple
integral in (\ref{section5. Selberg integral}) is less than a constant multiple of a $k$-multiple integral with $\{\alpha_{ij}\}$ replaced by $\{\overline{\alpha_{ij}}\}$. The desired result follows by induction.

The crux of the proof lies in the case $\sum_\Theta
\alpha_{i,k+1}>n$. The argument depends on the number of elements in $\Theta$. The simplest case is $|\Theta|=2$ in which the argument is direct. For $3\leq |\Theta|\leq k$, we shall reduce the matter to the case $|\Theta|=2$ by invoking a useful procedure. Indeed, if $|\Theta|=2$, we may assume $\Theta = \{1,2\}$ by the symmetry of parameters. Then by a similar treatment in Case 3 for $k=2$, put
\begin{equation}
\overline{\alpha_{ij}}= \alpha_{ij}+{\textrm{\boldmath $\delta$}}_{i}^{1}{\textrm{\boldmath $\delta$}}_{j}^{2}
\l(\sum_\Theta \alpha_{i,k+1}-n\r), \quad 1\leq i <j \leq k.
\end{equation}
It is easy to verify that (\ref{section5.
condition for local integrability}) are still true for
$\{\overline{\alpha_{ij}}\}$. Indeed, it suffices to show that
(\ref{section5. condition for local integrability}) holds for those
$J$ containing both 1 and 2 since we obviously have
$$\sum_J\alpha_{ij}=\sum_J\overline{\alpha_{ij}}$$
for all subsets $J$ of $S=\{1,2,\cdots,k\}$ satisfying
$\{1,2\}\nsubseteq J$. For $J\subseteq S$ containing $\Theta=\{1,2\}$ as a subset, it follows from the definition of $\{\overline{\alpha_{ij}}\}$ that
\begin{equation*}
\sum_J\overline{\alpha_{ij}}=\sum_{J\cup\{k+1\}}\alpha_{ij}-n
<(|J|-1)n
\end{equation*}
by the assumption (\ref{section5. condition for local integrability}). Hence by induction the integral in (\ref{section5. Selberg integral}) converges in the case $|\Theta|=2$.

If $|\Theta|=m$ with $3\le m \le k$, our idea is to show that the
$k+1-$multiple integral is dominated by summation of two similar kinds of integrals by distributing some powers into $\{\alpha_{ij}:1\le i<j \leq k\}$ appropriately. The first type of these integrals has a $k-point$ correlation integrand. The other type is the same as the integral in (\ref{section5. Selberg integral}) with $|\Theta|=m-1$. Without loss of generality, we may assume $\Theta=\{1,2,\cdots,m\}.$ By the assumption
$\sum_\Theta\alpha_{i,k+1}>n$, we use Theorem \ref{Estimtates of
Generalization of Riesz Potential} to obtain \begin{equation}
\int_{\mathbb{R}^n}
\prod\limits_{\Theta}|x_i-x_{k+1}|^{-\alpha_{i,k+1}}dx_{k+1} \le C
\sum_{\Theta} \boldsymbol{L}_{i},
\end{equation}
where $\boldsymbol{L}_{i}$ are defined as in Theorem \ref{Estimtates of Generalization of Riesz Potential}.

Replacing the integral relative to $x_{k+1}$ by each $\boldsymbol{L}_{i}$, we shall prove that each $k+1-$multiple integral is dominated by integrals of the above two types.

If $\sum_{i=2}^m \alpha_{i,k+1}<n$, we shall prove that
\begin{eqnarray*}
&&\int_{(B_1(0))^k}\l(\prod\limits_S |x_i-x_j|^{-\alpha_{ij}}\r)\l(\sum_\Theta |x_i-x_j|\r)^{n-\sum_\Theta \alpha_{i,k+1}}dx_1\cdots dx_k\\
&\leq &
C \int_{(B_1(0))^k}\prod\limits_S |x_i-x_j|^{-\alpha_{ij}-\delta_{ij}}dx_1\cdots dx_k,
\end{eqnarray*}
where $\{\overline{\alpha_{ij}}=\alpha_{ij}+\delta_{ij}\}$ satisfies (\ref{section5. condition for local
integrability}). Here $\sum_\Theta
\delta_{ij}=\sum_{\Theta}\alpha_{i,k+1}-n$ for $\delta_{ij}\geq 0$
and $\delta_{ij} = 0$ if either $i$ or $j$ does not lie in $\Theta$. Now we turn our attention to the existence of such a solution $\{\delta_{ij}\}$. In other words, we need solve the following system of linear inequalities:
\begin{equation}
(V.1)
\begin{cases}
(i)\quad\; \delta_{ij}\geq 0, \quad 1\le i <j \le m;\nonumber\\
(ii) \quad \sum\limits_{\Theta}\delta_{ij}=\sum\limits_{\Theta}\alpha_{i,k+1}-n;
\nonumber\\
(iii)\; \sum\limits_{J\cap\Theta}\delta_{ij}<(|J|-1)n-\sum\limits_{J}\alpha_{ij}
\quad {\rm{for}}\quad J\in \mathcal{F}_m, \nonumber
\end{cases}
\end{equation}
where the class $\mathcal{F}_m$ consists of all subsets $J$ of
$\{1,\cdots,k\}$ satisfying $|J\bigcap \Theta|\geq 2$. Note that we
have assumed $\Theta=\{1,2,\cdots,m\}$. Here we use the notation
$\mathcal{F}_m$ instead of $\mathcal{F}_{\Theta}$ for simplicity.

Now  we shall apply Lemma \ref{section2. existence System of LI} to show the existence of solutions $\delta_{ij}$. Obviously, $(i)$ and $(ii)$ in $(V.1)$ have solutions and so Lemma \ref{section2. existence System of LI} is applicable. For arbitrary nonnegative numbers $\lambda_{ij}$, $\theta_1$, $\theta_2$ and $\mu_J$ with at least one $\mu_J>0$ for some $J\in\mathcal{F}_m$ satisfying
\begin{equation}\label{section5. (1) conditions satisfied by parameters}
\lambda_{ij}+(\theta_1-\theta_2)-\sum\limits_{J\ni i,j}\mu_J=0
\end{equation}
for $1\le i<j \le m$, we must show that
\begin{equation}\label{section5. object function of LI}
(\theta_1-\theta_2)\Big(\sum\limits_{\Theta}\alpha_{i,k+1}-n \Big)
-\sum\limits_{J\in \mathcal{F}_m} \mu_J \Big((|J|-1)n-\sum\limits_{J}\alpha_{ij}\Big)<0.
\end{equation}
It suffices to prove this inequality when $\theta_1-\theta_2>0$
since there exists one $\mu_J>0$ for some $J$ and $\sum_J \alpha_{ij}<(|J|-1)n$. Now assume $\theta_1-\theta_2>0$. By dilation, put
$\theta_1-\theta_2=1$. Then $\mu_J$ and $\lambda_{ij}$ satisfy
\begin{equation}\label{section5. (2) conditions satisfied by parameters}
\sum\limits_{J\ni i,j}\mu_J=1+\lambda_{ij}
\end{equation}
for $1\le i<j \le m$. To prove this inequality, a basic idea is to
determine the supremum of the objective function in (\ref{section5. object function of LI}). Though the supremum cannot be attained generally, $\mu_J$ and
$\lambda_{ij}$ have simple forms when the value of the
objective function is sufficiently close to its supremun. More
precisely, for any $\{\mu_J(0),\lambda_{ij}(0)\}$, we shall construct a sequence $\{\mu_J(N),\lambda_{ij}(N)\}$ such that the objective function for $\{\mu_J(N),\lambda_{ij}(N)\}$ increases. By taking $N\rightarrow\infty$, the sign of the objective function will be easily verified.

Now we turn to construct such a process. Suppose $\{\mu_J(N-1)\}$ are given. For convenience, we define the following conditions for two subsets $J_1$ and $J_2$:
\begin{equation}\label{section 5. restrictions on subsets}
\textrm{$\mathbf{(a)~~\mu_{J_1}(N-1)\mu_{J_2}(N-1)>0}$;~~~~~
$\mathbf{(b)~~J_1\cap J_2\neq \emptyset}$;~~~~~
$\mathbf{(c)~~J_1\nsubseteq J_2~~\textrm{and}~~J_2\nsubseteq J_1}$}.
\end{equation}
If $\{\mu_J(N-1)\}$ and $\{\lambda_{ij}(N-1)\}$ are known, then we choose two subsets $J_1$ and $J_2$ satisfying suitable restrictions and set $\mu_{J}(N)$ as follows. We shall explain later why these restrictions on $J_1$ and $J_2$ are required. In the following, the notation $A\wedge B$ means $\min\{A,B\}$.
\\

\noindent\textrm{$\mathbf{Case~~I:~~J_1,J_2\in \mathcal{F}_m ~~satisfy~~the ~~above~~conditions~~(a),~(b),~(c)~~and~~
  J_1\bigcap J_2 \in \mathcal{F}_m.} $ }\\
 \begin{equation}\label{section5. Process Type I}
 \begin{cases}
 \mu_{J_1}(N)= \mu_{J_1}(N-1)- \mu_{J_1}(N-1)\wedge  \mu_{J_2}(N-1)\\

 \mu_{J_2}(N)= \mu_{J_2}(N-1)-\mu_{J_1}(N-1)\wedge  \mu_{J_2}(N-1)\\

  \mu_{J_1\cap J_2}(N)=\mu_{J_1\cap J_2}(N-1)+ \mu_{J_1}(N-1)\wedge  \mu_{J_2}(N-1)\\

 \mu_{J_1\cup J_2}(N)=\mu_{J_1\cup J_2}(N-1)+ \mu_{J_1}(N-1)\wedge  \mu_{J_2}(N-1)\\

 \mu_{J}(N)=\mu_J(N-1), \quad J\notin \{J_1,J_2,J_1\cap J_2,J_1\cup J_2\}.
 \end{cases}
 \end{equation}

\noindent\textrm{$\mathbf{Case~~II:~~J_1,J_2\in \mathcal{F}_m~~satisfy~~(a),~(b),~(c)~~and~~
J_1\bigcap J_2 \notin \mathcal{F}_m.}$} \\
 \begin{equation}\label{section5. Process Type II}
 \begin{cases}
 \mu_{J_1}(N)= \mu_{J_1}(N-1)- \mu_{J_1}(N-1)\wedge  \mu_{J_2}(N-1)\\

 \mu_{J_2}(N)= \mu_{J_2}(N-1)-\mu_{J_1}(N-1)\wedge  \mu_{J_2}(N-1)\\

 \mu_{J_1\cup J_2}(N)=\mu_{J_1\cup J_2}(N-1)+ \mu_{J_1}(N-1)\wedge  \mu_{J_2}(N-1)\\

 \mu_{J}(N)=\mu_J(N-1), \quad J\notin \{J_1,J_2,J_1\cup J_2\}.
 \end{cases}
 \end{equation}
We also define $\lambda_{ij}(N)$ by
(\ref{section5. (2) conditions satisfied by parameters}) with
$\{\mu_J(N)\}$ in place of $\{\mu_J\}$.

The motivation for construction of such a process is the following inequality,
\begin{eqnarray}\label{section5. cliam: object function is increasing}
&&\Big((|J_1\cap J_2|-1)n
-\sum\limits_{J_1\cap J_2}\alpha_{ij}\Big)
+\Big((|J_1\cup J_2|-1)n
-\sum\limits_{J_1\cup J_2}\alpha_{ij}\Big)\nonumber\\
&\leq &\sum\limits_{s=1}^2\Big((|J_s|-1)n
-\sum\limits_{J_s}\alpha_{ij}\Big)
\end{eqnarray}
for all $J_1$ and $J_2$ in the class $\mathcal{F}_m$. Here we use
the summation convention $\sum_J\alpha_{ij}=0$ if $|J|\leq 1$.

It is helpful to make some observations. First
we claim that $\lambda_{ij}(N)$ increases as $N$. Assume two subsets $J_1$ and $J_2$ are chosen in the $N-th$ step. For each pair $i$ and $j$, there are several possible
cases. If either $i\notin J_1\cup J_2$ or $j \notin J_1\cup J_2$,
then $\lambda_{ij}(N)=\lambda_{ij}(N-1)$. Now we assume $i,j \in J_1\cup J_2$ and divide this into three subcases. $\boldsymbol{(1)}$~~~If $\{i,j\} \subseteq
J_1\cap J_2$, then it is easy to see that
$\lambda_{ij}(N)=\lambda_{ij}(N-1)$. $\boldsymbol{(2)}$~~~If $i\notin J_1\cap J_2$ or
$j\notin J_1\cap J_2$ but either $\{i,j\}\subseteq J_1$ or $\{i,j\}\subseteq J_2$, then $\lambda_{ij}$ remains unchanged in the $N$ step. $\boldsymbol{(3)}$~~~The remaining
subcase is that $\{i,j\}\nsubseteq J_1$ and $\{i,j\}\nsubseteq J_2$ in which
$\lambda_{ij}$ increases in the $N$-th step. Thus we have established our claim.

The key observation is that the objective function also increases as $N$. Equivalently, we have
\begin{eqnarray}\label{section5. object function is increasing}
-\sum\limits_{J\in \mathcal{F}_m} \mu_J(N-1) \l((|J|-1)n-\sum\limits_{J}\alpha_{ij}\r)
\leq
-\sum\limits_{J\in \mathcal{F}_m} \mu_J(N) \l((|J|-1)n-\sum\limits_{J}\alpha_{ij}\r)
\end{eqnarray}
for any $N\geq 1$. Indeed, this observation is an immediate consequence of the inequality
(\ref{section5. cliam: object function is increasing}) and its simple variant
$$(|J_1\cup J_2|-1)n
-\sum\limits_{J_1\cup J_2}\alpha_{ij}
\leq \sum\limits_{s=1}^2\Big((|J_s|-1)n
-\sum\limits_{J_s}\alpha_{ij}\Big)$$
with the additional assumption $|J_1 \cap J_2|\geq 1$. This explains why the recursion (\ref{section5. Process Type II}) only applies to those $J_1$ and $J_2$ having nonempty intersection; condition $\boldsymbol{(b)}$ in (\ref{section 5. restrictions on subsets}).

Now we shall introduce some subclasses of $\mathcal{F}_m$. Let
$\mathcal{A}_m$, $\mathcal{B}_{m}$ and $\mathcal{C}_m$ be defined by
 \begin{equation}\label{section5. def of Am Bm Cm}
 \mathcal{A}_m=\{J\in\mathcal{F}_m:\; \mu_J>0\},\quad
  \mathcal{B}_m=\{J\in\mathcal{A}_m:\; \Theta \nsubseteq J \},\quad
   \mathcal{C}_m=\{J\in\mathcal{A}_m:\; \Theta \subseteq J\}.
 \end{equation}
For convenience, we also use $\mathcal{A}_m(N)$, $\mathcal{B}_{m}(N)$ and $\mathcal{C}_m(N)$ defined similarly
as above to keep track of the above process. It is clear that
$\mathcal{A}_{m}=\mathcal{B}_{m}\bigcup\mathcal{C}_m$ and $\mathcal{B}_{m}\bigcap\mathcal{C}_m=\emptyset$.

\begin{defn}
If $\{\mu_J: J\in \mathcal{F}_m\}$ is invariant under any possible
process described as in {\rm{(\ref{section5. Process Type I})}} and
{\rm(\ref{section5. Process Type II})}, then we say that
$\{\mu_J: J\in \mathcal{F}_m\}$ is stable.

Let $F(\{\mu_J\})$ be a function of $\mu_J \geq 0$ for
$J\in\mathcal{F}_m$. Assume $\{\mu_J(0):J\in\mathcal{F}_m\}$ is a
set of nonnegative numbers. We say that $F$ is stable with respect
to $\{\mu_J(0)\}$ if for all $N\geq 1$
$F(\{\mu_J(0)\})=F(\{\mu_J(N)\})$, where $\{\mu_J(N)\}$ is obtained
by an arbitrary process of $N$ steps.
\end{defn}
 By this definition,
$\{\mu_J: J\in \mathcal{F}_m\}$ is stable if and only if for all
$J_1,J_2\in\mathcal{A}_m$ one of the three relations holds: (i)
$J_1\cap J_2=\emptyset$; (ii) $J_1\subset J_2$; (iii) $J_2\subset J_1$. Further observation also shows
$\sum_{J\in\mathcal{C}_m}\mu_J=1+\min_{\Theta}\lambda_{ij}$ when
$\{\mu_J\}$ is stable. This observation will be proved later. For
convenience, we introduce the notation $\Omega_m$ to denote
\begin{equation}\label{section5. def of Omega}
\Omega_m=\sum_{J\in\mathcal{C}_m}\mu_J.
\end{equation}
 And $\Omega_m(N)$ is defined as above with $\mathcal{C}_m$ and $\mu_J$ replaced by $\mathcal{C}_m(N)$ and $\mu_J(N)$ respectively.

We do not know whether arbitrary $\{\mu_J(0): J\in \mathcal{F}_m\}$ and $\{\lambda_{ij}(0)\}$ satisfying (\ref{section5. (2) conditions
satisfied by parameters}) can reach a stable state by a process
consisting of finite steps. However, by a passage to the limit, we
can arrive at a special state, not stable generally, which is
enough for our purpose. Let $\{\mu_J^\ast(N):J\in \mathcal{F}_m\}$
be obtained by a process of $N$ steps described as in
(\ref{section5. Process Type I}) and (\ref{section5. Process Type
II}) such that the supremum
$$\Omega_{m}^\ast(N)=\sup\Omega_m(N)$$
is attained. Here the supremum is taken over all possible processes
consisting of $N$ steps. It is possible that these processes are not unique. We can take one of these processes by which $\Omega_{m}^{\ast}(N)$ is obtained. It should be pointed out that $\{\mu_J^\ast(N)\}$ is not obtained by a continuous
procedure with respect to $N$. Therefore in general we cannot obtain
$\{\mu_J^\ast(N)\}$ from $\{\mu_J^\ast(N-1)\}$ by one step. On the
one hand, it is clear that $\Omega_{m}^{\ast}(N)$ increases as $N$.
On the other hand, we also have
\begin{equation}\label{section5. mu S(N) is bounded}
\sum\limits_{J\ni i}\mu_J^\ast(N)
\leq
\sum_{J\ni i}\mu_J(0)
\leq
\max\limits_{1\leq i \leq k}\l(\sum_{J\ni i}\mu_J(0)\r),
\end{equation}
for any $i\in S=\{1,2,\cdots,k \}$. Suppose at the $k$-th step with
$k\leq N-1$ the recursion is applied to $J_1$ and $J_2$ in
$\mathcal{F}_m$. Then we see that $\sum_{J\ni i}\mu_J(k)=\sum_{J\ni
i}\mu_J(k+1)$ unless $i\in J_1\cap J_2$ and $J_1\cap J_2 \notin
\mathcal{F}_m$. In the latter case, we have $\sum_{J\ni
i}\mu_J(k)>\sum_{J\ni i}\mu_J(k+1)$.  There is another similar
observation as (\ref{section5. mu S(N) is bounded}). In the system
${\rm(V.1)}$, $\delta_{ij}$ is assumed to be zero if $i\notin
\Theta$ or $j\notin \Theta$. If $\{\mu_J(0)\}$ and
$\{\lambda_{ij}(0)\}$ satisfy (\ref{section5. (2)
conditions satisfied by parameters}), we may assume that
$\mu_J(0)=0$ for nonempty $J\subset S$ but $J\notin \mathcal{F}_m$.
Of course we have $\mu_J(N)=0$ if $J\notin \mathcal{F}_m$ for any
$N$. Then it is clear that
\begin{equation}\label{section5. muJ(N) are bounded}
\sum_{J\subset S}\mu_J(N)\leq \sum_{J\subset S}\mu_J(0),
\end{equation}
where $\mu_J(N)$ are obtained by any possible process with $N$ steps. Both $(\ref{section5. mu S(N) is bounded})$ and
$(\ref{section5. muJ(N) are bounded})$ imply that
$\{\Omega_{m}^{\ast}(N)\}$ has a uniform upper bound. Put
\begin{equation}\label{section5. limit of Omegam(infty)}
\Omega_{m}^{\ast}(\infty)=\lim\limits_{N\rightarrow\infty}\Omega_{m}^{\ast}(N).
\end{equation}
This limit is well defined since $\{\Omega_{m}^{\ast}(N)\}$ is a
bounded increasing sequence. It is possible that
$\Omega_{m}^{\ast}(\infty)$ can be obtained by a process of finite
steps, i.e., $\Omega_{m}^{\ast}(\infty)=\Omega_{m}^{\ast}(N)$ for
$N\ge N_0$ for some $N_0$. In this situation,
$\Omega_{m}^{\ast}(N_0)$ becomes stable. To calculate $\Omega_{m}^{\ast}(\infty)$ explicitly, we are going to establish a necessary and sufficient condition under which $\Omega_{m}$ is stable. It will be convenient to introduce a concept related to a sequence of sets.
\begin{defn}
If $\{J_i\}_{i=1}^a$ with $a\geq2$ is a sequence of sets with the property that each intersection $\big(\bigcup_{t=1}^{i}J_t\big)\bigcap J_{i+1}$ is nonempty for $1\leq i \leq a-1$, then we call $\{J_i\}_{i=1}^a$ a continuous chain.
\end{defn}

\begin{lemma}\label{section5. stable condition}
Assume $\{\mu_J:J\in \mathcal{F}_m\}$ is a set of nonnegative numbers and $\Omega_m$ is defined as in {\rm(\ref{section5. def of Omega})}. Then $\Omega_m$ is stable with respect to $\{\mu_J\}$ if and only if $\Theta \nsubseteqq \bigcup_{i=1}^{a}J_i$ for any continuous chain $\{J_i\}_{i=1}^a$ in $\mathcal{B}_m$.
\end{lemma}
\begin{proof}
We first show the necessity part. Assume the converse, i.e., $\Omega_m$ is stable and there exists a continuous chain $\{J_i\}_{i=1}^a$ in $\mathcal{B}_m(0)$
such that $\Theta \subseteq \bigcup_{i=1}^{a}J_i$. We shall see that
there is a process such that $\Omega_m(a-1)>\Omega_m(0)$. First
applying the recursion in (\ref{section5. Process Type
I}) and (\ref{section5. Process Type II}) to $J_1$ and $J_2$,
we obtain $\{\mu_J(1)\}$. In the second step, we repeat the recursion
with respect to $J_1\cup J_2$ and $J_3$. Likewise, in the $i$-th
step we apply the recursion to $\bigcup_{1\le t \le i}J_t$ and
$J_{i+1}$. Here we assume $\bigcup_{1\le t \le i}J_t$ and
$J_{i+1}$ satisfy $\mathbf{(c)}$ in (\ref{section 5. restrictions on subsets}). If this is not the case, we can skip over to apply the recursion to $\bigcup_{1\le t \le i+1}J_t$ and $J_{i+2}$.
After at most $a-1$ steps, it follows from $J_i\nsubseteqq\Theta$ that
\begin{equation}\label{section5. lower bounds between muS(N)}
\Omega_m(a-1)
\geq
\Omega_m(0)+\min_{1\leq i\leq a}\mu_{J_i}(0)
\geq
\Omega_m(0)+\min_{J\in \mathcal{A}_m}\mu_{J}(0)
\end{equation}
which contradicts the assumption that $\Omega_m(0)$ is stable.

The proof of the sufficiency part is intricate. We first establish a useful property of $\mathcal{B}_m$ under the assumption that for all continuous chains $\{J_i\}_{i=1}^a$ in $\mathcal{B}_m$, the union $\bigcup_{i=1}^aJ_i$ does not contain $\Theta$ as a subset.
\begin{prop}\label{section5. prop for Bm}
Suppose that $\Theta\nsubseteqq \bigcup_{i=1}^aJ_i$ for all continuous chains $\{J_i\}_{i=1}^a$ in $\mathcal{B}_m$. Then there exists a nonempty proper subset $\Theta_1$ of $\Theta$ such that for all $J_1$ and $J_2$ in $\mathcal{B}_m$ with $J_1\bigcap J_2\neq \emptyset$, it is true that:\\
\begin{equation}\label{section5. property of stabel Bm}
\textrm{either}\quad\Big(J_1\bigcup J_2\Big)\bigcap \Theta \subset \Theta_1 \quad \textrm{or}
\quad\Big(J_1\bigcup J_2\Big)\bigcap \Theta \subset \Theta-\Theta_1.
\end{equation}
Generally, this property also holds for all continuous chains $\{J_i\}_{i=1}^a$ in $\mathcal{B}_m$. More precisely, we have
\begin{equation*}
\textrm{either}\quad
\Big(\bigcup_{i=1}^aJ_i\Big)\bigcap \Theta \subset \Theta_1 \quad \textrm{or}
\quad\Big(\bigcup_{i=1}^aJ_i\Big)\bigcap \Theta \subset \Theta-\Theta_1.
\end{equation*}
\end{prop}
Now we turn to prove this proposition. Define $\eta$ by
\begin{equation}\label{section5. def of eta}
\eta=\sup\Big|\Big(\bigcup_{i=1}^aJ_i\Big)\bigcap \Theta\Big|,
\end{equation}
where the supremum is taken over all continuous chains $\{J_i\}$ in $\mathcal{B}_m$. Since $\eta$ is an integer, we see that $\eta$ can be achieved for some continuous chain $\{I_i\}_{i=1}^{a}$ in $\mathcal{B}_m$. Let $\Theta_1=\big(\bigcup_{i=1}^a I_i\big)\bigcap \Theta$. Then $\Theta_1$ is a proper subset of $\Theta$. For all $J$ in $\mathcal{B}_m$, we claim that either $J\bigcap\Theta\subseteq\Theta_1$ or
$J\bigcap \big(\bigcup_{i=1}^a I_i\big)=\emptyset$ is true. Otherwise, if there were some $I_{a+1}\in\mathcal{B}_m$ such that
$I_{a+1}\cap\Theta\nsubseteqq\Theta_1$ and
$I_{a+1}\bigcap \big(\bigcup_{i=1}^a I_i\big)\neq\emptyset$, it would follow that
$\{I_i\}_{i=1}^{a+1}$ is a continuous chain in $\mathcal{B}_m$ and satisfies
$\Big|\Big(\bigcup_{i=1}^{a+1}I_i\Big)\bigcap \Theta\Big|>\eta$. This contradicts the definition of $\eta$.

Define $\mathcal{B}_m^{(1)}$ and $\mathcal{B}_m^{(2)}$ by
\begin{equation*}
\mathcal{B}_m^{(1)}
=\Big\{J\in\mathcal{B}_m:\; J\bigcap \big(\bigcup_{i=1}^a
I_i\big)=\emptyset\Big\},
\quad
\mathcal{B}_m^{(2)}=\mathcal{B}_m-\mathcal{B}_m^{(1)}.
\end{equation*}
Then $J\bigcap\Theta\subseteq\Theta_1$ for all $J\in
\mathcal{B}_m^{(2)}$. Moreover, we also have that $J_1\bigcap
J_2=\emptyset$ if $J_1\in \mathcal{B}_m^{(1)}$ and $J_2\in
\mathcal{B}_m^{(2)}$. Actually, if there were $J_1\in
\mathcal{B}_m^{(1)}$ and $J_2\in \mathcal{B}_m^{(2)}$ with
$J_1\bigcap J_2\neq\emptyset$, it would follow that
$\{I_i\}_{i=1}^{a+2}$ is a continuous chain in $\mathcal{B}_m$ and
$\Big|\Big(\bigcup_{i=1}^{a+2}I_i\Big)\bigcap \Theta\Big|>\eta$
where $I_{a+1}=J_2$ and $I_{a+2}=J_1$. This contradicts our choice
of $\eta$. Therefore one of two relations in (\ref{section5.
property of stabel Bm}) is true if $J_1\bigcap J_2\neq\emptyset$. It remains to prove that the same conclusion holds for any continuous chain in $\mathcal{B}_m$. Without loss of generality, we may assume $a=3$. Since the intersection of $J_1$ and $J_2$ is nonempty, we see that either $J_1,J_2\in \mathcal{B}_m^{(1)}$ or $J_1,J_2\in\mathcal{B}_m^{(2)}$ is true. By the assumption that $J_1,J_2,J_3$ is a continuous chain, we have either $J_1\bigcap J_3\neq \emptyset$ or $J_2\bigcap J_3\neq \emptyset$. Thus either $J_1,J_2,J_3\in \mathcal{B}_m^{(1)}$ or $J_1,J_2,J_3\in\mathcal{B}_m^{(2)}$ holds. We conclude the proof of
Proposition \ref{section5. prop for Bm}.\\

Now we shall invoke Proposition \ref{section5. prop for Bm} to give
a complete proof of the sufficiency part of Lemma \ref{section5. stable
condition}. For initial data $\{\mu_J(0)\}$ and
$\{\lambda_{ij}(0)\}$ satisfying
(\ref{section5. (2) conditions satisfied by parameters}), it follows from Proposition \ref{section5. prop for Bm}
that there is a nonempty proper $\Theta_1\subseteq\Theta$ such that for
all $J_1,\;J_2\in\mathcal{B}_m(0)$ satisfying $J_1\bigcap J_2\neq
\emptyset$, one of two relations in (\ref{section5. property of
stabel Bm}) is true. We shall prove that $\Omega_m(0)$ is stable. In other words, $\Omega_m(0)=\Omega_m(N)$ for any process of $N$ steps. Here $\{\mu_J(N)\}$ is obtained from $\{\mu_J(0)\}$ by any process consisting of $N$ steps as in
(\ref{section5. Process Type I}) and (\ref{section5. Process Type II}). At the $k-$th step in which $\{\mu_J(k)\}$ is obtained, assume that we apply the recursion
to $J_1$ and $J_2$ in $\mathcal{A}_m(k-1)$. Then $J_1\bigcap
J_2\neq\emptyset$. We divide the $k$-th step into
four cases:\\
\quad\;(a) $J_1$ and $J_2$ are in $\mathcal{C}_m(k-1)$;\qquad\qquad
\quad \;(b) $J_1\in \mathcal{B}_m(k-1)$ and $J_2\in \mathcal{C}_m(k-1)$;\\
\quad\;(c) $J_1\in \mathcal{C}_m(k-1)$ and $J_2\in \mathcal{B}_m(k-1)$;
\quad\;(d) $J_1$ and $J_2$ are in $\mathcal{B}_m(k-1)$.\\
We claim that $\Omega_m(k-1)=\Omega_m(k)$. For $J_1$ and $J_2$ in
the cases (a), (b) and (c), this statement is easily verified. In
the case (d), we need an additional property of $\mathcal{B}_m(k)$.
In other words, Proposition \ref{section5.
prop for Bm} is still true for all $\mathcal{B}_m(k)$ with the same
$\Theta_1$. Since $\mathcal{B}_m(0)$ has the property in Proposition \ref{section5. prop for Bm}, it suffices to
show that $\mathcal{B}_m(1)$ shares this property. Assume in the
first step, the recursion applies to $J_1$ and $J_2$ in
$\mathcal{F}_m$. Observe that $J_1\cap J_2$ and $J_1\cup J_2$ are
the only two subsets which are possibly contained in
$\mathcal{B}_m(1)$ but not in $\mathcal{B}_m(0)$. It follows from
the above proposition that both intersections $(J_1\cap J_2)\cap
\Theta$ and $(J_1\cup J_2)\cap\Theta$ are subsets of $\Theta_1$ or
$\Theta-\Theta_1$. Hence the assumption in the proposition is also
valid for $\mathcal{B}_m(1)$. By induction, it follows that the same result holds for all $\mathcal{B}_m(k)$. Thus in the case (d), we still have $\Omega(k)=\Omega(k-1)$. This implies that $\Omega(0)$ is stable. The proof of Lemma \ref{section5. stable condition} is complete.
\end{proof}

We now turn our attention to prove that $\Omega_m^{\ast}(\infty)$ is stable. Recall that both (\ref{section5. mu S(N) is bounded}) and
(\ref{section5. muJ(N) are bounded}) imply the uniform boundedness
of $\mu_J^\ast(N)$ and $\lambda_{ij}^\ast(N)$. Thus by passing
$\{N\}$ to a subsequence, denoted by $\{N_t\}$, such that
$\mu_J^\ast(N_t)$ and $\lambda_{ij}^\ast(N_t)$ converge to
$\mu_J^\ast(\infty)$ and $\lambda_{ij}^\ast(\infty)$ respectively.
It is easy to see that the limits $\mu_J^\ast(\infty)$ and
$\lambda_{ij}^\ast(\infty)$ still satisfy (\ref{section5. (2)
conditions satisfied by parameters}). By the definition of
$\Omega^{\ast}_m(\infty)$, we have
$$\Omega_m^\ast(\infty)=\sum_{J\in C_m(\infty)}\mu_J^{\ast}(\infty)$$
with $\mathcal{C}_m(\infty)=\{J\in\mathcal{F}_m:\Theta\subseteq J,~\mu_J^\ast(\infty)>0\}$. Let $\mathcal{A}_m(\infty)$ be the class of
all $J\in\mathcal{F}_m$ satisfying $\mu_J^\ast(\infty)> 0$.
Similarly, $B_m(\infty)$ can be defined as (\ref{section5. def of Am
Bm Cm}). With these notations, we claim that $\Omega^{\ast}_m(\infty)$ is stable with respect to $\mu^{\ast}_J(\infty)$. This means that any application of above processes to $\mathcal{A}_m(\infty)$ does not change the value of $\Omega^{\ast}_m(\infty)$. Otherwise, by Lemma
\ref{section5. stable condition}, there exists a continuous chain
$\{J_i\}_{i=1}^{a}$ in $\mathcal{B}_m(\infty)$ satisfying
$\Theta\subseteq \bigcup_{i=1}^a J_i$. For all $t\geq t_0$ with some sufficiently large $t_0$, first observe that $\{\mu_J^\ast(N_t)\}$ has a uniform positive lower bound for each $J\in \mathcal{A}_m(\infty)$. By (\ref{section5. lower bounds between muS(N)}) and related remarks, it follows that
\begin{equation}
\Omega_m^\ast(N_t+a-1)
\geq \Omega_m^\ast(N_t)+\min_{J\in \mathcal{A}(\infty)}\mu_J^\ast(N_t)
\geq \Omega_m^\ast(N_t)+\frac{1}{2}\min_{J\in \mathcal{A}(\infty)}\mu_J^\ast(\infty)
\end{equation}
for all sufficiently large $N_t$. This contradicts the fact that
$\Omega_m^\ast(N)$ converges to $\Omega_m^{\ast}(\infty)$.

Another simple but key observation is
\begin{equation}\label{section5. limits of mu_S(N)}
\Omega_m^{\ast}(\infty)
= 1+\min\limits_{\Theta}\lambda_{ij}^\ast(\infty)
\geq 1+\min\limits_{\Theta}\lambda_{ij}(0).
\end{equation}
In fact, if $\mathcal{B}_m(\infty)$ is empty, then $\mathcal{A}_m(\infty)=\mathcal{C}_m(\infty)$ and $\Omega_m^{\ast}(\infty)=1$. Assume $\mathcal{B}_m(\infty)$ is nonempty. As in the proof of Proposition \ref{section5. prop for Bm}, we can choose a continuous chain $\{I_i\}_{i=1}^{a}$ from $\mathcal{B}_m(\infty)$ such that $\big|\big(\bigcup_{i=1}^a
I_i\big)\bigcap \Theta\big|$ is the largest over all continuous
chains in $\mathcal{B}_m(\infty)$. By Propositon \ref{section5. prop for Bm}, $\big(\bigcup_{i=1}^a I_i\big)\bigcap \Theta$ is a nonempty proper subset of $\Theta$. Hence we can choose
$i_0\in\big(\bigcup_{i=1}^{a}I_i\big)\bigcap\Theta$ and $j_0\in
\Theta$ but $j_0\notin\big(\bigcup_{i=1}^{a}I_i\big)\bigcap\Theta$.
Then it follows from Proposition \ref{section5. prop for Bm} that $\mu_J^{\ast}(\infty)=0$ for all $J\in \mathcal{F}_m$
satisfying $i_0,\;j_0\in J$ and $\Theta \nsubseteqq J$. Then (\ref{section5. (2) conditions
satisfied by parameters}) becomes $\sum_{J\supset
\Theta}\mu_J^{\ast}(\infty)= 1+\lambda_{ij}^{\ast}(\infty)$ with
$(i,j)=(i_0,j_0)$. This implies our claim.

Now the desired inequality (\ref{section5. object function of LI}) is easily verified. In fact, we have that
\begin{eqnarray*}
&&\l(\sum\limits_{\Theta}\alpha_{i,k+1}-n \r)
-\sum\limits_{J\in \mathcal{F}_m} \mu_J(0) \l((|J|-1)n-\sum\limits_{J}\alpha_{ij}\r)\\
&\leq&
\l(\sum\limits_{\Theta}\alpha_{i,k+1}-n \r)
-\sum\limits_{J\in \mathcal{F}_m} \mu_J^\ast(\infty) \l((|J|-1)n-\sum\limits_{J}\alpha_{ij}\r).
\end{eqnarray*}
By the assumption
(\ref{section5. condition for local integrability}), we observe that
$\sum_{\Theta}\alpha_{i,k+1}-n<(|J|-1)n-\sum_{J}\alpha_{ij}$ for all
$J\in\mathcal{F}_m$ satisfying $\Theta\subseteq J$ . Then it follows
from (\ref{section5. limits of mu_S(N)}) that
\begin{eqnarray*}
&&\l(\sum\limits_{\Theta}\alpha_{i,k+1}-n \r)
-\sum_{J\in\mathcal{F}_m}\mu_{J}^{\ast}(\infty)
\l((|J|-1)n-\sum\limits_{J}\alpha_{ij}\r)\\
&\leq&
\l(\sum\limits_{\Theta}\alpha_{i,k+1}-n \r)
-\sum_{\Theta\subset J\in\mathcal{F}_m}\mu_{J}^{\ast}(\infty)
\l((|J|-1)n-\sum\limits_{J}\alpha_{ij}\r) \\
&<&0.
\end{eqnarray*}
By Lemma \ref{section2. existence System of LI}, there exists a
solution $\{\delta_{ij}\}$ to the system $(V.1)$. By the induction
hypothesis for $k-1$, we have completed the proof in the first case
$\sum_{i=2}^m\alpha_{i,k+1}<n$.

Now we treat the second case $\sum_{i=2}^m\alpha_{i,k+1}=n$. By
Lemma \ref{Estimtates of Generalization of Riesz Potential}, it
suffices to show that the following integral is finite,
\begin{equation*}
\int_{(B_1(0))^k}\l(\prod\limits_S |x_i-x_j|^{-\alpha_{ij}}\r)
\l(\sum_\Theta |x_i-x_j|\r)^{n-\sum_\Theta \alpha_{i,k+1}}
\log \frac{2\sum_\Theta |x_i-x_j|}{\sum_{2\le i<j\le m}|x_i-x_j|} dV_S.
\end{equation*}
Since the existence of solutions to ${(V.1)}$ depends only on the conditions (\ref{section5. condition for local integrability}) and the assumption $\sum_{\Theta}\alpha_{i,k+1}>n$, we can find a solution
$\{\delta_{ij}:1\leq i<j\leq m\}$ to the system ${(V.1)}$. Let $\delta_{ij}=0$ for other $(i,j)$. It follows from $\sum_\Theta \alpha_{i,k+1}>n$ that there exist $i_0$
and $j_0$ with $1\leq i_0<j_0\leq m$ such that
$\delta_{i_0,j_0}>0$. Choose any pair $i_1$ and $j_1$ with $2\le i_1<j_1\le m$. For small $\varepsilon>0$, we put
\begin{equation}\label{section5. delta_ij with epsilon pertubation}
\overline{\alpha_{ij}}=\alpha_{ij}+\delta_{ij}
-{\textrm{\boldmath $\delta$}}_{i}^{i_0}{\textrm{\boldmath
$\delta$}}_{j}^{j_0}\varepsilon
+{\textrm{\boldmath $\delta$}}_{i}^{i_1}{\textrm{\boldmath
$\delta$}}_{j}^{j_1}\varepsilon
\end{equation}
for $1\leq i<j \leq k$, where ${\textrm{\boldmath
$\delta$}}_{s}^{t}$ denotes the Kronecker symbol. It is
easily verified that $\{\overline{\alpha_{ij}}\}$ satisfies the
integrability condition (\ref{section5. condition for local integrability}) with sufficiently small $\varepsilon$. By the induction hypothesis, the above integral converges.

The final case is $\sum_{i=2}^m\alpha_{i,k+1}>n$. The integral in
(\ref{section5. Selberg integral}) with
$\boldsymbol{L}_1(x_1,\cdots,x_m)$ in place of the integral with
respect to $x_{k+1}$ is bounded by a constant multiple of
\begin{equation}\label{section5. II Selberg integral}
\quad\int_{(B_1(0))^k}\l(\prod\limits_S |x_i-x_j|^{-\alpha_{ij}}\r)
\l(\sum_\Theta |x_i-x_j|\r)^{-\alpha_{1,k+1}}
\l(\int_{\mathbb{R}^n}\prod\limits_{i=2}^m |x_i-x_{k+1}|^{-\alpha_{i,k+1}}dx_{k+1}\r)dV_S.
\end{equation}
Likewise, our task is to distribute $\alpha_{1,k+1}$ into powers
$\{\alpha_{ij}:1\le i<j \le k\}$ appropriately. Then we will obtain
new parameters $\{\overline{\alpha_{ij}}:1\le i<j \le k+1\}$ with
$|\Theta|=m-1$. By the integrability condition (\ref{section5. condition for local integrability}), we shall solve the following system of linear inequalities,
\begin{equation*}
(V.2)
\begin{cases}
(i)\quad \;\delta_{ij}\geq 0, \quad i<j \quad in \quad \Theta;\\
(ii)\quad \sum\limits_\Theta \delta_{ij}=\alpha_{1,k+1};\\
(iii)\; \sum\limits_{J\cap\Theta} \delta_{ij}
<(|J|-1)n
-\sum\limits_{J}\l(\alpha_{ij}
-\boldsymbol{\delta}_i^1
\boldsymbol{\delta}_j^{k+1}\alpha_{1,k+1}\r),\quad J\in \overline{\mathcal{F}}_m;
\end{cases}
\end{equation*}
where $\overline{\mathcal{F}}_m$ consists of all subsets
$J\subseteq\{1,2,\cdots,k+1\}$ such that $|J\bigcap \Theta|\geq 2$. The existence of solutions can be proved similarly. Indeed, by Lemma
\ref{section2. existence System of LI}, for nonnegative $\lambda_{ij}$, $\theta_1$,
$\theta_2$ and $\mu_J$ with at least one $\mu_J>0$ for some $J$ in
the class $\overline{\mathcal{F}}_m$ satisfying
\begin{equation}\label{section5. equ of param 3d case}
\lambda_{ij}+(\theta_1-\theta_2)-\sum\limits_{J\ni i,j}\mu_J=0, ~~1\leq i<j \leq m,
\end{equation}
we must prove
\begin{equation*}
(\theta_1-\theta_2)\alpha_{1,k+1}
-\sum\limits_{J\in \overline{\mathcal{F}}_m} \mu_J \Big((|J|-1)n-
\sum\limits_{J}\big(\alpha_{ij}-\boldsymbol{\delta}_i^1
\boldsymbol{\delta}_j^{k+1}\alpha_{1,k+1}\big)\Big)<0.
\end{equation*}
We may assume $\theta_1-\theta_2=1$. To follow the argument in
the first case $\sum_{i=2}^m\alpha_{i,k+1}<n$, we shall make some remarks. The first observation is that the inequality (\ref{section5. cliam: object function is
increasing}) is still true with
$\alpha_{ij}-\textrm{\boldmath $\delta$}_{i}^1\textrm{\boldmath
$\delta$}_{j}^{k+1}\alpha_{1,k+1}$ in place of $\alpha_{ij}$. In other words, we have
$$\sum_{s=1}^{2}\sum_{J_s}\Big(\alpha_{ij}-\textrm{\boldmath $\delta$}_{i}^1\textrm{\boldmath $\delta$}_{j}^{k+1}\alpha_{1,k+1}\Big)
\leq
\sum_{J_1\cap J_2}\Big(\alpha_{ij}-\textrm{\boldmath $\delta$}_{i}^1\textrm{\boldmath $\delta$}_{j}^{k+1}\alpha_{1,k+1}\Big)
+\sum_{J_1\cup J_2}\Big(\alpha_{ij}-\textrm{\boldmath $\delta$}_{i}^1\textrm{\boldmath $\delta$}_{j}^{k+1}\alpha_{1,k+1}\Big)$$
for all subsets $J_1$ and $J_2$ of $\{1,2,\cdots,k+1\}$.

The process in (\ref{section5. Process Type I}) and
(\ref{section5. Process Type II}) is also applicable here with
$\mathcal{F}_m$ replaced by $\overline{\mathcal{F}}_m$ in Case I and Case II there. For arbitrary
$\{\mu_J(0)\geq 0:J\in\overline{\mathcal{F}}_m\}$ and
$\{\lambda_{ij}(0)\geq 0:1\leq i<j\leq m\}$ satisfying equations
(\ref{section5. equ of param 3d case}), we can use the above
argument to obtain
\begin{equation*}
\mu_J^\ast(\infty)=\lim_{N_t\rightarrow \infty}\mu_J^\ast(N_t),
\quad
\lambda_{ij}^\ast(\infty)=\lim_{N_t\rightarrow \infty}\lambda_{ij}^\ast(N_t),
\end{equation*}
where $\{\mu_J^\ast(N)\}$ is obtained by one of those processes such that $\Omega_m^\ast(N) =\sum_{\Theta\subset J\in
\overline{\mathcal{F}}_m}\mu_J^\ast(N)$ is the maximum over all
possible processes of $N$ steps. All symbols can be defined parallel to the first case. Similarly, we have
\begin{equation*}
\Omega_m^\ast(\infty)=\lim_{N\rightarrow \infty}\Omega_m^\ast(N)=1+\min_\Theta \lambda_{ij}^{\ast}(\infty).
\end{equation*}
Note that for any $J$ in $\overline{\mathcal{F}}_m$ satisfying $\Theta\subset J$,
\begin{equation*}
\alpha_{1,k+1}
-\Big((|J|-1)n-\sum_{J}\big(\alpha_{ij}-\textrm{\boldmath$\delta$}_{i}^1
\textrm{\boldmath $\delta$}_{j}^{k+1}\alpha_{1,k+1}\big)\Big)<0.
\end{equation*}
Actually, this inequality is obvious if $k+1\in J$ by the
integrability condition (\ref{section5. condition for local
integrability}). Assume $k+1\notin J$. Then the left side of this
inequality equals
\begin{eqnarray*}
\alpha_{1,k+1}+\sum_{J}\alpha_{ij}-\big(|J|-1\big)n
&=&\sum_{J\cup\{k+1\}}\alpha_{ij}-\sum_{i=2}^m\alpha_{i,k+1}-\big(|J|-1\big)n\\
&<& \sum_{J\cup\{k+1\}}\alpha_{ij}-|J|n<0
\end{eqnarray*}
since $\sum_{i=2}^m\alpha_{i,k+1}>n$.

Since $|x_i|\leq 1$ for $2\leq i\leq m$, it is clear that the integral of $\prod_{i=2}^m
|x_i-x_{k+1}|^{-\alpha_{i,k+1}}$ with respect to
$x_{k+1}$ over $\mathbb{R}^n$ is bounded by a constant multiple of the integral over $B_2(0)$. Thus the integral
in (\ref{section5. II Selberg integral}) is less than a constant
multiple of
\begin{equation*}
\int_{\boldsymbol{(B_1(0))^{k+1}}}\prod\limits_{\{1,2,\cdots,k+1\}}
|x_i-x_j|^{-\overline{\alpha_{ij}}}dx_1dx_2\cdots dx_{k+1},
\end{equation*}
where $\overline{\alpha_{ij}}=\big(\alpha_{ij}-\textrm{\boldmath
$\delta$}_{i}^1\textrm{\boldmath
$\delta$}_{j}^{k+1}\alpha_{1,k+1}\big)+\delta_{ij}$ and
$\{\delta_{ij}\}$ is a solution to the system $(V.2)$ with
$\delta_{ij}=0$ when $i\notin\Theta$ or $j\notin\Theta$. Thus we
have reduced the integral in (\ref{section5. II Selberg integral})
to an integral of form (\ref{section5. Selberg integral}) but with
$|\Theta|\leq m-1$. Since the $k+1$ fold integral in the
theorem converges for $|\Theta|=2$, we need at most $m-2$ steps to
reduce the case $|\Theta|=m$ to $|\Theta|=2$. Hence the
integral $(\ref{section5. II Selberg integral})$ is also finite.

Until now, we have obtained the desired conclusion for
$\boldsymbol{L}_{1}$. For other terms $\boldsymbol{L}_{i}$, the
treatment is the same as above. Therefore the proof of Theorem \ref{section5. nec-suff Selberg integral} is complete.
\end{proof}
\begin{remark}\label{section5 Remark 5.1}
The integrability criterion in the theorem is also true for Selberg
integrals on the sphere $S^n$. More precisely, for symmetric and
nonnegative exponents $\alpha_{ij}$, the following Selberg integral
$$\int_{(S^n)^{k+1}}\prod_{1\leq i<j\leq k+1}|\xi_i-\xi_j|^{-\alpha_{ij}}
d\sigma(\xi_1)d\sigma(\xi_2)\cdots d\sigma(\xi_{k+1})<\infty$$ if
and only if {\rm(\ref{section5. condition
for local integrability})} holds. Here $S^n$ is the unit sphere in
$\mathbb{R}^{n+1}$ and $d\sigma$ the induced Lebesgue measure on
$S^n$. In the conformally invariant situation, by using the
conformal equivalence of $S^n$ and $\mathbb{R}^n$, Beckner
{\rm\cite{beckner}} obtained explicitly the sharp constant of the
multilinear functional inequality {\rm(\ref{inequality of
multilinear fractional functionals})} in terms of the above Selberg
integrals. In {\rm\cite{grafakos1}}, Grafakos and Morpurgo
calculated a three fold integral of the above form when
$\alpha_{12}+\alpha_{13}+\alpha_{23}=n$. By using an analogue of
Theorem {\rm\ref{Estimtates of Generalization of Riesz Potential}}
on the sphere $S^n$ and the same argument as in the proof of Theorem {\rm\ref{section5. nec-suff Selberg integral}}, we can prove the above integrability criterion; see {\rm $\S$7} for a generalization of Theorem {\rm\ref{Estimtates of Generalization of Riesz Potential}} on $S^n$.
\end{remark}

Now we shall establish a useful estimate by which the $L^1$
estimate in Theorem \ref{section1. main theorem 2} follows
immediately from Theorem \ref{main theorem 1}.
\begin{theorem}\label{section5. estimate when p=infty}
Assume $\{\alpha_{ij}\}$ and $\{p_i\}$ with $p_{k+1}=\infty$ satisfy the assumptions in Theorem {\rm\ref{section1. main theorem 2}}. Then there exists a finite set $\Delta$. For each $t \in \Delta$ we have $\{\beta_{ij}(t)\geq 0: 1\leq i<j \leq k\}$ such that the datum
$\{p_i,\beta_{ij}(t)\}$ satisfies $(i)$, $(ii)$ and $(a)$ of ${{(iii)}}$ in Theorem {\rm\ref{main theorem 1}}, and the following estimate holds:
\begin{eqnarray}\label{section5. induction L infty estimate}
&&
\int_{\mathbb{R}^{n(k+1)}}{\prod\limits_{i=1}^{k}|f_{i}(x_{i})}|
\prod\limits_{1 \leq i<j\leq k+1}|x_i-x_j|^{-\alpha_{ij}}dx_1dx_2\cdots dx_{k+1}\nonumber\\
&\leq &
C \sum\limits_{t\in\Delta}
\int_{\mathbb{R}^{nk}}
{\prod\limits_{i=1}^{k}|f_{i}(x_{i})|}
\prod\limits_{1 \leq i<j\leq k}|x_i-x_j|^{-\beta_{ij}(t)}dx_1dx_2\cdots dx_{k},
\end{eqnarray}
where the constant $C$ depends only on $\alpha_{1,k+1}$,..., $\alpha_{k,k+1}$ and the dimension $n$.
\end{theorem}

\begin{proof}
Let $\Theta$ consist of those $i\in S=\{1,2,\cdots,k\}$ such that $\alpha_{i,k+1}>0$.
By assumptions in Theorem \ref{section1. main theorem 2} for $I=S$,
we have
\begin{equation*}
\sum_\Theta \alpha_{i,k+1}=\frac n{p_{k+1}} +\sum_S\alpha_{i,k+1}>n.
\end{equation*}
To reduce mapping properties of the $k+1-$linear functional
$\Lambda$ to that of a $k-$linear one, we need estimate the
following integral with respect to $x_{k+1}$,
\begin{equation*}
I(x_i:i\in \Theta)=\int_{\mathbb{R}^n}
\prod\limits_{\Theta}|x_i-x_{k+1}|^{-\alpha_{i,k+1}}dx_{k+1}.
\end{equation*}
By the symmetry of parameters, we may assume
$\Theta=\{1,2,\cdots,m\}$ with $m\geq 2$. By Lemma \ref{Estimtates
of Generalization of Riesz Potential}, we claim that (\ref{section5. induction L infty estimate}) is still true if $I(x_i:i\in \Theta)$ is
replaced by each term $\boldsymbol{L}_{i}$. We first prove this
statement for $m=2$. Observe that the integral $I(x_1,x_2)$ equals a constant multiple of
$|x_1-x_2|^{n-\alpha_{1,k+1}-\alpha_{2,k+1}}$. Put
$
\beta_{ij}
=\alpha_{ij}+{\textrm{\boldmath $\delta$}}_i^1{\textrm{\boldmath $\delta$}}_j^2(\alpha_{1,k+1}+\alpha_{2,k+1}-n)
$
for $1\leq i<j \leq k$. Then we see that $\{\beta_{ij}\}$
and $\{p_i\}$ satisfy (i), (ii) and (a) of (iii) in Theorem \ref{main theorem 1}. Thus our claim is true in the case $m=2$.

For $3\leq m \leq k$, we claim that there exist finite families $\{\beta_{ij}(t)\}$ such that
\begin{eqnarray}\label{section5. II induction L infty estimate}
&&
\int_{\mathbb{R}^{nk}}\prod\limits_{i=1}^{k}|f_{i}(x_{i})|
 \l({\prod\limits_{S}|x_i-x_j|^{-\alpha_{ij}}}\r)
\l(\sum\limits_{\Theta}\boldsymbol{L}_{i}(x_1,\cdots,x_m)\r)dx_1dx_2\cdots dx_{k} \nonumber\\
&\leq &
C \sum\limits_{t\in \Delta} \int_{\mathbb{R}^{nk}}\prod\limits_{i=1}^{k}|f_{i}(x_{i})|
 {\prod\limits_{S}|x_i-x_j|^{-\beta_{ij}(t)}}dx_1dx_2\cdots dx_{k},
\end{eqnarray}
where $\Delta$ is a finite set and $\{p_i,\beta_{ij}(t)\}$ satisfy conditions $(i)$, $(ii)$ and $(a)$ of $(iii)$ in Theorem $\ref{main theorem 1}$. Since arguments for different $\boldsymbol{L}_i$'s are similar, we need only establish the desired estimate concerning $\boldsymbol{L}_{1}$. As shown in Lemma \ref{Estimtates of Generalization of Riesz Potential}, there are three
possible cases for $\boldsymbol{L}_{1}$. The treatment of two
previous cases $\sum_{i=2}^m\alpha_{i,k+1}\leq n$ will be reduced to
the following system of linear inequalities:
\begin{equation}
(V.3)
\begin{cases}
(i)\quad\; \delta_{ij}\geq 0, \quad 1\le i <j \le m;\nonumber\\
(ii) \quad \sum\limits_{\Theta}\delta_{ij}=\sum\limits_{\Theta}\alpha_{i,k+1}-n;
\nonumber\\
(iii) \;\sum\limits_{J\cap\Theta}\delta_{ij}<
\l((|J|-1)n-\sum\limits_{J}\alpha_{ij}\r)\wedge
\l(|J|n-\sum\limits_{J}\alpha_{ij}-\sum\limits_J \frac{n}{p_i}\r)
\quad {\textrm{for $J\in \mathcal{F}_m$ and $J\neq S$}}; \nonumber\\
(iv) \; \sum\limits_{S\cap\Theta}\delta_{ij}\leq
\l((k-1)n-\sum\limits_{S}\alpha_{ij}\r)\wedge
\l(kn-\sum\limits_{S}\alpha_{ij}-\sum\limits_S \frac{n}{p_i}\r). \nonumber
\end{cases}
\end{equation}
Recall that $\mathcal{F}_m$ is the class of all subsets $J$ of $S$ which contains at least two members in $\{1,2,\cdots,m\}$. We shall point out that $(iv)$ is an equality. By $(ii)$ in the system
$(V.3)$, this observation follows from
$\sum_{S\cup\{k+1\}}\alpha_{ij}<kn$ and the condition $(i)$ in
Theorem \ref{main theorem 1}. For this reason, the system $(V.3)$ is equivalent to the system of inequalities (i), (ii) and (iii). We add $(iv)$ for convenience of notations. As in the previous theorem, we shall also use Lemma \ref{section2. existence System of LI} to prove existence of a solution. Assume $\lambda_{ij}$, $\theta_1$, $\theta_2$ and $\mu_J$ are nonnegative numbers satisfying
\begin{equation}\label{section5. (1)' cond by parameters}
\lambda_{ij}-\sum\limits_{J\ni i,j}\mu_J
+(\theta_1-\theta_2)=0,~~1\leq i<j\leq m.
\end{equation}
Here there is at least one $\mu_J>0$ for some proper subset $J$ of
$S$ in the class $\mathcal{F}_m$. Under these conditions, we have to prove
\begin{eqnarray}\label{section5. object function II}
&&(\theta_1-\theta_2)\l(\sum\limits_{\Theta}\alpha_{i,k+1}-n \r)
\nonumber\\
&<&\sum\limits_{J\in {\mathcal{F}}_m} \mu_J \l((|J|-1)n-
\sum\limits_{J}\alpha_{ij}
\r)\wedge
\l(|J|n-
\sum\limits_{J}\alpha_{ij}
-\sum\limits_J\frac{n}{p_i}\r).
\end{eqnarray}
This inequality is obviously true if $\theta_1-\theta_2\leq 0$ since
$\mu_J>0$ for some proper subset $J\in \mathcal{F}_m$ of $S$. Thus
it suffices to show the inequality for $\theta_1-\theta_2> 0$. By
dilation, we may assume $\theta_1-\theta_2=1$. There is an important
observation like (\ref{section5. cliam: object function is
increasing}) given as follows:
\begin{eqnarray}\label{section5. cliam II: object function is
increasing}
&&\Big((|J_1\cap J_2|-1)n
-\sum\limits_{J_1\cap J_2}\alpha_{ij}\Big)
\wedge
\Big(|J_1\cap J_2|n
-\sum\limits_{J_1\cap J_2}\alpha_{ij}-\sum\limits_{J_1\cap J_2}
\frac{n}{p_i}\Big)\nonumber\\
&&+
\Big((|J_1\cup J_2|-1)n
-\sum\limits_{J_1\cup J_2}\alpha_{ij}\Big)
\wedge
\Big(|J_1\cup J_2|n
-\sum\limits_{J_1\cup J_2}\alpha_{ij}-\sum\limits_{J_1\cup J_2}
\frac{n}{p_i}\Big)\nonumber\\
&\leq&
\sum\limits_{s=1}^2\l\{\Big((|J_s|-1)n
-\sum\limits_{J_s}\alpha_{ij}\Big)
\wedge
\Big(|J_s|n
-\sum\limits_{J_s}\alpha_{ij}-\sum\limits_{J_s}\frac{n}{p_i}
\Big)\r\}
\end{eqnarray}
for all subsets $J_1$ and $J_2$ of $S$. To show this inequality, we have to verify it in all possible cases.

Case (i): $\sum_{J_1\cup J_2}n/p_i\leq n$.

The above inequality is just (\ref{section5. cliam: object function is increasing}). We also note that (\ref{section5. cliam: object function is increasing}) becomes an equality if and only if
\begin{equation}\label{section5. for which i,j a_ij=0}
\alpha_{ij}=0,\quad (i,j)\in \{(s,t):s<t,s,t\in J_1\cup J_2\}-
\bigcup\limits_{i=1}^2\{(s,t):s<t,s,t \in J_i\}.
\end{equation}

Case (ii): $\sum_{J_1\cap J_2}n/p_i\geq n$.

It is true that
\begin{eqnarray*}
&&\Big(|J_1\cap J_2|n
-\sum\limits_{J_1\cap J_2}\alpha_{ij}-\sum\limits_{J_1\cap J_2}
\frac{n}{p_i}\Big)
+\Big(|J_1\cup J_2|n
-\sum\limits_{J_1\cup J_2}\alpha_{ij}-\sum\limits_{J_1\cup J_2}
\frac{n}{p_i}\Big)\nonumber\\
&\leq &\sum\limits_{s=1}^2\Big(|J_s|n
-\sum\limits_{J_s}\alpha_{ij}-\sum\limits_{J_s}\frac{n}{p_i}\Big),
\end{eqnarray*}
where the equality is valid if and only if
(\ref{section5. for which i,j a_ij=0}) holds.

Case (iii): $\sum_{J_1\cap J_2}n/p_i\leq n$ and $\sum_{J_1\cup J_2}n/p_i\geq n$, but $\sum_{J_1}n/p_i\leq n$ and $\sum_{J_2}n/p_i\geq n$.

We have
\begin{eqnarray*}
&&\Big((|J_1\cap J_2|-1)n -\sum\limits_{J_1\cap J_2}\alpha_{ij}\Big)
+\Big(|J_1\cup J_2|n -\sum\limits_{J_1\cup J_2}\alpha_{ij}
-\sum\limits_{J_1\cup
J_2}\frac{n}{p_i}\Big)\nonumber\\
&\leq& \Big((|J_1|-1)n -\sum\limits_{J_1}\alpha_{ij}\Big)
+\Big(|J_2|n -\sum\limits_{J_2}\alpha_{ij}
-\sum\limits_{J_2}\frac{n}{p_i}\Big)\nonumber.
\end{eqnarray*}
The equality is true if and only if
(\ref{section5. for which i,j a_ij=0}) holds and $p_{i}=\infty$ for
$i$ in $J_1$ but not in $J_2$. The remaining cases of (\ref{section5.
cliam II: object function is increasing}) can be proved similarly.

Case (iv): $\sum_{J_1\cap J_2}n/p_i\leq n$ and $\sum_{J_1\cup J_2}n/p_i\geq n$, but $\sum_{J_1}n/p_i\geq n$ and $\sum_{J_2}n/p_i\leq n$.

As in Case (iii), it is clear that
\begin{eqnarray*}
&&\Big((|J_1\cap J_2|-1)n -\sum\limits_{J_1\cap J_2}\alpha_{ij}\Big)
+\Big(|J_1\cup J_2|n -\sum\limits_{J_1\cup J_2}\alpha_{ij}
-\sum\limits_{J_1\cup
J_2}\frac{n}{p_i}\Big)\nonumber\\
&\leq& \Big(|J_1|n -\sum\limits_{J_1}\alpha_{ij}-\sum\limits_{J_1}\frac{n}{p_i}\Big)
+\Big((|J_2|-1)n -\sum\limits_{J_2}\alpha_{ij}\Big)\nonumber.
\end{eqnarray*}

Combining above results, we see that if (\ref{section5. cliam II:
object function is increasing}) becomes an equality then we must have (\ref{section5. for which i,j a_ij=0}). This fact will be used
later.

To show the inequality (\ref{section5. object function II}), we also need the recursion in (\ref{section5. Process Type I}) and
(\ref{section5. Process Type II}). By the inequality (\ref{section5. cliam II:
object function is increasing}), the objective function also increases as the recursion continues. In fact, we have an analogue of (\ref{section5.
object function is increasing}) in the present situation,
\begin{eqnarray}\label{section5.II obj-func is increasing}
& &
-\sum_{J\in \mathcal{F}_m}
\mu_J(N-1)\left((|J|-1)n-\sum\limits_{J}\alpha_{ij}\right)
\wedge
\left(|J|n-\sum\limits_{J}\alpha_{ij}-\sum\limits_J\frac{n}{p_i}
\right) \nonumber\\
&\leq&
-\sum_{J\in \mathcal{F}_m}
\mu_J(N)\left((|J|-1)n-\sum\limits_{J}\alpha_{ij}\right)
\wedge
\left(|J|n-\sum_{J}\alpha_{ij}-\sum\limits_J\frac{n}{p_i}
\right)
\end{eqnarray}
where $\mu_J(N)$ are obtained by the recursion in (\ref{section5. Process Type I}) and (\ref{section5. Process Type II}).

We shall follow some notations in the proof of Theorem
\ref{section5. nec-suff Selberg integral}. Let
$\{\mu_J^\ast(N)\}$ be obtained by one of processes consisting of
$N$ steps such that $\Omega_m^\ast(N)$ achieves its maximum. By a similar argument as in the proof of Theorem
\ref{section5. nec-suff Selberg integral}, we can
obtain $\{\mu_J^\ast(N)\}$ and $\{\lambda_{ij}^\ast(N)\}$ for
initial data $\{{\mu}_J(0)\}$ and $\{{\lambda}_{ij}(0)\}$ satisfying (\ref{section5. (1)' cond by parameters}) with $\theta_1-\theta_2=1$. By passing to a subsequence $\{N_t\}$, we also use $\mu_J^\ast(\infty)$ and
$\lambda_{ij}^\ast(\infty)$ to denote the limits of
$\mu_J^\ast(N_t)$ and $\lambda_{ij}^\ast(N_t)$, respectively. It is clear that (\ref{section5. limits of mu_S(N)}) is also true and hence $\Omega_m^\ast(\infty)\geq 1$. The following argument is somewhat different depending on whether $\mu_S^\ast(\infty)=1$ and $\mu_J^\ast(\infty)=0$ for all proper subsets $J$ of $S$ in $\mathcal{F}_m$. Observe that
 $$\sum\limits_{\Theta}\alpha_{i,k+1}-n=
\l((|S|-1)n-\sum\limits_{S}\alpha_{ij}\r)\wedge
\l(|S|n-\sum\limits_{S}\alpha_{ij}-\sum\limits_S\frac{n}{p_i}\r)$$
and the right side equals
$|S|n-\sum_{S}\alpha_{ij}-\sum_S{n}/{p_i}$. And we also have
 $$\sum\limits_{\Theta}\alpha_{i,k+1}-n<
\l((|J|-1)n-\sum\limits_{J}\alpha_{ij}\r)\wedge
\l(|J|n-\sum\limits_{J}\alpha_{ij}-\sum\limits_J\frac{n}{p_i}\r)$$
for all proper subsets $J$ of $S$ satisfying $\Theta\subset J$. By
(\ref{section5.II obj-func is increasing}), we have
\begin{eqnarray*}
&&-\sum\limits_{J\in \mathcal{F}_m}
\mu_J(0)\Big((|J|-1)n-\sum\limits_{J}\alpha_{ij}\Big)
\wedge
\Big(|J|n-\sum\limits_{J}\alpha_{ij}-\sum\limits_J\frac{n}{p_i}
\Big)\\
&\leq&
-\sum\limits_{J\in \mathcal{F}_m}
\mu_J^\ast(\infty)\Big((|J|-1)n-\sum\limits_{J}\alpha_{ij}\Big)
\wedge
\Big(|J|n-\sum\limits_{J}\alpha_{ij}-\sum\limits_J\frac{n}{p_i}
\Big).
\end{eqnarray*}
Let $\boldsymbol{H}(\lambda_{ij},\mu_J)$ be the objective function
\begin{equation*}\label{section5. def of the function H}
\left(\sum_{\Theta}\alpha_{i,k+1}-n \right)
-\sum_{J\in {\mathcal{F}}_m} \mu_J \left((|J|-1)n-
\sum_{J}\alpha_{ij}
\right)\wedge\left(|J|n-\sum_{J}\alpha_{ij}-\sum\limits_J\frac{n}{p_i}\right).
\end{equation*}
Recall that we have $\Omega_m^{\ast}(\infty)
= 1+\min\limits_{\Theta}\lambda_{ij}^\ast(\infty)$.

Now we divide the proof into three cases.

Case (i): $\Omega_m^\ast(\infty)>1$.

The inequality (\ref{section5. object function II}) follows immediately. Since
$\boldsymbol{H}(\lambda_{ij}(0),\mu_J(0))
\leq \boldsymbol{H}(\lambda_{ij}^\ast(\infty),\mu_J^\ast(\infty))$,
$$\boldsymbol{H}(\lambda_{ij}^{\ast}(\infty),\mu_J^{\ast}(\infty))
< \Omega_m^\ast(\infty)\Big(\sum_\Theta\alpha_{i,k+1}-n\Big)
-\sum_{J\in
C_m(\infty)}\mu_J^{\ast}(\infty)\Big((|J|-1)n-\sum_J\alpha_{ij}\Big)
\leq 0,$$
where $C_m(\infty)=\{J\in\mathcal{F}_m:\Theta\subseteq J,~\mu_J^{\ast}(\infty)>0\}$.

Case (ii): $\Omega_m^\ast(\infty)=1$ and $\mu_J^\ast(\infty)>0$ for some proper subset $J$ of $S$.

As in Case (i), we also have $\boldsymbol{H}(\lambda_{ij}(0),\mu_J(0))<0$.

Before we consider Case (iii), we point out that Case (i) and (ii) contain a special case, i.e., $\Omega_m(0)<1$ and all $\lambda_{ij}(0)=0$. In fact, this implies that there exists a positive
$\mu_J(0)$ for some $J\in\mathcal{F}_m$ satisfying $\Theta\nsubseteqq J$. Now we claim that there exist $J_1$ and $J_2$ in $\mathcal{B}_m(0)$ with nonempty intersection such that
\begin{equation}
\quad \quad J_1\bigcap \Theta \nsubseteq J_2\bigcap \Theta \quad
\textrm{and} \quad J_2\bigcap \Theta \nsubseteq J_1\bigcap \Theta.
\end{equation}
Assume the converse. Then either $J_1\bigcap \Theta \subset J_2\bigcap \Theta$
or $J_2\bigcap \Theta \subset J_1\bigcap \Theta$ is true for all $J_1$ and $J_2$ in $\mathcal{B}_m(0)$ with $J_1\bigcap J_2\neq \emptyset$. Choose a $J_0\in \mathcal{B}_m(0)$ such that
\begin{equation*}
\Big|J_0\bigcap \Theta \Big|=\max_{J\in \mathcal{B}_m(0)}\Big|J \bigcap \Theta\Big|.
\end{equation*}
Then $J\bigcap\Theta\subset J_0\bigcap\Theta$ for all $J\in \mathcal{B}_m(0)$. Since
$J_0\bigcap \Theta$ is a proper subset of $\Theta$, we may choose
$i_0\in\Theta$ but $i_0\notin J_0\bigcap\Theta$. Choose a $j_0\in
J_0\bigcap\Theta$ arbitrarily. Then by the choice of $J_0$, we see
that all $\mu_J(0)=0$ for all $J\in\mathcal{F}_m$ satisfying
$i_0,j_0\in J$ and $\Theta\nsubseteqq J$. Then we have $\Omega_m(0)=1$  by the equation $\sum_{J\ni i,j}\mu_J(0)=1$ with $(i,j)=(i_0,j_0)$ or
$(i,j)=(j_0,i_0)$. As a consequence, $\mathcal{B}_m(0)$ becomes an
empty class which contradicts our assumption $\Omega_m(0)<1$. Recall that we have assumed $\Theta=\{1,2\cdots,m\}$. Choose $J_1,J_2\in
\mathcal{B}_m(0)$ with $J_1\bigcap J_2\neq \emptyset$ satisfy
\begin{eqnarray*}
&& J_1\bigcap \{1,2,\cdots,m\}
\nsubseteq
J_2\bigcap \{1,2,\cdots,m\}\\
&&
J_2\bigcap \{1,2,\cdots,m\}
\nsubseteq
J_1\bigcap \{1,2,\cdots,m\}.
\end{eqnarray*}
Then we apply the recursion (\ref{section5. Process Type I}) or
(\ref{section5. Process Type II}) to $J_1$ and $J_2$ and then obtain
$\{\mu_J(1)\}$ and $\{\lambda_{ij}(1)\}$. By this process, we will
obtain at least one $\lambda_{ij}(1)>0$. Indeed, we may choose
\begin{eqnarray*}
&& i\in J_1\bigcap \{1,2,\cdots,m\}\quad but \quad
i\notin J_2\bigcap \{1,2,\cdots,m\}\\
&&
j\in J_2\bigcap \{1,2,\cdots,m\}\quad but \quad
j\notin J_1\bigcap \{1,2,\cdots,m\}.
\end{eqnarray*}
It follows from $J_1,J_2\in \mathcal{A}_m(0)$ that $\lambda_{ij}(1)=\mu_{J_1}(0) \wedge \mu_{J_2}(0)>0$. Thus the datum $\{\mu_J(1),\lambda_{ij}(1)\}$ has been considered in Case (i) and (ii). Since the objective function
$\boldsymbol{H}(\lambda_{ij}(N),\mu_J(N))$ increases as $N$, we see that
\begin{equation*}
\boldsymbol{H}\Big(\lambda_{ij}(0)=0,\mu_J(0) \Big) \leq
\boldsymbol{H}\Big(\lambda_{ij}(1),\mu_J(1) \Big)<0.
\end{equation*}

Case (iii): $\Omega_m^\ast(\infty)=1$ and $\mu_J^\ast(\infty)=0$ for all proper subsets $J\in \mathcal{F}_m$ of $S$.

By previous analysis, it suffices to consider $\Omega_m(0)=1$.
Recall $\Omega_m(0)=\sum_{J\in \mathcal{C}_m(0)}\mu_J(0)$. By equation
$\sum_{J\ni i,j}\mu_J(0)=1$, we have $\mu_J(0)=0$ for
all $J\in\mathcal{F}_m$ satisfying $\Theta\nsubseteqq J$. The choice of $\{\mu_J(0)\}$ implies the existence of $\mu_J(0)>0$ for some proper subset $J$ of $S$ in the class $C_m(0)$. Thus
$$\boldsymbol{H}\Big(\;\lambda_{ij}(0)=0,\;\mu_J(0)\Big)<0.$$

Combining above results, we conclude that there exists at least one solution to the system $(V.3)$.

Let $\{\delta_{ij}\}$ be a solution to the system $(V.3)$. Here we
also put $\delta_{ij}=0$ if either $i$ or $j$ does not lie in $\Theta$. If
$\sum_{i=2}^m\alpha_{i,k+1}<n$, then our claim (\ref{section5. II
induction L infty estimate}) is true by setting
$\beta_{ij}=\alpha_{ij}+\delta_{ij}$. In the case
$\sum_{i=2}^k\alpha_{i,k+1}=n$, the treatment is similar as
(\ref{section5. delta_ij with epsilon pertubation}) and we omit the details here.

Now we turn our attention to the final case
$\sum_{i=2}^m\alpha_{i,k+1}>n$. Then $\boldsymbol{L}_1$ equals
\begin{equation*}
\l(\sum\limits_\Theta|x_i-x_j|\r)^{-\alpha_{1,k+1}}
\int_{\bR^n}\prod\limits_{i=2}^m
\big|x_i-x_{k+1}\big|^{-\alpha_{i,k+1}}
dx_{k+1}.
\end{equation*}
In this case, we shall adapt the treatment of the system $(V.2)$. The corresponding system of linear inequalities is given as follows:
\begin{equation}
(V.4)
\begin{cases}
(i)\qquad \delta_{ij}\geq 0, \quad 1\le i <j \le m;\nonumber\\
(ii)\quad\;\,\sum\limits_{\Theta}\delta_{ij}=\alpha_{1,k+1};
\nonumber\\
(iii)\quad\sum\limits_{J\cap \Theta}\delta_{ij}<
\Big((|J|-1)n-\boldsymbol{B}_J \Big)
\wedge
\l(|J|n-\boldsymbol{B}_J-\sum\limits_J \frac{n}{p_i}\r),
\quad
J\in \overline{\mathcal{F}}_m,\; J\neq S\cup \{k+1\}; \nonumber\\
(iv) \sum\limits_{S\cup\{k+1\}\cap\Theta}\delta_{ij}\leq
\l(kn-\boldsymbol{B}_{S\cup\{k+1\}}\r)\wedge
\l((k+1)n-\boldsymbol{B}_{S\cup\{k+1\}}
-\sum\limits_{S\cup\{k+1\}} \frac{n}{p_i}\r); \nonumber
\end{cases}
\end{equation}
where $\overline{\mathcal{F}}_m$ is the class of all subsets
$J\subset\{1,2,\cdots,k+1\}$ with $|J\cap\{1,2,\cdots,m\}|\geq 2$ and $\boldsymbol{B}_J$ is given by, for each $J\in\overline{\mathcal{F}}_m$,
$$\boldsymbol{B}_J=\sum\limits_{J}\l(\alpha_{ij}
-\boldsymbol{\delta}_i^1
\boldsymbol{\delta}_j^{k+1}\alpha_{1,k+1}\r).$$

The existence of a solution to $(V.4)$ can be proved similarly.
The argument can be outlined as follows. For nonnegative
$\lambda_{ij}$, $\mu_J$, $\theta_1$ and $\theta_2$ satisfying
\begin{equation}
\lambda_{ij}-\sum\limits_{J\ni i,j}\mu_J +(\theta_1-\theta_2)=0
,\quad 1\leq i<j\leq m,
\end{equation}
where there exists one $\mu_J>0$ for some proper subset $J$ of
$S\cup \{k+1\}$ in the class $\overline{\mathcal{F}}_m$, it is
enough to show that
\begin{equation}\label{section5. L infty estimate for sum>n}
(\theta_1-\theta_2)\alpha_{1,k+1}-
\sum\limits_{J\in \overline{\mathcal{F}}_m}
\mu_J\Big((|J|-1)n-\boldsymbol{B}_J\Big)
\wedge
\Big(|J|n-\boldsymbol{B}_{J}-\sum\limits_J\frac{n}{p_i}\Big)<0.
\end{equation}
The above inequality is obvious for $\theta_1-\theta_2\leq 0$. By
scaling, we may assume $\theta_1-\theta_2=1$. In this setting, the
argument is the same as the proof of existence of solutions to
the system $(V.2)$ and $(V.3)$. We omit the details here.

Put $\delta_{ij}=0$ for $i\notin\Theta$ or $j\notin\Theta$. Let
$\{\delta_{ij}:1\leq i<j \leq m\}$ be a solution to the system
$(V.4)$. Set
$$\overline{\alpha_{ij}}
=\big(\alpha_{ij}-\textrm{\boldmath $\delta$}_{i}^1\textrm{\boldmath
$\delta$}_{j}^{k+1}\alpha_{1,k+1}\big) +\delta_{ij}$$ for $1\leq i<j
\leq k+1$. Then the datum $\{\overline{\alpha_{ij}},p_i\}$ still satisfies assumptions in the theorem. Moreover, the left side integral in (\ref{section5. II
induction L infty estimate}) with $\boldsymbol{L}_1$ in place of
$\sum_{\Theta}\boldsymbol{L}_i$ is bounded by a constant multiple of
$$
\int_{\mathbb{R}^{n(k+1)}}\prod\limits_{i=1}^{k}|f_{i}(x_{i})|
 {\prod\limits_{1\leq i<j \leq k+1} |x_i-x_j|^{-\overline{\alpha_{ij}}}}dx_1dx_2\cdots dx_{k+1}
$$
which reduces $|\Theta|=m$ to $|\Theta|=m-1$.

Repeating the above argument finite times, we will obtain the
desired inequality in the theorem.
\end{proof}

\section{Proof of Theorem \ref{section1. main theorem 2}}
In this section, we shall first give another proof of Theorem \ref{section5. nec-suff Selberg integral}, i.e., the integral (\ref{multilinear functional}) is absolutely convergent for arbitrary $f_i\in C_0^\infty$ if and only if $\sum_J\alpha_{ij}<(|J|-1)n$ for all subsets $J\subseteq\{1,2,\cdots,k+1\}$ with $|J|\geq 2$. Of course, the existence of such nonnegative numbers $\{\alpha_{ij}\}$ is obvious. For any given $\{\alpha_{ij}\}$ satisfying this integrability condition, a natural question arises whether there is a set of positive numbers $\{p_i\}$
such that $\{\alpha_{ij}\}$ and $\{p_i\}$ satisfy conditions $(i)$,
$(ii)$ and $(iii)$ in Theorem \ref{main theorem 1}. If this is true, we will obtain the local integrability of the integral (\ref{multilinear
functional}). In fact, the answer is affirmative except some trivial cases. For example, the simplest case in which all $\alpha_{ij}=0$ should be ruled out since the boundedness of $\Lambda$ is valid only if all $p_i=1$. Moreover, we may further
assume that for any given $i$ not all $\alpha_{ij}$ are zero with $j$ ranging over $\{1,2,\cdots,k+1\}$. More precisely, the existence of $\{p_i\}$ can be stated as follows.

\begin{theorem}\label{section 6. exi of infty pairs p}
Assume $\alpha_{ij}\geq 0$ satisfy the system of inequalities $(ii)$ in
Theorem {\rm\ref{main theorem 1}} and
\begin{equation}\label{section 6. Condition for undecomposable of 1,k+1}
\sum\limits_{i\in J}\sum\limits_{j\in J^c}\alpha_{ij}>0
\end{equation}
for any nonempty proper subset $J$ of $\{1,2,\cdots,k+1\}$. Then
there exist infinitely many $\{p_i\}$ such that
\begin{equation*}
(VI.1)
\begin{cases}
(i)\quad  1<p_i<\infty,\quad 1\leq i \leq k+1;\nonumber\\
(ii)\; \sum\limits_{i=1}^{k+1}\frac 1{p_i}
+\sum\limits_{1\leq i<j \leq k+1}\frac{\alpha_{ij}}{n}=k+1;\nonumber\\
(iii)\; \sum\limits_J\frac 1{p_i}
+\sum\limits_J\frac {\alpha_{ij}}{n}<|J|,\quad J\neq \emptyset,
\quad J\subsetneqq \{1,2,\cdots,k+1\}.\nonumber
\end{cases}
\end{equation*}
\end{theorem}
\begin{proof}
We begin with discussing the necessity of the additional assumption
(\ref{section 6. Condition for undecomposable of 1,k+1}) which does
not lose generality. This assumption is only necessary to ensure the existence of a solution to the system $(VI.1)$. For general
$\{\alpha_{ij}\}$ satisfying $\sum_{J}\alpha_{ij}<(|J|-1)n$ for all
subsets $J$ with $|J|\geq 2$, we need only assume that for each $i$
there exists some $\alpha_{ij}>0$ with $j\in\{1,2,\cdots,k+1\}$. By this weaker assumption, we can divide $\{1,2,\cdots,k+1\}$ into disjoint subsets $J_1$, $J_2$, $\cdots$, $J_l$ with $|J_i|\geq 2$ such that $\alpha_{ij}$ is positive with $i<j$ only if
\begin{equation*}
(i,j)\in \bigcup_{u=1}^l\{(s,t): s<t,\; s,t\in J_u\}.
\end{equation*}
If each $J_i$ can not be decomposed further as above, i.e.,
$\sum_{u\in I}\sum_{v\in J_i-I}\alpha_{uv}>0$ for any nonempty
proper subset $I$ of $J_i$, then we can reduce matters to $l$
multilinear functionals $\{\Lambda_{J_i}\}$ of form
(\ref{multilinear functional}). Therefore we may assume that
$\{1,2,\cdots,k+1\}$ cannot be decomposed as above. This is
equivalent to $\sum_{i\in J}\sum_{j\in J^c}\alpha_{ij}>0 $ for all
nonempty proper subsets $J$ of $\{1,2,\cdots,k+1\}$. Thus the
additional assumption (\ref{section 6. Condition for undecomposable
of 1,k+1}) does not lose generality.

Now we turn to verify the existence of $\{p_i\}$. Define
$\delta_i=1/p_i$. As required in System $(VI.1)$, $\{\delta_i\}$ should satisfy the following system of linear inequalities:
\begin{equation*}
(VI.2)
\begin{cases}
(i)\quad\; \sum\limits_{i=1}^{k+1}\delta_i
=k+1-\sum\limits_{1\leq i<j \leq k+1}\frac{\alpha_{ij}}{n};\\
(ii)\quad\; \delta_i>0,\quad 1\leq i \leq k+1;\\
(iii)\quad \sum\limits_J\delta_i
<|J|-\sum\limits_J\frac {\alpha_{ij}}{n},\quad J\neq \emptyset,
\quad J\subsetneqq \{1,2,\cdots,k+1\};\\
(iv)\quad \sum\limits_{\{1,2,\cdots,k+1\}}\delta_i
\leq k+1-\sum\limits_{\{1,2,\cdots,k+1\}}\frac {\alpha_{ij}}{n}.
\end{cases}
\end{equation*}
If we take $J=\{i\}$, then $(iii)$ implies $\delta_i<1$. Our proof of the existence of $\{p_i\}$ needs Lemma \ref{section2. existence System of LI}. It is clear that the system consisting of $(i)$ and $(iv)$ does have infinitely many solutions. To apply Lemma \ref{section2. existence System of LI}, we assume that
$\theta_1$, $\theta_2$, $\pi_i$,  $\mu_J$ are nonnegative numbers
satisfying
\begin{equation}\label{section 6. I cond by parameters}
(\theta_1-\theta_2)+\pi_i-\sum\limits_{J\ni i}\mu_J=0,
\quad 1\leq i \leq k+1,
\end{equation}
where either $\pi_i>0$ for some $i$ or $\mu_J>0$ for some nonempty proper subset $J$. By Lemma \ref{section2. existence System of
LI}, it is enough to show
\begin{equation}\label{section 6. objective function 1}
(\theta_1-\theta_2)\Big(k+1
-\sum\limits_{1\leq i<j \leq k+1}\frac{\alpha_{ij}}{n}\Big)
-\sum\limits_J\mu_J\Big(|J|-\sum\limits_{J}\frac{\alpha_{ij}}{n}\Big)
<0.
\end{equation}
If $\theta_1-\theta_2\leq 0$, then the above inequality follows immediately by our choice of parameters. Now we consider the case $\theta_1-\theta_2> 0$. By scaling, we may assume $\theta_1-\theta_2=1$. For any initial data $\{\mu_J(0)\}$
and $\{\pi_i(0)\}$ satisfying
$$\sum_{J\ni i}\mu_J=1+\pi_i,$$
we apply a process to obtain new data $\{\mu_J(N)\}$
and $\{\pi_i(N)\}$ for which the objective function in (\ref{section 6. objective function 1}) increases as $N$ with $\mu_J(N)$ in place of $\mu_J$. Now we describe this process. Choose two nonempty subsets
$J_1$ and $J_2$ satisfy conditions \textrm{$\mathbf{(a)~~and~~(c)}$} in $\mathbf{(\ref{section 5. restrictions on subsets})}$. It should be pointed out that we do not impose the condition $\mathbf{(b)}$ on $J_1$ and $J_2$. Then put
 \begin{equation}\label{section 6. Process Type III}
 \begin{cases}
 \mu_{J_t}(N)= \mu_{J_t}(N-1)- \mu_{J_1}(N-1)\wedge  \mu_{J_2}(N-1),
 \quad t=1,2\\
 \mu_{J_1\cap J_2}(N)=\mu_{J_1\cap J_2}(N-1)+ \mu_{J_1}(N-1)\wedge
 \mu_{J_2}(N-1)
 \quad \textrm{ if }\quad J_1\cap J_2 \neq \emptyset\\
\mu_{J_1\cup J_2}(N)=\mu_{J_1\cup J_2}(N-1)+ \mu_{J_1}(N-1)\wedge
 \mu_{J_2}(N-1)
 \end{cases}
 \end{equation}
and $\mu_J(N)=\mu_J(N-1)$ for other nonempty subsets $J$ of
$\{1,2,\cdots,k+1\}$. Though this process is not unique, it does not change the value of $\pi_i$,
i.e., $\pi_i(N)=\pi_i(0)$, and equations in (\ref{section 6. I cond by
parameters}) are still true with $\{\mu_J(0)\}$ replaced by
$\{\mu_J(N)\}$. Let $\mu_{\{1,\cdots,k+1\}}^\ast(N)$ be the supremum of
$\mu_{\{1,\cdots,k+1\}}(N)$ among all possible processes consisting of $N$ continuous steps. Let $\{\mu_J^\ast(N)\}$ be obtained by one of those $N$ continuous steps. Then $\mu_{\{1,\cdots,k+1\}}^\ast(N)$ increases as $N$ and
$\{\mu_J^\ast(N)\}$ has a uniform upper bound for each nonempty
subset $J$. By passing to a subsequence $\{N_t\}$, we obtain Cauchy
sequences $\{\mu_J^\ast(N_t)\}$. Let
\begin{equation}
\mu_{J}^\ast(\infty)
=\lim\limits_{t\rightarrow \infty}\mu_{J}^\ast(N_t),
\quad J\subseteq \{1,2,\cdots,k+1\}.
\end{equation}
Now we can show that $\mu_{\{1,\cdots,k+1\}}^\ast(\infty)$ is stable, i.e., any process described in (\ref{section 6. Process Type III}) does not change the value of $\mu_J$ for the new initial data $\widetilde{\mu}_J(0)=\mu_{J}^\ast(\infty)$. Indeed, the union of all proper subsets $J$ satisfying $\mu_J^\ast(\infty)>0$ is a proper subset of $\{1,\cdots,k+1\}$. Thus we get
\begin{equation}
\mu_{\{1,\cdots,k+1\}}^\ast(\infty)=1+\min\limits_{1\leq i\leq k+1}\pi_i.
\end{equation}
Notice that the objective function in (\ref{section 6. objective
function 1}) increases as $N$ with $\mu_J$ replaced by $\mu_J(N)$.
For this reason, we see that the inequality (\ref{section 6.
objective function 1}) is true when there exists a positive $\pi_i$.
Indeed, if some $\pi_i$ is positive and $\mu_{
\{1,\cdots,k+1\}}^\ast(\infty)=1$, then we obtain a $\mu_J^\ast(\infty)>0$
for some nonempty $J\subsetneqq \{1,\cdots,k+1\}$.

Assume now all $\pi_i$ are zero. In this case, the argument is
somewhat different and (\ref{section 6.
Condition for undecomposable of 1,k+1}) will be used. By our choice
of the parameters in (\ref{section 6. I cond by parameters}), it
follows that there is a positive $\mu_J=\mu_J(0)$ with $J$ being a
nonempty proper subset of $\{1,\cdots,k+1\}$. Thus
$\mu_{\{1,\cdots,k+1\}}(0)<1$. For all $N$, it is true that
\begin{equation*}
\sum\limits_{J\ni i}\mu_J^\ast(N)=1, \quad 1\leq i \leq k+1.
\end{equation*}
Recall that $\mu_{\{1,\cdots,k+1\}}^\ast(N)$ tends to $1$. This implies that each sequence $\mu_J^\ast(N)$ tends to zero for all proper subsets $J$. Choose a sufficiently large $N_0$ such that
\begin{equation*}
\mu_{\{1,\cdots,k+1\}}^\ast(N_0)>1-\varepsilon
>\mu_{\{1,\cdots,k+1\}}(0)
\end{equation*}
for small enough $\varepsilon>0$. This observation shows that in
the process, consisting of $N_0$ continuous steps, by which we obtain $\{\mu_J^\ast(N_0)\}$, there is a $M$-th step such that
$\mu_{\{1,\cdots,k+1\}}(M)$ is larger than
$\mu_{\{1,\cdots,k+1\}}(M-1)$. Combining the additional assumption
(\ref{section 6. Condition for undecomposable of 1,k+1}) together
with the observation (\ref{section5. for which i,j a_ij=0}), we
obtain
\begin{equation*}
\sum\limits_J\mu_J(M)\l(|J|-\sum\limits_{J}\frac{\alpha_{ij}}{n}\r)
<
\sum\limits_J\mu_J(M-1)\l(|J|-\sum\limits_{J}\frac{\alpha_{ij}}{n}\r).
\end{equation*}
The argument may vary depending on whether
$\mu_{\{1,\cdots,k+1\}}(N_0)=1$. If $\mu_{\{1,\cdots,k+1\}}(N_0)$ is equal
to $1$, we will obtain, using the notation
$A_k=k+1-\sum_{\{1,\cdots,k+1\}}\alpha_{ij}/n$,
\begin{eqnarray*}
A_k-\sum\limits_J\mu_J(0)\l(|J|-\sum\limits_{J}\frac{\alpha_{ij}}{n}\r)
&\leq&
A_k-\sum\limits_J\mu_J(M-1)\l(|J|-\sum\limits_{J}\frac{\alpha_{ij}}{n}\r)\\
&<&
A_k-\sum\limits_J\mu_J(M)\l(|J|-\sum\limits_{J}\frac{\alpha_{ij}}{n}\r).
\end{eqnarray*}
Notice that we also have
\begin{equation*}
A_k-\sum_J\mu_J(M)\Big(|J|-\sum\limits_{J}\frac{\alpha_{ij}}{n}
\Big)
\leq
A_k-\sum\limits_J\mu_J(N_0)\Big(|J|-\sum\limits_{J}\frac{\alpha_{ij}}{n}
\Big),
\end{equation*}
where the process used to obtain $\mu_J(N_0)$ is one of those $N$
continuous steps such that
$\mu_{\{1,\cdots,k+1\}}(N_0)$ equals $\mu_{\{1,\cdots,k+1\}}^\ast(N_0)$ and
$\mu_{\{1,\cdots,k+1\}}(M)>\mu_{\{1,\cdots,k+1\}}(M-1)$ for some $1\leq M\leq N_0$. The desired
inequality (\ref{section 6. objective function 1}) follows from the
assumption $\mu_{\{1,\cdots,k+1\}}(N_0)=1$.

If $\mu_{\{1,\cdots,k+1\}}(N_0)<1$, we only have
\begin{eqnarray*}
A_k-\sum\limits_J\mu_J(0)\l(|J|-\sum\limits_{J}\frac{\alpha_{ij}}{n}\r)
&<&
A_k-\sum\limits_J\mu_J(N_0)\l(|J|-\sum\limits_{J}\frac{\alpha_{ij}}{n}\r).
\end{eqnarray*}
However, we may regard $\{\mu_J^\ast(N_0)\}$ as new initial data
satisfying (\ref{section 6. I cond by
parameters}) with all $\pi_i=0$ and $\theta_1-\theta_2=1$ since
there exist at least two positive $\mu_J^\ast(N_0)$ for two nonempty proper subsets $J$. This implies
\begin{equation*}
A_k-\sum\limits_J\mu_J(N_0)\l(|J|-\sum\limits_{J}\frac{\alpha_{ij}}{n}
\r)
=
A_k-\sum\limits_J\mu_J^\ast(N_0)
\l(|J|-\sum\limits_{J}\frac{\alpha_{ij}}{n}\r)
\leq 0.
\end{equation*}
Combining above results, we have completed the proof of
(\ref{section 6. objective function 1}). Thus the system $(VI.2)$ has a
solution.

It remains to show that there are infinitely many solutions to the
system $(VI.2)$. Let $\mathcal{H}$ be the hyperplane in $\bR^{k+1}$
given by
\begin{equation}\label{section6. THe hyperplane P}
\sum\limits_{i=1}^{k+1}x_i=k+1
-\sum\limits_{1\leq i<j \leq k+1}\frac{\alpha_{ij}}{n}
\end{equation}
for $x=(x_1,\cdots,x_{k+1})$ in $\bR^{k+1}$. We equip $\mathcal{H}$ the subset topology of $\bR^{k+1}$. Then the set of solutions of the
system $(VI.2)$ forms an open convex subset of $\mathcal{H}$. Let $\{\delta_i^{(1)}\}$ and $\{\delta_i^{(2)}\}$ be two solutions. Then it is clear that their convex combinations
$\{\lambda\delta_i^{(1)}+(1-\lambda)\delta_i^{(2)}\}$ are also
solutions for all $0<\lambda<1$. Assume $\{\delta_i\}$ is a
solution to the system $(VI.2)$. For sufficiently small
$\varepsilon>0$, all points in the $\varepsilon$-neighborhood of
$\{\delta_i\}$ in $\mathcal{H}$ are also solutions. Indeed, if
$\rho=(\rho_1,\cdots,\rho_{k+1})\in \mathcal{H}$ and
$(\sum_{i=1}^{k+1}|\rho_i-\delta_i|^2)^{1/2}<\varepsilon,$ then $(ii)$ and $(iii)$ in $(VI.2)$ are also true for $\{\rho_i\}$ with
sufficiently small $\varepsilon>0$. Thus we have established our
claim. The proof is therefore concluded.
\end{proof}

As a corollary, we can apply Theorem \ref{main theorem 1} to obtain Theorem \ref{section5. nec-suff Selberg integral}.
\begin{coro}
Assume $\alpha_{ij}$ are nonnegative numbers for $1\leq i<j \leq k+1$. Then for all $f_i\in C_0^{\infty}$
$$
\int_{\mathbb{R}^{n(k+1)}}
\prod_{i=1}^{k+1}|f_{i}(x_{i})|
\prod_{1 \leq i<j\leq k+1}|x_i-x_j|^{-\alpha_{ij}}
dx_1dx_2\cdots dx_{k+1}<\infty
$$
if and only if $\{\alpha_{ij}\}$ satisfies the condition (ii) in Theorem \ref{main theorem 1}.
\end{coro}

Now we turn to prove that $T$ has a bounded extension from
$L^{p_1}\times \cdots \times L^{p_k}$ into $BMO$ under the
assumptions in Theorem \ref{section1. main theorem 2} with
$p_{k+1}=1$. It is worth noting that we may replace $BMO$ by
$L^\infty$ in Theorem \ref{section1. main theorem 2} if in addition $p_1,p_2,\cdots,p_k$ satisfy $\sum_{i=1}^k1/{p_i}\geq 1.$
But the boundedness on $BMO$ does not require this assumption. Now we prove the $L^{\infty}$ estimate for $T$ under the additional assumption $\sum_{i=1}^k1/p_i\geq 1$. It suffices to show that there exists a constant $C$ such that
\begin{equation}\label{section 6 Linfty estimate}
\int_{\mathbb{R}^{nk}}
\prod\limits_{i=1}^{k}\Big(f_{i}(x_{i})|x_i|^{-\alpha_{i,k+1}}\Big)
\prod\limits_{1 \leq i<j\leq k}|x_i-x_j|^{-\alpha_{ij}}
dx_1dx_2\cdots dx_{k}\leq C\prod_{i=1}^k\|f_i\|_{p_i}.
\end{equation}
For any nonempty proper subset $J\subseteq\{1,\cdots,k\}$, we put $I=J\cup\{k+1\}$ and then $I$ is a proper subset of $\{1,\cdots,k+1\}$. By the assumptions in Theorem \ref{section1. main theorem 2}, we have $\sum_I1/p_i+\sum_I\alpha_{ij}/n<|I|$. It follows immediately from $p_{k+1}=1$ that
$$\sum_{J}\left(\frac1{p_i}+\frac{\alpha_{i,k+1}}{n}\right)
+\sum_{J}\frac{\alpha_{ij}}{n}<|J|.$$
By the interpolation technique in $\S$4, it is easy to see that the inequality (\ref{section 6 Linfty estimate}) holds.

For $f_i\in C_0^\infty$, both Theorem \ref{section5. nec-suff Selberg integral} and Theorem \ref{section 6. exi of infty pairs p} imply that $T(f_1,\cdots,f_k)$ is locally integrable. For each cube $Q$ with sides parallel to the axes, we use
$\leftidx{^\ast}{Q}$ to denote the cube which is concentric to $Q$
but has the side length twice as long as that of $Q$. We first
decompose $T$, corresponding to $Q$, as a major term
\begin{equation*}
T_S(f_1,\cdots,f_k)(x_{k+1})=\int_{\boldsymbol{(\leftidx{^\ast}{Q}^c)}^k}
\prod_{S}f_i(x_i)\prod_{\{1,\cdots,k+1\}}|x_i-x_j|^{-\alpha_{ij}}dV_{S}
\end{equation*}
and $k$ terms of the form
\begin{equation*}
T_i(f_1,\cdots,f_k)=\int_{\boldsymbol{\leftidx{^\ast}{Q}}}
\l(\int_{(\mathbb{R}^n)^{k-1}}\prod_{S}f_i(x_i)
\prod_{\{1,\cdots,k+1\}}|x_i-x_j|^{-\alpha_{ij}}dV_{S-\{i\}}\r)dx_i
\end{equation*}
for $i\in \{1,\cdots,k\}$. Here $dV_J$ is the product Lebesgue measure
$\prod_{j\in J}dx_j$. As in $\S$4, for $1\leq i \leq k$ we claim
that
$$\int_Q|T_i(f_1,\cdots,f_k)(x_{k+1})|dx_{k+1}\leq C|Q|\prod_S\|f_i\|_{p_i}.$$
Take $i=1$ for example. Put $
\overline{p_1}=\l(\frac{\epsilon}{1+\epsilon}+\frac1{p_1}\r)^{-1},$
$\overline{p_{k+1}}=1+\epsilon,$ and $\overline{p_i}=p_i$ for $2\leq i \leq k$. Then $\{\overline{p_i}\}$ and $\{\alpha_{ij}\}$ satisfy conditions $(i)$, $(ii)$ and $(a)$ of $(iii)$ in Theorem \ref{main theorem 1} with small $\epsilon>0$. Thus
\begin{eqnarray*}
\int_Q|T_1(f_1,\cdots,f_k)|dx_{k+1} &\leq&
\Lambda(|f_1|\chi_{\boldsymbol{\leftidx{^\ast}{Q}}},|f_2|,\cdots,
|f_k|,\chi_{Q})\\
&\leq& C|Q|\prod_S\|f_i\|_{p_i}.
\end{eqnarray*}
Similarly, we can show that the same estimate still holds for
each $T_i(f_1,\cdots,f_k)$ with $2\leq i\leq k$. This also implies the local integrability of $T(f_1,\cdots,f_k)$.

Now it remains to show the average of
$|T_S(f_1,\cdots,f_k)-T_S(f_1,\cdots,f_k)_Q|$ over $Q$ is bounded by a constant multiple of $\prod_S\|f_i\|_{p_i}$. Let $\Theta$ consist of those $i\in \{1,\cdots,k\}$ such that $\alpha_{i,k+1}>0$. By assumptions in Theorem \ref{section1. main theorem 2}, $\Theta$ is nonempty. It is easy to see that
\begin{eqnarray}\label{section6. BMO EST k-linear over Q^c}
&&\frac{1}{|Q|}\int_Q
\Big|T_S(f,\cdots,f_k)(x_{k+1})-T_S(f_1,\cdots,f_k)_Q\Big|
dx_{k+1}\nonumber\\
&\leq& C |Q|^{1/n}
\sum_{t\in \Theta}\int_{\boldsymbol{(\leftidx{^\ast}{Q}^c)^k}}
\prod\limits_{i=1}^k |\widetilde{f_i^{(t)}}(x_i)|
\prod\limits_{1\leq i<j \leq k}|x_i-x_j|^{-\alpha_{ij}}dV_S,
\end{eqnarray}
with
$$\widetilde{f_i^{(t)}}(x_i)
=f_i(x_i)|x_i-c_Q|^{-\alpha_{i,k+1}-\boldsymbol{\delta}_i^t}, \quad
1\leq i \leq k.$$
The treatment of each term in the above summation
is similar. To obtain the desired estimate, we shall prove that each term is not greater than a constant multiple of
$\prod_{i=1}^k\|f_i\|_{p_i}$. By Theorem \ref{main theorem 1}, we
need to solve the following system of linear inequalities:
\begin{equation}
(VI.6)
\begin{cases}
(i)\quad\quad \delta_{i}\geq \frac{1}{p_i}, \quad 1\le i \le k;
\nonumber\\
(ii) \quad\;\;\, \delta_{i}<
1/p_i+(\alpha_{i,k+1}+\boldsymbol{\delta}_i^t)/n,
\quad \textrm{if $\alpha_{i,k+1}>0$, i.e., $i\in \Theta$};\nonumber\\
 \quad \quad\quad\, \delta_{i}\leq
1/{p_i}+\alpha_{i,k+1}/n,
\qquad\quad\;\;\; \textrm{if $\alpha_{i,k+1}=0$, i.e., $i\notin \Theta$};\nonumber\\
(iii)\quad \sum\limits_{i=1}^k\delta_{i}
=k-\sum\limits_{S}\alpha_{ij}/n; \nonumber \\
(iv) \quad\;
\sum\limits_{J}\delta_{i} < |J|-\sum\limits_J \alpha_{ij}/n
\quad \qquad\quad\textrm{for nonempty proper $J\subseteq \{1,\cdots,k\}$}. \nonumber
\end{cases}
\end{equation}
For $i\notin\Theta$, we indeed have $\delta_i=\frac1{p_i}$ and include it here for convenience of notations. For $i\in \Theta$, we see that $\widetilde{f_i^{(t)}}\in L^{\delta_i^{-1}}(\leftidx{^\ast}{Q}^c)$. To prove the existence of solutions to $(VI.6)$, we first present a useful observation:
\begin{equation}\label{section 6. Con for undecom for S}
\sum\limits_{i\in J}\sum\limits_{j\in J^c}\alpha_{ij}>0
\end{equation}
for all nonempty proper subsets $J$ of $\{1,2,\cdots,k\}$. Here $J^c$ is the
complement of $J$ relative to $\{1,2,\cdots,k\}$. If there were some nonempty proper
$J_1\subseteq \{1,\cdots,k\}$ such that all $\alpha_{ij}=0$ for $i\in J_1$ and
$j\in J_1^c$, we would obtain a contradiction. By assumptions in
Theorem \ref{section1. main theorem 2}, we have
\begin{eqnarray*}
&&\sum\limits_{J}\frac 1{p_i}
+\sum\limits_{J\cup \{k+1\} } \frac{\alpha_{ij}}{n}
+\frac 1{p_{k+1}}<|J|+1
\end{eqnarray*}
with $J=J_1$ and $J=J_1^c$. Recall $p_{k+1}=1$. The above two
inequalities contradict the homogeneity condition $(i)$ in Theorem
\ref{section1. main theorem 2}. Thus
(\ref{section 6. Con for undecom for S}) is true.

Suppose that $u_i$, $v_i$, $\theta_1$, $\theta_2$ and $\mu_J$ are
nonnegative numbers satisfying
\begin{equation}\label{section6. (2) cond satisfied by parameters}
u_i-v_i+(\theta_1-\theta_2)-\sum\limits_{J\ni i}\mu_J=0,\quad
1\leq i\leq k,
\end{equation}
where either $v_i>0$ for some $i\in\Theta$ or $\mu_J>0$ for some nonempty proper subset $J$ of $\{1,\cdots,k\}$. Our task is to show
\begin{eqnarray}\label{section 7. IIobjective function}
&&\sum\limits_{S}\frac{u_i}{p_i}
-\sum\limits_{S}v_i\l(\frac{1}{p_i}
+\frac{\alpha_{i,k+1}+\boldsymbol{\delta}_i^t}{n}\r)
+(\theta_1-\theta_2)\l(k-\sum\limits_{S}\frac{\alpha_{ij}}{n}\r)
-\sum\limits_{J\subset S}
\mu_J\l(|J| -\sum\limits_{J}\frac{\alpha_{ij}}{n}\r)\nonumber \\
&=& -\sum\limits_{J\subset S}\mu_J\l(|J|
-\sum\limits_{J}\frac{1}{p_i}
-\sum\limits_{J}\frac{\alpha_{ij}}{n}\r) +(\theta_1-\theta_2){B}_S
-\sum\limits_{S}v_i\frac{\alpha_{i,k+1}}{n}-\frac{v_t}{n}<0,
\end{eqnarray}
where $$B_J=|J|-\sum_J1/{p_i}-\sum_{J}{\alpha_{ij}}/{n}$$
for nonempty subsets $J$ of $\{1,\cdots,k\}$. For convenience, we use $\boldsymbol{H}$ to
denote the objective function in the above inequality. Similarly, we may assume $\theta_1-\theta_2=1$.

Now we also need a process described as in (\ref{section 6. Process
Type III}). For nonempty subsets $J_1$ and $J_2$ of $\{1,2,\cdots,k\}$ satisfying conditions $\mathbf{(a)}$ and $\mathbf{(c)}$ in $\mathbf{(\ref{section 5. restrictions on subsets})}$, put
\begin{equation}\label{section 6. Process Type IV}
\begin{cases}
\mu_{J_i}(N)= \mu_{J_i}(N-1)- \mu_{J_1}(N-1)\wedge  \mu_{J_2}(N-1),
\quad i=1,2.\\
\mu_{J_1\cap J_2}(N)=\mu_{J_1\cap J_2}(N-1)+ \mu_{J_1}(N-1)\wedge
\mu_{J_2}(N-1)
\quad \textrm{if} \quad J_1\cap J_2 \neq \emptyset\\
\mu_{J_1\cup J_2}(N)=\mu_{J_1\cup J_2}(N-1)+ \mu_{J_1}(N-1)\wedge
\mu_{J_2}(N-1)
\end{cases}
\end{equation}
and $\mu_J(N)=\mu_J(N-1)$ for other nonempty subsets $J$ of $\{1,\cdots,k\}$.
Here we also do not require $J_1\cap J_2\neq \emptyset$ in the
recursion. For any initial data $\{\mu_J(0)\}$, let
$\{\mu_J^{\ast}(N)\}$ be obtained by one of those processes
consisting of $N$ steps such that $\mu_S(N)$ attains its maximum
$\mu_S^{\ast}(N)$. Since the process does not change the values of
$u_i$ and $v_i$, by passing to a subsequence $\{N_t\}$, we can denote the limit of $\mu_J^{\ast}(N_t)$ by
$\mu_J^{\ast}(\infty)$. We can prove that $\mu_S^{\ast}(\infty)$ is stable with respect to the process (\ref{section 6. Process Type IV}). However, $\mu_J^{\ast}(\infty)$ may not be stable generally for $J\subsetneqq \{1,\cdots,k\}$. Similarly, if $\{\mu_J^{\ast}(\infty)\} $ is regarded as new initial data, then we can apply the process (\ref{section 6. Process Type IV}) again to $\mu_J^{\ast}(\infty)$ for all proper subsets $J$. This procedure will continue if not all $\mu_J^{\ast}(\infty)$ are stable. For this reason, we can assume that all $\mu_J^{\ast}(\infty)$ are stable. By this process, we also have
\begin{flalign}\label{section7. object function is increasing}
&-\sum\limits_{J\subset S}\mu_J(0)\l(|J|
-\sum\limits_{J}\frac{1}{p_i}
-\sum\limits_{J}\frac{\alpha_{ij}}{n}\r)
\leq
-\sum\limits_{J\subset S}\mu_J^\ast(\infty)\l(|J|
-\sum\limits_{J}\frac{1}{p_i}
-\sum\limits_{J}\frac{\alpha_{ij}}{n}\r)&
\end{flalign}
which becomes a strict inequality if $\mu_S^\ast(\infty)>\mu_S(0)$
by the property (\ref{section 6. Con for undecom for S}). Let
$w_i=u_i-v_i$ for $1\le i \le k$. Let $\sigma(1),\sigma(2),\cdots,\sigma(k)$ be a permutation of $1,2,\cdots,k$ such that
$$w_{\sigma(1)}\geq w_{\sigma(2)}\geq \cdots \geq w_{\sigma(k)}.$$
Since all $\mu_J^{\ast}(\infty)$ are stable, it follows from (\ref{section6. (2) cond satisfied by parameters}) that we have
\begin{eqnarray*}
\mu_{J_{k}}^{\ast}(\infty)&=&\mu_{\{1,\cdots,k\}}^{\ast}(\infty)
~=~1+w_{\sigma(k)},\\
\mu_{J_{k-i}}^{\ast}(\infty)&=&w_{\sigma(k-i)}-w_{\sigma(k-i+1)},~~~
1\leq i\leq k-1,
\end{eqnarray*}
where
$$J_k =\{\sigma(1),\sigma(2),\cdots,\sigma(k)\}$$
and
$$J_{k-i}= \{\sigma(1),\sigma(2),\cdots,\sigma(k-i)\}$$
for $1\leq i\leq k-1$.
Since $p_{k+1}=1$, it follows from $(i)$ in Theorem \ref{section1. main theorem 2} that $B_S=\sum_{S}\alpha_{i,k+1}/n$. Then the objective function $\boldsymbol{H}$ is equal to
\begin{eqnarray*}
\boldsymbol{H}
&=&-\sum\limits_{i=1}^{k-1}(w_{\sigma(i)}-w_{\sigma(i+1)})
{B}_{J_i}-(1+w_{\sigma(k)})B_{S}+B_S
-\sum_Sv_i\frac{\alpha_{i,k+1}}{n}
-\frac{v_t}{n}\\
&=&-\sum\limits_{i=1}^{k-1}(w_{\sigma(i)}-w_{\sigma(i+1)})
{B}_{J_i}-(1+w_{\sigma(k)})B_{S}
+\sum_S\Big((u_i-v_i)+1\Big)\frac{\alpha_{i,k+1}}{n}\\
& &-\sum_Su_i\frac{\alpha_{i,k+1}}{n}
-\frac{v_t}{n}\\
&=&-\sum\limits_{i=1}^{k-1}(w_{\sigma(i)}-w_{\sigma(i+1)})
{B}_{J_i}-(1+w_{\sigma(k)})B_{S}
+\sum_S\l(\sum_{J\ni i}\mu_J^{\ast}(\infty)\r)\frac{\alpha_{i,k+1}}{n}\\
& &-\sum_Su_i\frac{\alpha_{i,k+1}}{n}
-\frac{v_t}{n}\\
&=&-\sum\limits_{i=1}^{k-1}(w_{\sigma(i)}-w_{\sigma(i+1)}){B}_{J_i}
-(1+w_{\sigma(k)})B_{S}
+\sum_{J\subset S}\mu_J^{\ast}(\infty)\l(\sum_{J\ni i}\frac{\alpha_{i,k+1}}{n}\r)\\
& &-\sum_Su_i\frac{\alpha_{i,k+1}}{n}
-\frac{v_t}{n}.
\end{eqnarray*}
Hence we have
\begin{eqnarray*}
\boldsymbol{H}
&=&-\sum\limits_{i=1}^{k-1}(w_{\sigma(i)}-w_{\sigma(i+1)}){B}_{J_i}
+\sum_{i=1}^{k-1}(w_{\sigma(i)}-w_{\sigma(i+1)})\l(\sum_{J_i\ni j}
\frac{\alpha_{j,k+1}}{n}\r)\\
& &-\sum_Su_i\frac{\alpha_{i,k+1}}{n}
-\frac{v_t}{n}\\
&=&-\sum\limits_{i=1}^{k-1}(w_{\sigma(i)}-w_{\sigma(i+1)})
{B}_{J_i\cup\{k+1\}}
-\sum\limits_S u_i\frac{\alpha_{i,k+1}}{n}-\frac{v_t}{n}.
\end{eqnarray*}
Since $w_{\sigma(i)}-w_{\sigma(i+1)}\geq 0$ and ${B}_{J_i\cup\{k+1\}}>0$ for $1\leq
i \leq k-1$, the desired inequality $\boldsymbol{H}<0$ follows if
$w_{\sigma(i)}-w_{\sigma(i+1)}> 0$ for some $1\leq i\leq k-1$. Assume now
$w_i=w_j$ for all $i,j$ in $\{1,2,\cdots,k\}$. Then we have
$$\boldsymbol{H} =
-\sum_{\{1,\cdots,k\}}u_i\frac{\alpha_{i,k+1}}{n}-\frac{v_t}{n}.$$
If $w_1=u_1-v_1>0$, then all $u_i$ are positive and hence
$\boldsymbol{H}<0$. If $w_1=u_1-v_1=0$, then $u_i=v_i$ for
$i\in \{1,\cdots,k\}$. If there is a positive $u_i$ for those $i$ such that
$\alpha_{i,k+1}>0$, we also get $\boldsymbol{H}<0$. Otherwise all $u_i$ and $v_i$ are zero. By the choice of $u_i$, $v_i$ and $\mu_J$, we see that there exists a
$\mu_J(0)>0$ for a nonempty $J\subsetneqq \{1,\cdots,k\}$. However, (\ref{section7. object function is increasing}) becomes a strict inequality in this case. Therefore we also obtain $\boldsymbol{H}<0$.

For each $t\in\Theta$, let $\{\delta_i(t):1\leq i\leq k\}$ be a
solution of the system $(VI.6)$. Then by Theorem \ref{main theorem
1} and H\"{o}lder's inequality we obtain
\begin{eqnarray*}
&&\frac{1}{|Q|}
\int_Q\Big|T_S(f,\cdots,f_k)(x_{k+1})-T_S(f_1,\cdots,f_k)_Q\Big| dx_{k+1}\\
&\leq&
C|Q|^{1/n}\sum_{t\in\Theta}\prod_{S}\|\widetilde{f_i^{(t)}}
\chi_{\boldsymbol{\leftidx{^\ast}{Q}}^c}\|_{\frac{1}{\delta_i(t)}}\\
&\leq& C\prod_{S}\|f_i\|_{p_i}.
\end{eqnarray*}
This proves that $T$ has a bounded extension from $L^{p_1}\times
\cdots \times L^{p_k}$ into $BMO$.

\section{Appendix}
In this section, we shall prove the claim in Remark \ref{section5 Remark 5.1} in $\S$5. Indeed, there is an analogue of Theorem \ref{Estimtates of Generalization of Riesz Potential} on the sphere $S^n=\{x\in\bR^{n+1}:|x|=1\}$. For $\xi,\eta\in S^n$, we shall use $|\xi-\eta|$ to denote the standard Euclidean metric between $\xi$ and $\eta$ in $\bR^{n+1}$.
\begin{theorem}\label{section7. Estimtates of Generalization of Riesz Potential on Sphere}
Assume that $0<\alpha_i<n$ for $1\leq i \leq k$ satisfy $\sum_{i=1}^k\alpha_i>n$. We have the following estimate:
\begin{equation*}
\int_{S^n}\prod\limits_{i=1}^k
\big|\xi-\xi_i\big|^{-\alpha_i}d\sigma(\xi)
\leq C
\sum\limits_{u=1}^k
\boldsymbol{L}_u(\xi_1,\xi_2,\cdots,\xi_k),~~~\xi_i\in S^n~~{\rm for}~~1\leq i\leq k,
\end{equation*}
where $d\sigma$ is the Lebesgue measure on the $S^n$. Here $\boldsymbol{L}_u$ is defined by
\begin{equation*}
\boldsymbol{L}_u(\xi_1,\xi_2,\cdots,\xi_k)=
\begin{cases}
d_S^{n-\sum_{S}\alpha_i}
\l(\mathlarger{\mathlarger{\chi}}_{\{\sum_{S-\{u\}}\alpha_i<n\}}
+\chi_{\{\sum
_{S-\{u\}}\alpha_i=n\}}\log \frac{2d_S}{d_{S-\{u\}}}\r)\\
d_S^{-\alpha_u}
\int_{S^n}\prod_{S-\{u\}}|\xi-\xi_i|^{-\alpha_i}d\sigma(\xi) ,\quad {\rm if}
~~\sum_{S-\{u\}}\alpha_i>n,
\end{cases}
\end{equation*}
where the above characteristic functions $\chi$ are functions of
$\alpha_1,\cdots,\alpha_k$. Also $S=\{1,2,\cdots,k\}$ and
$d_I=\sum_I|\xi_i-\xi_j|$ for subsets $I$ of $S$ with
$|I|\geq 2$.
\end{theorem}
The proof of the above theorem is the same as that of Theorem
\ref{Estimtates of Generalization of Riesz Potential}. First, we state a variant of Lemma \ref{Sec2. estimate of FI on the critical order} on the sphere $S^n$.
\begin{lemma}\label{section7. estimate of FI on the critical order}
If $\alpha_i>0$, $1\leq i \leq k$, satisfy $\sum_{i}\alpha_i=n$ with $k\geq 2$, then there holds
\begin{equation}
\int_{S^n}\prod\limits_{i=1}^k |\xi-\xi_i|^{-\alpha_i}d\sigma(\xi) \leq
C\log\frac{C}{d_S},~~~\xi_i\in S^n~~{\rm for}~~1\leq i\leq k,
\end{equation}
where $C$ depends on $\alpha_1,\cdots,\alpha_k$ and the dimension $n$ but not on choices of $\xi_i\in S^n$. Moreover, the reverse of the above inequality is also true.
\end{lemma}
We now turn to the proof of Lemma \ref{section7. estimate of FI on the critical order}.\\
\begin{proof}
If there exist two points $\xi_i$ and $\xi_j$ such that $|\xi_i-\xi_j|\geq 1$, then the integral in the lemma is finite with a bound depending only on $\alpha_i$ and the dimension $n$. Now we assume that $|\xi_i-\xi_j|\leq 1$ for all $1\leq i<j\leq k$. By rotation, we assume $\xi_1=(0,\cdots,0,-1)$. It suffices to show that
\begin{equation}\label{section 7. ineq 1 in proof of Lem1}
\int_{\xi\in S^n,~|\xi-\xi_1|\leq \sqrt{2}}\prod\limits_{i=1}^k |\xi-\xi_i|^{-\alpha_i}d\sigma(\xi) \leq
C\log\frac{C}{d_S}.
\end{equation}
Let $\pi$ be the stereographic projection from $\bR^n$ onto the sphere $S^n$ minus the north pole. Then
\begin{equation*}
\pi(x_1,x_2,\cdots,x_n)=\l(\frac{2x_1}{1+|x|^2},
\frac{2x_2}{1+|x|^2},\cdots,
\frac{2x_n}{1+|x|^2},\frac{|x|^2-1}{1+|x|^2}\r),
~~~x=(x_1,x_2,\cdots,x_n)\in\bR^n.
\end{equation*}
The inverse of $\pi$ is given by
\begin{equation*}
\pi^{-1}(\xi_1,\xi_2,\cdots,\xi_{n+1})=\l(\frac{\xi_1}{1-\xi_{n+1}},
\frac{\xi_2}{1-\xi_{n+1}},\cdots, \frac{\xi_n}{1-\xi_{n+1}}\r)
\end{equation*}
for $(\xi_1,\xi_2,\cdots,\xi_{n+1})\in S^n-\{e_{n+1}\}$, where $e_{n+1}$ is the north pole of $S^n$. By direct computations, we have the following properties:
\begin{eqnarray*}
|\pi(x)-\pi(y)|&=&
\frac{2|x-y|}{\sqrt{1+|x|^2}\sqrt{1+|y|^2}},
~~~x,y\in\bR^n,\\
|\pi^{-1}(\xi)-\pi^{-1}(\eta)|
&=&\frac{|\xi-\eta|}{\sqrt{1-\xi_{n+1}}\sqrt{1-\eta_{n+1}}}
=\frac{2|\xi-\eta|}{|e_{n+1}-\xi||e_{n+1}-\eta|},
~~~\xi,\eta\in S^n-\{e_{n+1}\}.
\end{eqnarray*}
Define $\pi(x)=\xi$ and $\pi(x_i)=\xi_i$ for $1\leq i \leq k$. Then
the integral in (\ref{section 7. ineq 1 in proof of Lem1}) is equal to
\begin{eqnarray*}
\int_{B_1(0)}\prod_{i=1}^k\l(\frac{2|x-x_i|}{\sqrt{1+|x|^2}
\sqrt{1+|x_i|^2}}\r)^{-\alpha_i}\l(\frac2{1+|x|^2}\r)^ndx
\approx
\int_{B_1(0)}\prod_{i=1}^k|x-x_i|^{-\alpha_i}dx
\end{eqnarray*}
since all $x_i\in B_1(0)\subset\bR^n$. By previous assumptions, we see that $|\xi_i-\xi_j|$ are comparable to $|x_i-x_j|$. Thus we can apply Lemma \ref{Sec2. estimate of FI on the critical order} to obtain the desired estimate. The proof is complete.
\end{proof}

By Lemma \ref{section7. estimate of FI on the critical order}, we can prove Theorem \ref{section7. Estimtates of Generalization of Riesz Potential on Sphere} now.\\
\begin{proof}
Without loss of generality, we assume $L=\max_{S}|\xi_i-\xi_j|=|\xi_1-\xi_k|>0$. Since the desired estimate is rotation invariant, we may assume that $\xi_1$ is the south pole on $S^n$. Divide the integral in the theorem into two parts as in the proof of Theorem \ref{Estimtates of Generalization of Riesz Potential}. We first consider the integral over $|\xi-\xi_1|\leq L/2$.

$\boldsymbol{Estimate~of~the~integral~over~\textrm{$|\xi-\xi_1|\leq L/2$}}$

It is easy to see that
\begin{equation*}
\int_{\xi\in S^n,~|\xi-\xi_1|\leq L/2}\prod_{i=1}^k |\xi-\xi_i|^{-\alpha_i}d\sigma
\leq
Cd_S^{-\alpha_k}\int_{\xi\in S^n,~|\xi-\xi_1|\leq L/2}\prod_{i=1}^{k-1} |\xi-\xi_i|^{-\alpha_i}d\sigma.
\end{equation*}

Case (i): $\sum_{i=1}^{k-1}\alpha_i<n$.

We claim that the right side of the above inequality is bounded by a constant multiple of $L^{n-\sum_{S}\alpha_i}$. For $L\geq 1$, the claim is obvious since $\prod_{i=1}^{k-1} |\xi-\xi_i|^{-\alpha_i}$ is integrable on $S^n$. Now assume $L<1$. Then we can use the stereographic projection $\pi$ to change variables as in the proof of Lemma \ref{section7. estimate of FI on the critical order}. Our claim follows from the observation that $|\pi^{-1}(\xi)-\pi^{-1}(\eta)|$ is comparable to $|\xi-\eta|$ for all $\xi,\eta\in \{\xi\in S^n:|\xi-\xi_1|\leq L\}$.

Case (ii) $\sum_{i=1}^{k-1}\alpha_i=n$.

We shall prove that the integral of
$\prod_{i=1}^{k-1} |\xi-\xi_i|^{-\alpha_i}$
over $\{\xi:|\xi-\xi_1|<L/2\}$ is bounded by a constant multiple of $\log(2d_S/d_{S-\{k\}})$. Assume $|\xi_i-\xi_1|\leq 2L/3$ for all $2\leq i \leq k-1$. Otherwise if $|\xi_{i_0}-\xi_1|> 2L/3$ for some
$2\leq i_0 \leq k-1$, then it is clear that
\begin{eqnarray*}
\int_{S^n:~|\xi-\xi_1|\leq L/2}\prod_{i=1}^{k-1} |\xi-\xi_i|^{-\alpha_i}d\sigma
&\leq&
CL^{-\alpha_{i_0}}\int_{S^n:~|\xi-\xi_1|\leq L/2}\prod_{i=1,i\neq i_0}^{k-1} |\xi-\xi_i|^{-\alpha_i}d\sigma\\
&\leq&CL^{-\alpha_{i_0}}L^{n-\sum_{i\neq i_0,k}\alpha_i}
\end{eqnarray*}
which is bounded by a constant since $\sum_{i=1}^{k-1}\alpha_i=n$. Therefore we can now assume $|\xi_i-\xi_1|\leq 2L/3$ for $2\leq i \leq k-1$. Note that $2L/3\leq 4/3<\sqrt{2}$. Using the stereographic projection $\pi$ again, we define $\pi(x)=\xi$ and $\pi(x_i)=\xi_i$ for $1\leq i \leq k-1$. Then
\begin{eqnarray*}
\int_{S^n:~|\xi-\xi_1|\leq L/2}\prod_{i=1}^{k-1} |\xi-\xi_i|^{-\alpha_i}d\sigma
&\leq &
C\int_{|x|\leq CL}\prod_{i=1}^{k-1} |x-x_i|^{-\alpha_i}dx\\
&\leq&
C\log\l(\frac{CL}{\sum_{S-\{k\}}|x_i-x_j|}\r).
\end{eqnarray*}
Since $L\approx d_S$ and $|x_i-x_j|\approx |\xi_i-\xi_j|$, the desired estimate follows.

Case (iii): $\sum_{i=1}^{k-1}\alpha_i>n$.

In this case, the desired estimate is obvious.

$\boldsymbol{Estimate~of~the~integral~over~\textrm{$|\xi-\xi_1|> L/2$}}$

Case (i): $\sum_{i=2}^{k}\alpha_i>n$.

As above, the desired estimate is obvious. Indeed, we have
\begin{eqnarray*}
\int_{S^n:~|\xi-\xi_1|> L/2}\prod_{i=1}^{k} |\xi-\xi_i|^{-\alpha_i}d\sigma
&\leq& CL^{-\alpha_1}\int_{S^n}\prod_{i=2}^{k} |\xi-\xi_i|^{-\alpha_i}d\sigma.
\end{eqnarray*}

Case (ii): $\sum_{i=2}^{k}\alpha_i\leq n$ and $L\geq 1$.

It follows from $1\leq L\leq 2$ that
\begin{eqnarray*}
\int_{S^n:~|\xi-\xi_1|> L/2}\prod_{i=1}^{k} |\xi-\xi_i|^{-\alpha_i}d\sigma
&\approx & \int_{S^n:~|\xi-\xi_1|> L/2}\prod_{i=2}^{k} |\xi-\xi_i|^{-\alpha_i}d\sigma.
\end{eqnarray*}
Then the desired estimate is true since the integral of $\prod_{i=2}^{k} |\xi-\xi_i|^{-\alpha_i}$ over $\{S^n:~|\xi-\xi_1|> L/2\}$ is bounded by a constant.
If $\sum_{i=2}^{k}\alpha_i<n$, then the desired estimate is true. In the case $\sum_{i=2}^{k}\alpha_i=n$, our estimate is also true by Lemma \ref{section7. estimate of FI on the critical order}.

Case (iii): $\sum_{i=2}^{k}\alpha_i\leq n$ and $L<1$.

Since $0<\alpha_1<n$, we see that $\sum_{S}\alpha_i<2n$. Then by the stereographic projection $\pi$, we define $\pi(x)=\xi$ and $\pi(x_i)=\xi_i$ for $1\leq i \leq k$. Since $\xi_1$ is the south pole and $|\xi_i-\xi_1|\leq L<1$, there exist two positive constants $c_1$ and $c_2$, independent of $L$, such that $|x_i|\leq c_1L$ and $|x|\geq c_2L$ for $|\xi-\xi_1|>L/2$. Then
\begin{eqnarray*}
\int_{S^n:~|\xi-\xi_1|> L/2}\prod_{i=1}^{k} |\xi-\xi_i|^{-\alpha_i}d\sigma
&\leq&
\int_{|x|\geq c_2L}\prod_{i=1}^k\l(\frac{2|x-x_i|}{\sqrt{1+|x|^2}
\sqrt{1+|x_i|^2}}\r)^{-\alpha_i}\l(\frac2{1+|x|^2}\r)^ndx\\
&\leq & C\int_{|x|\geq c_2L}\prod_{i=1}^{k} |x-x_i|^{-\alpha_i}dx\\
&\leq & CL^{-\alpha_1} \int_{c_2L\leq|x|\leq 2L}\prod_{i=2}^{k} |x-x_i|^{-\alpha_i}dx,
\end{eqnarray*}
where we have used the fact $\sum_{i=1}^k\alpha_i<2n$ in the second inequality. By a similar argument as in our treatment of the integral over $\{S^n:~|\xi-\xi_1|\leq L/2\}$, we can also obtain the desired estimate. The proof of the theorem is complete.
\end{proof}\\

\noindent{\bf Acknowledgements.} The authors would like to express their deep gratitude to the referees for many valuable suggestions on this paper. This work
was supported by the National Natural Science Foundation of China under Grant Nos. 11471309 and 11561062.


\begin{thebibliography}{10000}
\addtolength{\itemsep}{-1.5ex}

\bibitem{barthe}F. Barthe, On a reverse form of the Brascamp-Lieb inequality. \emph{Invent. Math.}, \textbf{134} (1998): 335-361.

\bibitem{beckner} W. Beckner, Geometric inequalities in Fourier analysis. Essays on
Fourier Analysis in honor of Elias M. Stein (Princeton, N. J., 1991), 36-68, Princeton
Math. Ser. 42, eds,\; C. Fefferman, R. Fefferman and S. Wainger, Prinction University, NJ, 1995.

\bibitem{beckner2013} W. Beckner, Multilinear embedding and Hardy's inequality. arXiv:1311.6747, 2013.

\bibitem{beckner2015} W. Beckner, Functionals for multilinear fractional embedding. \emph{Acta Mathematica Sinica}, English Series, \textbf{31} (2015), 1-28.

\bibitem{bennet1}J. Bennett, A. Carbery, M. Christ and T. Tao, The Brascamp-Lieb inequalities: finiteness, structure and extremals. \emph{Geometric and Functional Analysis}, \textbf{17} (2008), 1343-1415.

\bibitem{bennet2}J. Bennett, A. Carbery, M. Christ and T. Tao, Finite bounds for H\"{o}lder-Brascamp-Lieb multilinear inequalities. \emph{Math. Res. Lett.}, \textbf{17} (2010), 647-666.

\bibitem{Brascamp76}H. J. Brascamp and E. H. Lieb, Best constants in Young's inequality, its converse, and its generalization to more than three functions. \emph{Advances in Mathematics}, \textbf{20} (1976), 151-173.

\bibitem{Brascamp}H. J. Brascamp, E. H. Lieb and J. M. Luttinger, A general rearrangement inequality for multiple integrals. \emph{J. Funct. Anal.},
\textbf{17} (1974), 227-237.

\bibitem{christ} M. Christ, On the restriction of the Fourier transform to curves: endpoint
results and the degenerate case. \emph{Trans. Amer. Math. Soc.},
\textbf{287} (1985), 223-238.

\bibitem{fanky} K. Fan, On systems of linear inequalities.
\emph{Annals of Mathematics Studies}, \textbf{38} (1956), 99-156.

\bibitem{grafakos92} L. Grafakos, On multilinear fractional
integrals. \emph{Studia Math}, \textbf{102} (1992), 49-56.

\bibitem{grafakos1} L. Grafakos and C. Murpurgo, A Selberg integral formula and its
applications. \emph{Pac. J. Math}, \textbf{191} (1999), 85-94.

\bibitem{gra-kal} L. Grafakos and N. Kalton, Some remarks on multilinear maps and
interpolation. \emph{Math. Ann.}, \textbf{319} (2001), 151-180.

\bibitem{janson} S. Janson, On interpolation of multi-linear operators, Function spaces and applications (Lund 1986) 290-302. Lecture Notes in Math
\textbf{1302}, Springer, Berlin-New York 1988.

\bibitem{john-nirenberg} F. John and L. Nirenberg, On functions of bounded
mean oscillation. \emph{Comm. Pure Appl. Math.}, \textbf{14} (1961),
415-426.

\bibitem{kenigstein} C. Kenig and E. M. Stein, Multilinear estimates and fractional integration. \emph{Math. Res. Lett.}, \textbf{6} (1999), 1-15.

\bibitem{Lieb}E. H. Lieb, Sharp constants in the Hardy-Littlewood-Sobolev and related inequalities, \emph{Ann. of Math.}, \textbf{118} (1983), 349-374.

\bibitem{Lieb90}E. H. Lieb, Gaussian kernels have only Gaussian maximizers, \emph{Invent. Math.}, \textbf{102} (1990): 179-208.

\bibitem{morpurgo}C. Morpurgo, Sharp inequalities for functional integrals and traces of conformally invariant operators. \emph{Duke Mathematical Journal}, \textbf{114} (2002), 477-553.

\bibitem{perry}A. P. Perry, Global well-posedness and long-time asymptotics for the defocussing Davey-Stewartson II equation in $ H^{1, 1}(\mathbb C) $, with an appendix by M. Christ. \emph{Journal of Spectral Theory}, \textbf{6} (2016): 429-481.

\bibitem{rockafellar} R. T. Rockafellar, Convex Analysis. Princeton Univ. Press, 1970.

\bibitem{stein} E. M. Stein, Singular integrals and differentiability properties of
functions. Princeton university press, 1970.

\bibitem{steinweisspaper} E. M. Stein and G. Weiss, Fractional integrals in $n$-dimensional Euclidean
spaces. \emph{J. Math. Mech.}, \textbf{7} (1958), 503-514.

\bibitem{stein-weissbook} E. M. Stein and G. Weiss, On the theory of harmonic functions of several variables. \emph{Acta Math.}, \textbf{103} (1960), 26-62.

\bibitem{S.i.valdimarsson}S. Valdimarsson, The Brascamp-Lieb polyhedron. \emph{Canadian Journal of Mathematics}, \textbf{62} (2010): 870-888.

\bibitem{valdimarsson}S. I. Valdimarsson, A multilinear generalisation of the Hilbert transform and fractional integration, \emph{Revista Matem\'{a}tica Iberoamericana}, \textbf{28} (2012), 25-55.

\bibitem{wudi} D. Wu, k-linear weighted fractional integrals: the boundedness, best constants and
related questions. Ph.D. Dissertation, 2013, University of Chinese
Academy of Sciences, Beijing.

\end{thebibliography}
\end{document}